\documentclass[11pt]{article}
\usepackage[margin=1in]{geometry}
\usepackage[utf8]{inputenc}
\usepackage[dvipsnames]{xcolor}
\usepackage{fullpage}
\usepackage{enumerate}
\usepackage{amsthm,amsmath, amssymb, amsfonts,multicol}
\usepackage{authblk}
\usepackage{verbatim}
\usepackage{tikz}
\usetikzlibrary{fit,matrix,positioning,shapes}
\usetikzlibrary{decorations.shapes}
\usepackage{graphicx,float}
\usepackage{mathtools}
\usepackage[colorlinks=true, citecolor=ForestGreen, linkcolor=BrickRed, urlcolor=blue]{hyperref}

\allowdisplaybreaks

\newtheorem{theorem}{Theorem}[section]
\newtheorem{lemma}[theorem]{Lemma}
\newtheorem{corollary}[theorem]{Corollary}

\newtheorem{conj}[theorem]{Conjecture}

\newtheorem{problem}[theorem]{Problem}
\newtheorem{definition}[theorem]{Definition}

\newtheorem{Ex}[theorem]{Example}

\newtheorem{prop}[theorem]{Proposition}
\newtheorem*{theorem*}{Theorem}
\newtheorem{remark}[theorem]{Remark}

\newcommand{\bl}[1]{{\color{blue} #1}}

\newcommand{\pp}[1]{{\color{purple} #1}}

\usepackage{xargs}
\usepackage[colorinlistoftodos,prependcaption,textsize=tiny,linecolor=redbackgroundcolor=red!25,bordercolor=red]{todonotes}
\newcommandx{\nick}[2][1=]{\todo[linecolor=ForestGreen,backgroundcolor=ForestGreen!25,bordercolor=ForestGreen,#1]{#2 ---Nick}}
\newcommandx{\laura}[2][1=]{\todo[linecolor=purple,backgroundcolor=purple!25,bordercolor=purple,#1]{#2 ---Laura}}

\newcommand{\precdot}{\prec\mathrel{\mkern-5mu}\mathrel{\cdot}}

\makeatletter
\newcommand{\preceqdot}{\mathrel{\mathpalette\pr@ceqd@t\relax}}
\newcommand{\pr@ceqd@t}[2]{%
  \begingroup
  \sbox\z@{$#1\prec$}\sbox\tw@{$#1\preceq$}%
  \dimen@=\dimexpr\ht\tw@-\ht\z@\relax
  {\preceq}%
  \mkern-5mu
  \raisebox{\dimen@}{$\m@th#1\cdot$}%
  \endgroup
}

\newcount\tableauRow\newcount\tableauCol
\newcommand\Tableau[1]{%
  \begin{tikzpicture}[scale=0.5,draw/.append style={thick,black},baseline=4mm]
    \tableauRow=0
    \foreach \Row in {#1} {
       \tableauCol=1
       \foreach\k in \Row {
          \draw(\the\tableauCol,\the\tableauRow)+(-.5,-.5)rectangle++(.5,.5);
          \draw(\the\tableauCol,\the\tableauRow)node{\k};
          \global\advance\tableauCol by 1
       }
       \global\advance\tableauRow by -1
    }
  \end{tikzpicture}
}
\newcommand\Tabloid[1]{%
  \begin{tikzpicture}[scale=0.5,draw/.append style={thick,black},baseline=4mm]
    \tableauRow=0
    \foreach \Row in {#1} {
       \tableauCol=1
       \foreach\k in \Row {
          \draw($(\the\tableauCol,\the\tableauRow)+(-.5,-.5)$)--++(1,0);
          \draw($(\the\tableauCol,\the\tableauRow)+(-.5,.5)$)--++(1,0);
          \draw(\the\tableauCol,\the\tableauRow)node{\k};
          \global\advance\tableauCol by 1
       }
       \global\advance\tableauRow by -1
    }
  \end{tikzpicture}
}
\newcommand\ColumnTabloid[1]{%
  \begin{tikzpicture}[scale=0.5,draw/.append style={thick,black},baseline=4mm]
    \tableauRow=0
    \foreach \Row in {#1} {
       \tableauCol=1
       \foreach\k in \Row {
          \draw($(\the\tableauCol,\the\tableauRow)+(-.5,-.5)$)--++(0,1);
          \draw($(\the\tableauCol,\the\tableauRow)+(.5,-.5)$)--++(0,1);
          \draw(\the\tableauCol,\the\tableauRow)node{\k};
          \global\advance\tableauCol by 1
       }
       \global\advance\tableauRow by -1
    }
  \end{tikzpicture}
}

\tikzset{decorate sep/.style 2 args=
{decorate,decoration={shape backgrounds,shape=circle,shape size=#1,shape sep=#2}}}

\newcommand{\demph}[1]{\textcolor{RoyalBlue}{\emph{#1}}}
\newcommand{\memph}[1]{\mathcolor{RoyalBlue}{#1}}
\def\supp{\qopname\relax m{supp}}
\usepackage{bm,fancybox}

\title{The quantum $k$-Bruhat order}

\begin{document}

\author[*]{Laura Colmenarejo}

\author[**]{Nicholas Mayers}

\affil[*]{Department of Mathematics, North Carolina State University, Raleigh, NC, 27605}

\affil[**]{Department of Mathematics, Kennesaw State University, Kennesaw, GA, 30144}

\maketitle

\begin{abstract}
   In this paper, we extend the study of the quantum $k$-Bruhat order initiated in the work of Benedetti, Bergeron, Colmenarejo, Saliola, and Sottile concerning the quantum Murnaghan-Nakayama rule. Specifically, identifying maximal chains in intervals of the quantum $k$-Bruhat order with sequences of transpositions, we investigate a naturally associated free monoid $\mathcal{F}_n^{\mathbf{q}}$ with an action on a $q$-extension of $S_n$, denoted $S_n[\mathbf{q}]$, which encodes the chain structure of the quantum $k$-Bruhat order. Aside from numerous structural results, our main contribution is an identification of a large family of equivalences satisfied by the elements of $\mathcal{F}_n^{\mathbf{q}}$ as operators on $S_n[\mathbf{q}]$. In fact, we conjecture that our list of equivalences is complete. As a consequence of the quantum Monk's rule, a complete understanding of such equivalences can be used to gain information about the multiplicative structure of quantum Schubert polynomials.   
\end{abstract}

\section{Introduction}

Schubert~\cite{AssafSchu,BB93,BJS93,SF94,Schubert,McDonald,Winkel1,Winkel2} and quantum Schubert polynomials~\cite{Quantum,Kirillov,Kirillov2,qBP} are two families of polynomials indexed by permutations whose multiplicative structures are of interest. Starting in the classical setting, geometrically, Schubert polynomials are polynomial representatives of Schubert classes in the cohomology of a flag variety~\cite{Schubert}. Due to the nature of this relationship, the structure constants of Schubert polynomials are known to be nonnegative integers that correspond to values of geometric interest. Specifically, they enumerate flags in a suitable triple intersection of Schubert varieties. While combinatorial formulas for such structure constants have been found in some special cases~\cite{BS2,Schubert,Monk,Sottile,Winkel}, it is a longstanding open problem to find a general combinatorial formula. Shifting to the quantum setting, through the works of~\cite{Ciocan, Quantum, GK1, K1, K2, K3, KM, RT}, a quantum version of Schubert polynomials is defined with similar properties and problems of interest. Geometrically, quantum Schubert polynomials are polynomial representatives of Schubert classes in the (small) quantum cohomology of a flag variety. As in the case of classical Schubert polynomials, this relationship results in the structure constants of quantum Schubert polynomials being nonnegative integers with geometric significance. In particular, in the quantum setting, the structure constants are 3-point Gromov invariants~\cite{Quantum}. Also, as in the classical case, while some results are known~\cite{qkB1}, finding general combinatorial interpretations for the structure constants of quantum Schubert problems is an open problem. In fact, a solution in the quantum case would immediately provide one in the classical setting. Here, following the work of Bergeron and Sottile~\cite{BS1,BS2} in the classical setting as well as the work of Benedetti, Bergeron, Colmenarejo, Saliola, and Sottile~\cite{qkB1} in the quantum setting, we study the quantum $k$-Bruhat order, whose chain structure encodes information concerning the structure constants of quantum Schubert polynomials.
 
To understand the classical motivation for this work, we briefly outline the contributions of Bergeron and Sottile~\cite{BS1,BS2}; see Section~\ref{sec:classic story} for further details. In the classical setting, Monk's rule determines the multiplicative structure of Schubert polynomials. Given a permutation $u$, let $\mathfrak{S}_u$ denote the corresponding Schubert polynomial, $\ell(u)$ denote the length of $u$, and $s_{ij}$ denote the permutation exchanging $i$ and $j$. 
\begin{theorem}[Monk's rule~\cite{Macdonald}]\label{thm:MonkRule}
  For $u\in S_n$ and $1\le k<n$,
$$\mathfrak{S}_u\cdot \mathfrak{S}_{s_{k,k+1}}=\sum_{1\le i\le k<j\le n\atop l(us_{ij})=l(u)+1}\mathfrak{S}_{us_{ij}}.$$  
\end{theorem}
\noindent
Using the indexing set of Monk's rule, one can define the $k$-Bruhat order on $S_n$, a poset whose interval structure encodes information concerning structure constants of Schubert polynomials. In particular, recall that the Schubert polynomial $\mathfrak{S}_v$ is a Schur polynomial whenever $v=(v_1,\hdots, v_n)\in S_n$ has a unique $k$ such that $v_k>v_{k+1}$. Then, letting $c_{uv}^w$ denote the coefficient of $\mathfrak{S}_w$ in $\mathfrak{S}_u\cdot \mathfrak{S}_v$, iterated forms of Monk's formula suggest a combinatorial interpretation of the form 
$$c_{uv}^w=\left|\left\{\text{maximal chains in}~[u,w]_k~\text{satisfying} \atop \text{\textit{some condition} imposed by}~v\right\}\right|,$$
where $[u,w]_k$ denotes an interval in the $k$-Bruhat order (see~\cite[Eq. (1.1.3)]{BS1}).

In a series of articles~\cite{BS1,BS2}, Bergeron and Sottile looked at the $k$-Bruhat order. First, in~\cite{BS1}, they present a compact means of comparing elements in the $k$-Bruhat order and then use it to relate the $k$-Bruhat order to what is called the Grassmannian Bruhat order; this latter ordering contains all intervals of the $k$-Bruhat order as intervals starting from a unique minimal element. Then, in a follow-up article~\cite{BS2}, they characterize a monoid $\mathcal{M}_n$ that encodes the chain structure of intervals in both the $k$-Bruhat and Grassmannian Bruhat orders on $S_n$. To describe the monoid $\mathcal{M}_n$, let $\lessdot_k$ denote the cover relation in the $k$-Bruhat order and $\mathcal{F}_n$ denote the collection of elements $\mathbf{v}_{ab}$ for $1\le a<b\le n$. Define an action of $\mathcal{F}_n$ together with $\mathbf{0}$ on the permutations of $S_n\cup\{0\}$ as follows: for $u\in S_n$, $$\mathbf{0}\bullet_ku=\mathbf{v_{ab}}\bullet_k0=0$$ and $$\mathbf{v}_{ab}\bullet_k w=\begin{cases}
    s_{ab}w,& \text{if}~w\lessdot_k s_{ab}w\\
    0, & otherwise.
\end{cases}$$
The monoid $\mathcal{M}_n$ is then the quotient of $\mathcal{F}_n\cup\{\mathbf{0}\}$ by the equivalence $\equiv$ defined by $\mathbf{v}\equiv \mathbf{u}$ if and only if $\mathbf{v}\bullet_kw=\mathbf{u}\bullet_kw$ for all $w\in S_n$ and all $k\in [n-1]$. After determining a complete list of equivalences defining $\mathcal{M}_n$~\cite{BS2}, they use their results to prove several formulas and relations for the structure constants of Schubert polynomials. 

Now, moving to our quantum inspiration, the story initially follows the same narrative with a quantum Monk's rule determining the multiplicative structure of quantum Schubert polynomials. Given a permutation $u$, let $\mathfrak{S}_u^q$ denote the corresponding quantum Schubert polynomial and let $\ast$ denote the quantum product.
\begin{theorem}[Quantum Monk's rule~\cite{Quantum}]
\label{thm:quantum-MonkRule}
For $u\in S_n$ and $1\le k<n$,   
$$\mathfrak{S}^{\bf q}_u\ast\mathfrak{S}^{\bf q}_{s_{k,k+1}}=\sum_{1\le i\le k<j\le n\atop \ell(us_{ij})=l(u)+1}\mathfrak{S}^{\bf q}_{us_{i,j}}+\sum_{1\le i\le k<j\le n\atop \ell(us_{i,j})+2(j-i)=\ell(u)+1}\mathbf{q_{ij}}\mathfrak{S}^{\bf q}_{us_{i,j}},
$$
where $\mathbf{q_{ij}} = q_i\cdots q_{j-1}$.
\end{theorem}
\noindent
As in the classical case, one defines a poset structure using quantum Monk's rule known as the ``quantum $k$-Bruhat order", now with underlying set $$S_n[\mathbf{q}]=\{q_1^{\alpha_1}\cdots q_{n-1}^{\alpha_{n-1}} w~|~w\in S_n~\text{and}~(\alpha_1,\hdots,\alpha_{n-1})\in\mathbb{Z}_{\ge 0}^{n-1}\}.$$ Contrary to the classical case, far less is known about the quantum $k$-Bruhat order. In~\cite{qkB1}, the authors introduce the quantum $k$-Bruhat order and initiate an investigation into a collection of operators with an action on $S_n[\mathbf{q}]$ analogous to that of $\mathcal{F}_n\cup\{\mathbf{0}\}$ and its action on $S_n$. Specifically, letting $\lessdot_k^{\mathbf{q}}$ denote the cover relation in the quantum $k$-Bruhat order and $\mathcal{F}_n^{\mathbf{q}}$ denote the collection of elements $\mathbf{v}_{ab}$ for $a\neq b\in [n]$, they define an action of $\mathcal{F}_n^{\mathbf{q}}$ together with $\mathbf{0}$ on $S_n[\mathbf{q}]\cup\{0\}$ as follows: for $w\in S_n$,and $k\in [n-1]$, we set $\mathbf{v}_{ab}\bullet_k0=\mathbf{0}\bullet_kw=0$, 
$$\mathbf{v}_{ab}\bullet_ku=\begin{cases}
    s_{ab}u,& \text{if}~a<b~\text{and}~u\lessdot_k^{\bf q}s_{ab}u \\
    \mathbf{q_{ij}}s_{ab}u, & \text{if}~a>b~\text{and}~u\lessdot_k^{\bf q}\mathbf{q_{ij}}s_{ab}u~(u(i)=a,u(j)=b)\\
    0, & otherwise,
\end{cases}$$
and extend $q$-multiplicatively to the remainder of $S_n[\mathbf{q}]$. Now, in their investigation of $\mathcal{F}_n^{\mathbf{q}}\cup\{\mathbf{0}\}$ and its action on $S_n[\mathbf{q}]$, the authors of~\cite{qkB1} study a equivalence weaker than that of its classical analogue (see Remark~\ref{rem:equivcomp}). Using their results, they then establish a quantum Murnaghan-Nakayama rule, combinatorially describing the structure constants arising when one multiplies a quantum Schubert polynomial and a quantum power sum.

In this paper, our focus is on a quantum analogue of the equivalence $\equiv$ studied in~\cite{BS2}, which will encode the chain structure of intervals in the quantum $k$-Bruhat order. Specifically, extending $\equiv$ to be the equivalence on $\mathcal{F}_n^{\mathbf{q}}\cup\{\mathbf{0}\}$ defined by $\mathbf{v}\equiv \mathbf{u}$ if and only if $\mathbf{v}\bullet_kw=\mathbf{u}\bullet_kw$ for all $w\in S_n[\mathbf{q}]$ and all $k\in [n-1]$, our main problem is as follows.
\begin{problem}
    Characterize the set of equivalences $\mathbf{u}\equiv\mathbf{v}$ and $\mathbf{u}\equiv\mathbf{0}$ for $\mathbf{u},\mathbf{v}\in\mathcal{F}^{\mathbf{q}}_n$.
\end{problem}
\noindent
With regards to this problem, we determine all equivalences associated with compositions of at most four elements of $\mathcal{F}_n^{\mathbf{q}}$. In addition, we identify two families of equivalences consisting of compositions of an arbitrary number of elements of $\mathcal{F}_n^{\mathbf{q}}$. We conjecture that the equivalences found, in fact, constitute a complete collection.

The remainder of this paper is structured as follows. In Section~\ref{sec:prelim}, we introduce basic notions needed concerning permutations and poset theory. In Section~\ref{sec:classic story}, we summarize the study done by Bergeron and Sottile regarding the classical framework, including properties of the $k$-Bruhat and the Grassmannian Bruhat order as well as their study of the associated monoid $\mathcal{M}_n$. Then, in Sections~\ref{sec:quantumintro} --~\ref{sec:arblength}, we present our study on the quantum $k$-Bruhat order. Specifically, starting in Section~\ref{sec:quantumintro}, we introduce the quantum $k$-Bruhat order, the collection $\mathcal{F}_n^{\mathbf{q}}$ with its action on $S_n[\mathbf{q}]$, and the main problem of focus here. Following this, in Section~\ref{sec:ops}, we introduce and study equivalence preserving transformations on compositions of elements of $\mathcal{F}_n^{\mathbf{q}}$ which play a crucial role in what follows. Finally, in Sections~\ref{sec:ldequiv}  and~\ref{sec:arblength}, we identify defining equivalences satisfied by compositions of elements of $\mathcal{F}_n^{\mathbf{q}}$, with Section~\ref{sec:ldequiv} characterizing equivalences involving compositions of between 2 and 4 elements, and Section~\ref{sec:arblength} focusing on equivalences involving compositions of an arbitrary number of elements. As noted above, we conjecture that the equivalences found in these latter two sections form a complete collection needed to define a quantum analogue of $\mathcal{M}_n$.

\section{Preliminaries}\label{sec:prelim}

In this section, we cover the necessary concepts from the theories of permutations and posets.

Let $[n]:=\{1,2,\hdots,n\}$ for $n\in\mathbb{Z}_{>0}$. Recall that the \demph{symmetric group} $S_n$ consists of all bijections $w:[n]\to [n]$, called \demph{permutations}, with group operation defined by composition of functions. We denote by $e$ the identity element in $S_n$. Ongoing, we often express permutations $w\in S_n$ in one-line notation $w=(w(1),w(2),\ldots,w(n))\in [n]^n$. In our examples, we also write $w=w(1)w(2)\ldots w(n)$, where the numbers are listed without parentheses and commas. We reserve the notation $w_1,\ w_2,\ldots, w_k$ to denote sequences of permutations $w_i\in S_n$.

For $i,j\in [n]$ with $i<j$, we let $s_{i,j}=s_{ij}\in S_n$ denote the permutation for which $s_{ij}(i)=j$, $s_{ij}(j)=i$, and $s_{ij}(k)=k$ for $k\neq i,j$; such permutations are referred to as \demph{transpositions}. For $k\in [n-1]$, we denote $s_{k,k+1}\in S_n$ simply as $s_k$. Given a permutation $w=(w_1,\hdots,w_n)\in S_n$, we define its \demph{support} as the set of non-fixed points, $\supp(w):=\{i\in [n]~|~w(i)\neq i\}$, and its \demph{length} as the number of inversions, $\ell(w):=|\{(i,j)~|~1\leq i<j\leq n,~w(i)>w(j)\}|$.

As noted in the introduction, in this paper, we will be concerned with certain posets in $S_n$. Briefly, recall that a \demph{poset} (partial ordered set) $(\mathcal{P},\preceq_{\mathcal{P}})$ consists of a set $\mathcal{P}$ along with a partial order relation $\preceq_{\mathcal{P}}$ on $\mathcal{P}$, i.e., a binary relation which is reflexive, anti-symmetric, and transitive. We say that two posets $(\mathcal{P},\preceq_P)$ and $(\mathcal{Q},\preceq_Q)$ are \demph{isomorphic} if there exists an order-preserving bijection $f:\mathcal{P}\to\mathcal{Q}$ for which its inverse $f^{-1}$ is also order-preserving. When no confusion will arise, we simply denote a poset $(\mathcal{P}, \preceq_{\mathcal{P}})$ by $\mathcal{P}$, and $\preceq_{\mathcal{P}}$ by $\preceq$.

Given a poset $(\mathcal{P},\preceq)$ and $x,y\in\mathcal{P}$, we write $x\prec y$ if $x\preceq y$ and $x\neq y$. In the case that $x\prec y$ and there exists no $z\in \mathcal{P}$ satisfying $x\prec z\prec y$, we say that $x\prec y$ is a \demph{covering relation} and write $x\precdot y$. Note that, as a consequence of transitivity, the relations of $\mathcal{P}$ are completely determined by its collection of covering relations, i.e., we can define $\mathcal{P}$ by specifying the set $\mathcal{P}$ and along with the set of covering relations associated with $\preceq$. Covering relations are used to define a visual representation of $\mathcal{P}$ called the \demph{Hasse diagram} -- a graph whose vertices correspond to elements of $\mathcal{P}$ and whose edges correspond to covering relations (see Figure~\ref{fig:qkBru}). Given a subset $S\subseteq\mathcal{P}$, the \demph{induced subposet generated by $S$} is the poset $\mathcal{P}_S=(S,\preceq_S)$ on $S$, where $i\preceq_S j$ if and only if $i\preceq j$. A subset $\mathcal{C}\subseteq\mathcal{P}$ for which $\mathcal{P}_C$ is a total order, i.e., any two elements are related, is referred to as a \demph{chain}. We refer to a chain as \demph{maximal} if it is not contained in any strictly larger chain. Given $x,y\in \mathcal{P}$, we define the \demph{interval between $x$ and $y$}, denoted $[x,y]_{\preceq}$, to be the induced subposet of $\mathcal{P}$ corresponding to the subset $\{z\in \mathcal{P}~|~x\preceq z\preceq y\}$.

\section{The classical $k$-Bruhat order}\label{sec:classic story}

In this section, we outline the classical framework regarding the $k$-Bruhat order and its associated monoid; for more details, see~\cite{BS1, BS2}. 

As noted in the introduction, the $k$-Bruhat order is defined via the indexing set in Monk's rule (Theorem~\ref{thm:MonkRule}). The \demph{$k$-Bruhat order on $S_n$} is the poset $(S_n,\leq_k)$, where $\leq_k$ is defined by the following cover relations: $u \lessdot_k us_{ij}$ whenever $i\leq k < j$ and $\ell(us_{ij}) = \ell(u)+1$. Equivalently, $u \lessdot_k us_{ij}$ whenever $i\leq k < j$, $u(i) < u(j)$, and for all $l\in [i,j]$, $u(l)\notin [u(i),u(j)]$. We denote the intervals in the $k$-Bruhat order by $[u,v]_k$ for $u,v\in S_n$. In Figure~\ref{fig:kBru}~(a), we illustrate the $2$-Bruhat order on $S_3$.

\begin{figure}[H]
    \centering
    $$\scalebox{0.7}{\begin{tikzpicture}
    \node (1) at (0,0) {$12|3$};
    \node (2) at (-1,1.5) {$21|3$};
    \node (3) at (1,1.5) {$13|2$};
    \node (4) at (-1,3) {$31|2$};
    \node (5) at (1,3) {$23|1$};
    \node (7) at (0,4.5) {$32|1$};
    \draw (1)--(3);
    \draw (7)--(4)--(2)--(5)--(3);
    \node at (0,-1) {$(a)$};
\end{tikzpicture}}\quad\quad\quad\quad\quad\quad\quad\quad \scalebox{0.7}{\begin{tikzpicture}
    \node (1) at (0,0) {$123$};
    \node (2) at (-2,1.5) {$213$};
    \node (3) at (0,1.5) {$321$};
    \node (4) at (2,1.5) {$132$};
    \node (5) at (-2,3) {$312$};
    \node (6) at (2,3) {$321$};
    \draw (6)--(4)--(1)--(2)--(5);
    \draw (1)--(3);
    \node at (0,-1) {$(b)$};
\end{tikzpicture}}$$
    \caption{(a) 2-Bruhat order on $S_3$ and (b) Grassmannian Bruhat order on $S_3$}
    \label{fig:kBru}
\end{figure}
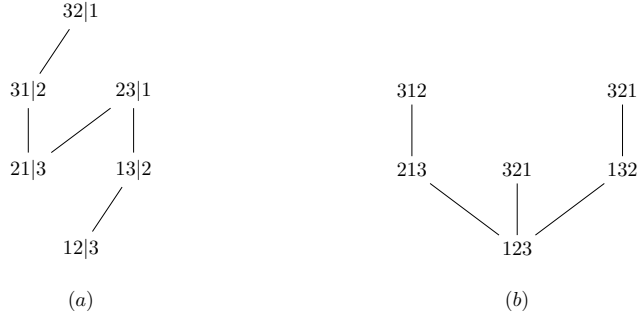

Bergeron and Sottile~\cite{BS1} studied the $k$-Bruhat order in detail and showed the following.
\begin{prop}[\cite{BS1}]\label{prop:BS}
    Let $n\geq 1$, $u,v\in S_n$, and $k,l\in [n-1]$. \begin{enumerate}
        \item[$(a)$] $u\le_kw$ if and only if for all $a,b\in[n]$, $(i)$ $a\le k\le b$ implies $u(a)\le w(a)$ and $u(b)\ge w(b)$, and $(ii)$ if $a<b$, $u(a)<u(b)$, and $w(a)>w(b)$, then $a\le k<b$.  
        \item[$(b)$] Suppose that $u\le_kw$, $x\le_lz$, and $wu^{-1}=zx^{-1}$. Then $v\longmapsto vu^{-1}x$ induces an isomorphism between $[u,w]_k$ and $[x,z]_l$. 
    \end{enumerate}
\end{prop}

Note that covering relations $u\lessdot_k v$ correspond to transpositions that keep track of the positions (resp., values) exchanged, i.e., $v=us_{ij}$ (resp., $v=s_{ab}u$), which we refer to as the left (resp., right) transposition associated with $u\lessdot_k v$. Consequently, maximal chains in an interval $[u,v]_k$ correspond to sequences of left (resp., right) transpositions that one ``applies'' sequentially to $u$, and one can label the edges of the Hasse diagram accordingly. 
In~\cite{BS1,BS2}, the authors note that when considering intervals $[u,w]_k$ and $[x,z]_l$ as in Proposition~\ref{prop:BS}~(b), by the particular isomorphism identified, their Hasse diagrams have identical labelings by left transpositions; see Figures~\ref{fig:isoint}~(a) and (b) for an example. Note that this is not necessarily the case for labelings by right transpositions. 

Motivated by this phenomenon, Bergeron and Sottile were led to an additional poset structure on $S_n$. The \demph{Grassmannian Bruhat order} $(S_n,\preceq)$ is defined as follows: For $\eta, \xi \in S_n$, $\eta \preceq \xi$ if there exist $u\in S_n$ and $k\in [n-1]$ such that $u\leq_k \eta u \leq_k \xi u$. In Figure~\ref{fig:kBru}~(b), we illustrate the Grassmannian Bruhat order on $S_3$.

\begin{figure}[H]
    \centering
    $$\scalebox{0.7}{\begin{tikzpicture}[scale=0.8]
        \node at (0,-1) {(a) $\le_2$};
        \node (1) at (0,0) {$13|425$};
        \node (2) at (-2,3) {$14|325$};
        \node (3) at (2,3) {$23|415$};
        \node (4) at (-2,6) {$15|324$};
        \node (5) at (2,6) {$24|315$};
        \node (6) at (0,9) {$25|314$};
        \node at (-1.5,1.5) {$s_{34}$};
        \node at (1.5,1.5) {$s_{12}$};
        \node at (-2.5,4.5) {$s_{45}$};
        \node at (2.5,4.5) {$s_{34}$};
        \node at (0,5) {$s_{12}$};
        \node at (-1.5,7.5) {$s_{12}$};
        \node at (1.5,7.5) {$s_{45}$};
        \draw[red] (1)--(2);
        \draw (2)--(4)--(6);
        \draw[red] (6)--(5);
        \draw (5)--(3)--(1);
        \draw[red] (2)--(5);    \end{tikzpicture}}\quad\quad\quad\quad\scalebox{0.7}{\begin{tikzpicture}[scale=0.8]
        \node at (0,-1) {(b) $\le_3$};
        \node (1) at (0,0) {$613|425$};
        \node (2) at (-2,3) {$614|325$};
        \node (3) at (2,3) {$623|415$};
        \node (4) at (-2,6) {$615|324$};
        \node (5) at (2,6) {$624|315$};
        \node (6) at (0,9) {$625|314$};
        \node at (-1.5,1.5) {$s_{34}$};
        \node at (1.5,1.5) {$s_{12}$};
        \node at (-2.5,4.5) {$s_{45}$};
        \node at (2.5,4.5) {$s_{34}$};
        \node at (0,5) {$s_{12}$};
        \node at (-1.5,7.5) {$s_{12}$};
        \node at (1.5,7.5) {$s_{45}$};
        \draw (1)--(2)--(4)--(6)--(5)--(3)--(1);
        \draw (2)--(5);
    \end{tikzpicture}}\quad\quad\quad\quad\scalebox{0.7}{\begin{tikzpicture}[scale=0.8]
        \node at (0,-1) {(c) $\preceq$};
        \node (1) at (0,0) {$12345$};
        \node (2) at (-2,3) {$12435$};
        \node (3) at (2,3) {$21345$};
        \node (4) at (-2,6) {$12534$};
        \node (5) at (2,6) {$21435$};
        \node (6) at (0,9) {$21534$};
        \node at (-1.5,1.5) {$s_{34}$};
        \node at (1.5,1.5) {$s_{12}$};
        \node at (-2.5,4.5) {$s_{45}$};
        \node at (2.5,4.5) {$s_{34}$};
        \node at (0,5) {$s_{12}$};
        \node at (-1.5,7.5) {$s_{12}$};
        \node at (1.5,7.5) {$s_{45}$};
        \draw (1)--(2)--(4)--(6)--(5)--(3)--(1);
        \draw (2)--(5);
    \end{tikzpicture}}$$
    \caption{Isomoprhic intervals in (a) $\le_2$, (b) $\le_3$, and (c) $\preceq$}
    \label{fig:isoint}
\end{figure}
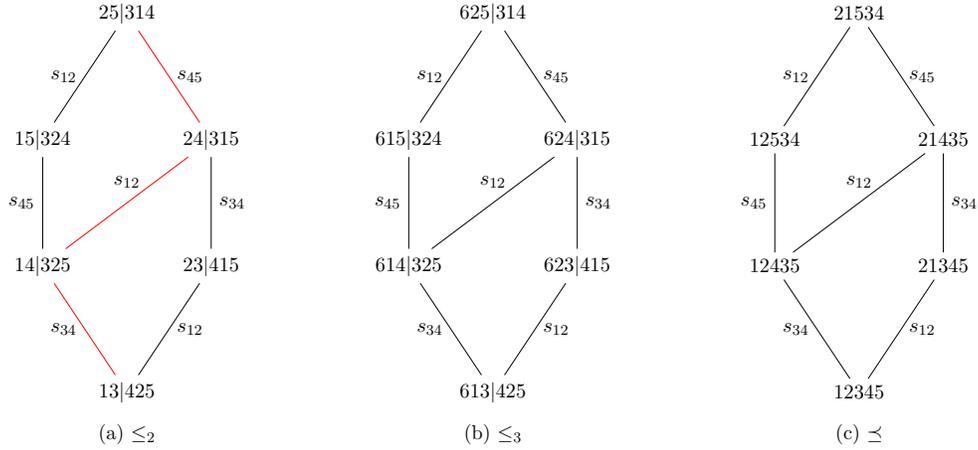

The following result shows that the Grassmannian Bruhat order provides a uniform means of studying the intervals of the $k$-Bruhat order.
\begin{theorem}[\cite{BS1}]\label{thm:BS1}
Let $u,\zeta\in S_n$ and $k\in[n-1]$ be such that $u\le_k\zeta u$. Then $\nu\longmapsto\nu u$ induces an isomorphism between $[e,\zeta]_\preceq$ and $[u,\zeta u]_k$.
\end{theorem}

Analogously to the case of the $k$-Bruhat order, covering relations in the Grassmannian Bruhat order have associated left and right transpositions. Considering the isomorphism identified in Theorem~\ref{thm:BS1}, intervals $[e,\zeta]_\preceq$ and $[u,\zeta u]_k$ have identical labelings by left transpositions; see Figure~\ref{fig:isoint}~(c) for an example. 

\subsection*{The monoid structure of $S_n$}

In attempting to extract some of the multiplicative information of Schubert polynomials encoded by intervals of the $k$-Bruhat order, Bergeron and Sottile introduced a free monoid whose generators represent transpositions. 
Recall from above that, for an interval $[u,v]_k$ in the $k$-Bruhat order, there exists an isomorphic interval $[e,\zeta]_\preceq$ in the Grassmannian Bruhat order. Under this isomorphism, maximal chains are identified by their sequence of left transpositions; that is, ``applying'' such a sequence to $u$ and $e$ sequentially on the left results in $v$ and $\zeta$, respectively. 
Consequently, the chain structure of intervals in the $k$-Bruhat order on $S_n$ for all choices of $k\in [n-1]$ is uniformly encoded in that of the Grassmannian Bruhat order on $S_n$ via sequences of left transpositions. Replacing transpositions with ``operators'', we are led to the following. 

For $n\geq 1$, let \demph{$\mathcal{F}_n$} be the free monoid with generators $\mathbf{v}_{ab}$ for $a,b\in [n]$ with $a<b$ and let $\mathbf{0}$ be the zero element. Then, using the above poset structures on $S_n$, we define \demph{actions of $\mathcal{F}_n \cup \{\mathbf{0}\}$ on $S_n \cup \{0\}$} as follows: For $\mathbf{v}_{ab} \in \mathcal{F}_n$, $w\in S_n\cup\{0\}$, and $k\in [n-1]$, $$\mathbf{0} \bullet_k w=\mathbf{v}_{ab}\bullet_k0 = \mathbf{0} \cdot w=\mathbf{v}_{ab}\cdot 0 = 0,$$
$$
\begin{array}{llc}
\mathbf{v}_{ab}\bullet_k w=\begin{cases}
    s_{ab}w,& \text{if}~w\lessdot_k s_{ab}w\\
    0, & otherwise,
\end{cases}
&\qquad\text{and}\qquad \qquad &
\mathbf{v}_{ab}\cdot w=\begin{cases}
    s_{ab}w,& \text{if}~w\precdot s_{ab}w\\
    0, & otherwise.
\end{cases}
\end{array}
$$
We refer to a single generator of $\mathcal{F}_n$ as an \demph{operator} and to a product of generators as a \demph{composition of operators}. For a composition of operators $\mathbf{v}=\mathbf{v}_{a_rb_r}\cdots\mathbf{v}_{a_1b_1}\in\mathcal{F}_n$, we define the \demph{support} of $\mathbf{v}$ by $\supp(\mathbf{v})=\{a_1,b_1,\hdots,a_r,b_r\}$. Note that for $\mathbf{v}=\mathbf{v}_{a_rb_r}\cdots\mathbf{v}_{a_1b_1}\in\mathcal{F}_n$ and $u,w\in S_n$, we have $\mathbf{v}\bullet_k u=w$ (resp., $\mathbf{v}\cdot u=w$) if and only if there is maximal chain in the $k$-Bruhat order (resp., Grassmannian Bruhat order) from $u$ to $w$ whose cover relations, in order, have associated left transpositions $s_{a_ib_i}$ for $i\in [r]$. For example, since $\mathbf{v}=\mathbf{v}_{45}\mathbf{v}_{12}\mathbf{v}_{34}\in\mathcal{F}_5$ satisfies $\mathbf{v}\bullet_2 13425=25314$, $\mathbf{v}$ corresponds to the highlighted maximal chain in Figure~\ref{fig:isoint}~(a). Ongoing, we will identify maximal chains in $\preceq$ (resp., $\le_k$) with this associated sequence of left transposition.

Studying the actions of $\mathcal{F}_n\cup\{\mathbf{0}\}$ on $S_n\cup \{0\}$ defined above, Bergeron and Sottile were led to the following monoid.
\begin{definition}\label{def:monoid}
    Define $\mathcal{M}_n$ to be the monoid formed as the quotient of $\mathcal{F}_n\cup \{{\bf 0}\}$ by the relations
\begin{itemize}
    \item[\textup{(i)}] $\mathbf{v}_{bc}\mathbf{v}_{cd}\mathbf{v}_{ac}\equiv \mathbf{v}_{bd}\mathbf{v}_{ab}\mathbf{v}_{bc}$ \quad and \quad $\mathbf{v}_{ac}\mathbf{v}_{cd}\mathbf{v}_{bc}\equiv\mathbf{v}_{bc}\mathbf{v}_{ab}\mathbf{v}_{bd}$ \quad if $a<b<c<d$,
    \item[\textup{(ii)}]
    $\mathbf{v}_{ab}\mathbf{v}_{cd}\equiv \mathbf{v}_{cd}\mathbf{v}_{ab}$ \quad if $b<c$ or $a<c<d<b$,
    \item[\textup{(iii)}] $\mathbf{v}_{ac}\mathbf{v}_{bd}\equiv\mathbf{v}_{bd}\mathbf{v}_{ac}\equiv\mathbf{0}$ \quad if $a\le b < c\le d$,
    \item[\textup{(iv)}] $\mathbf{v}_{bc}\mathbf{v}_{ab}\mathbf{v}_{bc}\equiv\mathbf{v}_{ab}\mathbf{v}_{bc}\mathbf{v}_{ab}\equiv\mathbf{0}$ \quad if $a<b<c$.
\end{itemize}
\end{definition}
The monoids $\mathcal{F}_n$ and $\mathcal{M}_n$ are related to the $k$-Bruhat and the Grassmannian Bruhat orders in the following way.
\begin{theorem}[\cite{BS2}]\label{thm:BS2}
\begin{enumerate}
    \item[$(a)$] The map $\mathcal{F}_n\cup \{\mathbf{0}\} \to S_n\cup \{0\}$ defined by $\mathbf{v}\longmapsto\mathbf{v}\cdot e$ is surjective, factors through $\mathcal{M}_n$, and induces a bijection between $\mathcal{M}_n$ and $S_n\cup \{0\}$.
            
    \item[$(b)$] For $u,w\in S_n$ and $k\in [n-1]$, the set $\{\mathbf{v}\in\mathcal{F}_n~|~\mathbf{v}\bullet_k u=w\}$ is in bijection with the set of maximal chains in $[u,w]_k$ via the identification through sequences of left transpositions noted above. In fact, if $\mathbf{v},\mathbf{v}'$ belong to the aforementioned set, then they are related by a sequence of relations $(i)-(ii)$ in Definition~\ref{def:monoid}. Moreover, it is never possible to apply any of the relations $(iii)-(v)$ to obtain such operators.

    \item[$(c)$] For $u,w\in S_n$ and $k\in [n-1]$, if $\mathbf{v}\bullet_k u=w$ and $\mathbf{v'}$ is related to $\mathbf{v}$ by a sequence of relations $(i)-(ii)$ in Definition~\ref{def:monoid}, then $\mathbf{v'}\bullet_k u=w$.

    \item[$(d)$] For $\zeta\in S_n$, the set $\{\mathbf{v}\in\mathcal{F}_n~|~\mathbf{v}\cdot e=\zeta\}$ is in bijection with the set of maximal chains in $[e,\zeta]_{\preceq}$.
    \end{enumerate}
\end{theorem}

As a consequence of the work in~\cite{BS2}, we have the following.

\begin{theorem}\label{thm:equivdef}
    Two elements $\mathbf{u},\mathbf{v}\in \mathcal{F}_n\cup\{\mathbf{0}\}$ belong to the same equivalence class of $\mathcal{M}_n$ if and only if $\mathbf{v}\bullet_kw=\mathbf{u}\bullet_kw$ for all $w\in S_n$ and all $k\in[n-1]$.
\end{theorem}
\begin{proof}
\Ovalbox{$\Longrightarrow$} 
    Consider $w\in S_n$ and $k\in [n-1]$. It suffices to show that $\mathbf{u}\bullet_kw=\mathbf{v}\bullet_kw$. If $\mathbf{u}\bullet_kw=\mathbf{v}\bullet_kw=0$, then we are done. So, assume that $\mathbf{u}\bullet_kw=w^*\in S_n$. Then, applying Theorem~\ref{thm:BS2}~(c), we find that $\mathbf{v}\bullet_kw=w^*\in S_n$, as desired.

\Ovalbox{$\Longleftarrow$} 
    Suppose first that $\mathbf{u}$ and $\mathbf{v}$ have non-zero action on some permutation of $S_n$; that is, there exist $w,w'\in S_n$ and $k\in [n-1]$ for which $\mathbf{v}\bullet_kw=\mathbf{u}\bullet_kw=w'$. Then, by Theorem~\ref{thm:BS2}~(b), $\mathbf{v}$ and $\mathbf{u}$ correspond to maximal chains of $[w,w']_k$ and $\mathbf{u}\equiv\mathbf{v}$, as desired. Now, suppose that $\mathbf{v}\bullet_kw=\mathbf{u}\bullet_kw=0$ for all $w\in S_n$ and all $k\in[n-1]$. We claim that $\mathbf{v}\equiv\mathbf{0}\equiv\mathbf{u}$, and argue by contraction. Without loss of generality, assume that $\mathbf{v}\not\equiv\mathbf{0}$. Then, considering Theorem~\ref{thm:BS2}~(a), it follows that there exists $\zeta\in S_n$ for which $\mathbf{v}\cdot 1=\zeta$. By the definition of the Grassmannian Bruhat order, there exist $u\in S_n$ and $k\in [n-1]$ for which $u\le_k u\le_k\zeta u$. Thus, applying Theorem~\ref{thm:BS1}, we have that $[u,\zeta u]_k\cong [1,\zeta]_\preceq$. Moreover, considering the bijection provided in Theorem~\ref{thm:BS1}, it must be the case that $\mathbf{v}\bullet_k u=\zeta u$, which is a contradiction. The claim, and consequently the result, follows.
\end{proof}

For the ordering $\preceq$ (resp., $\le_k$), relations (i) and (ii) in the definition of $\mathcal{M}_n$ correspond to relations which can be applied to take any maximal chain between elements $u,v\in S_n$ to any other maximal chain in $[u,v]_k$ (resp., $[u,v]_{\preceq}$). On the other hand, relations (iii) and (iv) correspond to sequences of left transpositions that never occur in any maximal chain.

\section{The quantum $k$-Bruhat order}\label{sec:quantumintro}

In this section, mirroring the discussion in the classical case, we define the quantum version of the $k$-Bruhat order and introduce the quantum version of $\mathcal{F}_n$, whose study will occupy the rest of the paper. To start, we introduce the quantum $k$-Bruhat order.

In this case, the underlying set of our poset of interest is the following deformation of $S_n$. Consider the set of indeterminates $\memph{\mathbf{q}}:=\{q_1,q_2,\ldots,q_{n-1}\}$. We set $\displaystyle{\memph{\mathbf{q}^{\alpha}}:=\prod_{i=1}^{n-1}q_i^{\alpha_i}}$, for $\alpha=(\alpha_1,\hdots,\alpha_{n-1})\in\mathbb{Z}^{n-1}_{\ge 0}$, and $\displaystyle{\memph{\mathbf{q_{ij}}} := \prod_{k=i}^{j-1}q_k}$, for $1\le i<j\leq n-1$. Moreover, for $\alpha\in\mathbb{Z}^{n-1}_{\ge 0}$, we denote by $\deg \mathbf{q}^\alpha$ the usual (total) degree of the monomial $\mathbf{q}^\alpha$ in $\mathbb{Z}[\mathbf{q}]$. 
Now, define 
$$
\memph{S_n[\mathbf{q}]}:=\{{\bf q}^\alpha w~|~w\in S_n~\text{and}~{\alpha\in\mathbb{Z}_{\ge 0}^{n-1} }\},
$$ 
and extend the \demph{length function} $\ell$ on $S_n$ to $S_n[\mathbf{q}]$ by  
$\ell(\mathbf{q}^{\alpha}w)=\ell(w)
+2~\text{deg}~\mathbf{q}^{\alpha}$. Ongoing, we refer to $S_n$ and its framework as the \demph{classical case} while we refer to $S_n[\mathbf{q}]$ and its framework as the \demph{quantum case}. The classical framework can be recovered from the quantum one by setting $q_i=0$ for all $i=1,\ldots, n-1$.

\begin{remark}
    Moving forward, given an element $\mathbf{q}^{\alpha}w\in S_n[\mathbf{q}]$, it is assumed that $w\in S_n$.
\end{remark}

Analogously to the classical framework, the \demph{quantum k-Bruhat order}  $\left(S_n[{\bf q}],\preceq_k^{\bf q}\right)$ is defined by its cover relations, which arise from the indexing set in the quantum Monk's rule (Theorem~\ref{thm:quantum-MonkRule}): For $w\in S_n$,
\begin{enumerate}
    \item[(1)] $w\lessdot_k^{\bf q}ws_{ij}$ if $i\le k<j$ and $\ell(w)+1=\ell(ws_{ij})$,
    \item[(2)] $w\lessdot_k^{\bf q}\mathbf{q_{ij}}ws_{ij}$ if $i\le k<j$ and $\ell(w)+1=\ell(\mathbf{q_{ij}}ws_{ij})$, and
    \item[(3)] extend $\bf{q}$-multiplicatively: $\mathbf{q}^\alpha  u\le_k^{\bf q}\mathbf{q}^\beta v$ if and only if $\mathbf{q}^{\gamma}\mathbf{q}^\alpha u\le_k^{\bf q}\mathbf{q}^{\gamma}\mathbf{q}^\beta v$ for $\mathbf{q}^\alpha u,\mathbf{q}^\beta v\in S_n[\mathbf{q}]$ and any $\gamma\in\mathbb{Z}_{\ge 0}^{n-1}$.
\end{enumerate}
We refer to the cover relations of the form (1) as \textit{classical} and those of the form (2) as \textit{quantum}. The following result provides an alternative way to understand the cover relations of the quantum $k$-Bruhat order that does not require computing the length of the permutations. 
\begin{prop}[\cite{qkB1}]\label{prop:covercond}
Let $w\in S_n$ and $u\in S_n[\mathbf{q}]$. Then $w\lessdot_k^{\mathbf{q}} \mathbf{q}^\alpha u$ if and only if 
\begin{enumerate}
    \item[$(a)$] $\mathbf{q}^\alpha u=ws_{ij}$, for some $i\le k<j$, such that $w(i)<w(j)$, and there does not exist $l$ such that $i<l<j$ and $w(i)<w(l)<w(j)$; or
    \item[$(b)$] $\mathbf{q}^\alpha u=\mathbf{q_{ij}}ws_{ij}$, for some $i\le k<j$, such that $w(i)>w(j)$, and for every $l$ such that $i<l<j$, we have $w(j)<w(l)<w(i)$. 
\end{enumerate}
\end{prop}

\begin{remark}\label{rem:relbrels}
    Note that, considering the definitions of $<_k$ and $<_k^q$, for $u,v\in S_n$, we have that $u<_kv$ if and only if $\mathbf{q}^\alpha u<_k^q \mathbf{q}^\alpha v$ for any $\alpha \in \mathbb{Z}_{\geq 0}^{n-1}$.
\end{remark}

Ongoing, when discussing elements of the quantum $k$-Bruhat order on $S_n[\mathbf{q}]$, we present permutations of $S_n$ with a bar between the value in position $k$ and the value in position $k+1$. Then the cover relations correspond to exchanging two values $w(i)$ and $w(j)$, one on the left of the bar and the other one on the right, when the numbers in positions $i,\ldots,j$ satisfy one of the two conditions discussed in Proposition~\ref{prop:covercond}. 

\begin{Ex}
    Note that in the quantum 2-Bruhat order on $S_3[\mathbf{q}]$ we have that $21|3\lessdot_2^{\bf q} 23|1$ is a classical cover, while $32|1\lessdot_2^{\bf q}q_231|2$ is a quantum cover. A portion of the Hasse diagram of the quantum 2-Bruhat order on $S_3[\mathbf{q}]$ is illustrated below in Figure~\ref{fig:qkBru}.
\end{Ex}

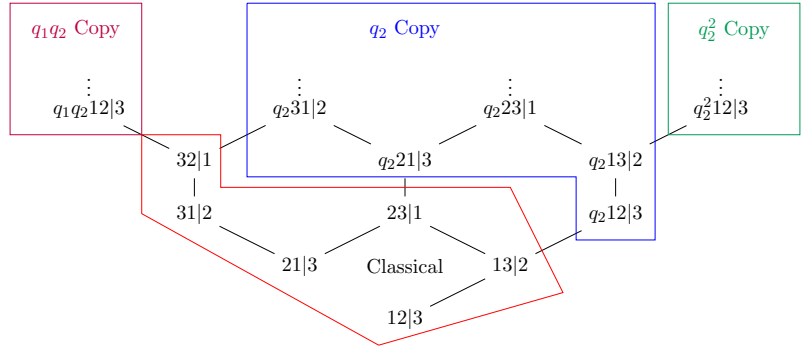
\begin{figure}[H]
    \centering
    $$\scalebox{0.7}{\begin{tikzpicture}
    \node (1) at (0,0) {$12|3$};
    \node (2) at (-2,1) {$21|3$};
    \node (3) at (2,1) {$13|2$};
    \node (4) at (-4,2) {$31|2$};
    \node (5) at (0,2) {$23|1$};
    \node (6) at (4,2) {$q_212|3$};
    \node (7) at (-4,3) {$32|1$};
    \node (8) at (0,3) {$q_221|3$};
    \node (9) at (4,3) {$q_213|2$};
    \node (10) at (-6,4) {$q_1q_212|3$};
    \node (11) at (-2,4) {$q_231|2$};
    \node (12) at (2,4) {$q_223|1$};
    \node (13) at (6,4) {$q^2_212|3$};
    \node at (-6,4.5) {$\vdots$};
    \node at (-2,4.5) {$\vdots$};
    \node at (2,4.5) {$\vdots$};
    \node at (6,4.5) {$\vdots$};
    \draw (1)--(3)--(6)--(9)--(12)--(8)--(11)--(7)--(4)--(2)--(5)--(3);
    \draw (5)--(8);
    \draw (9)--(13);
    \draw (7)--(10);
    \draw[red] (-0.5,-0.5)--(-5,2)--(-5,3.5)--(-3.5,3.5)--(-3.5,2.5)--(2,2.5)--(3,0.5)--(-0.5,-0.5);
    \node at (0,1) {Classical};
    \draw[blue] (-3,6)--(-3,2.7)--(3.25,2.7)--(3.25,1.5)--(4.75,1.5)--(4.75,6)--(-3,6);
    \node at (0,5.5) {\bl{$q_2$ Copy}};
    \draw[purple] (-5,3.5)--(-7.5,3.5)--(-7.5,6)--(-5,6)--(-5,3.5);
    \node at (-6.25,5.5) {\pp{$q_1q_2$ Copy}};
    \draw[ForestGreen] (5,3.5)--(7.5,3.5)--(7.5,6)--(5,6)--(5,3.5);
    \node at (6.25,5.5) {\textcolor{ForestGreen}{$q^2_2$ Copy}};
\end{tikzpicture}}$$
    \caption{Quantum 2-Bruhat order on $S_3[\mathbf{q}]$}
    \label{fig:qkBru}
\end{figure}

\subsection*{The monoid structure on $S_n[\mathbf{q}]$}

We define a monoid structure with an action on $S_n[\mathbf{q}]$ coming from the covering relations of the quantum $k$-Bruhat order. This structure and accompanying action were first introduced in~\cite{qkB1}, and they generalize the classical framework studied by Bergeron and Sottile in~\cite{BS2}.

Let $\memph{\mathcal{F}_n^{\bf q}}$ be the free monoid generated by the symbols $\{\mathbf{v}_{ab}~|~a,b\in[n],\ a\neq b\}$, and let $\mathbf{0}$ denote a zero element, referred to as the \demph{zero operator}. As in the classical setting, we refer to a single generator of $\mathcal{F}_n^{\mathbf{q}}$ as an operator and a product of generators as a composition of operators. The elements in $\mathcal{F}^{\mathbf{q}}_n$ are then compositions of operators of the form $\mathbf{v} = \mathbf{v}_{a_rb_r}\cdots\mathbf{v}_{a_1b_1}$ (where the labeling from right to left reflects the left action defined next). We define the \demph{support} of $\mathbf{v}$ as the set $\supp(\mathbf{v}) = \{a_1,b_1,\ldots, a_r,b_r\}$ and say that $\mathbf{v}$ has \demph{degree} $r$ to indicate that it involves $r$ operators.

Then, $\mathcal{F}^{\mathbf{q}}_n\cup\{\mathbf{0}\}$ acts on $S_n[\mathbf{q}] \cup \{0\}$ in the following way: For $\mathbf{v}_{ab}\in\mathcal{F}_n^{\mathbf{q}}$, $w\in S_n\cup\{0\}$, and $k\in [n-1]$, we set $\mathbf{v}_{ab}\bullet_k0=\mathbf{0}\bullet_kw=0$, 
$$\mathbf{v}_{ab}\bullet_ku=\begin{cases}
    s_{ab}u,& \text{if}~a<b~\text{and}~u\lessdot_k^{\bf q}s_{ab}u \\
    \mathbf{q_{ij}}s_{ab}u, & \text{if}~a>b~\text{and}~u\lessdot_k^{\bf q}\mathbf{q_{ij}}s_{ab}u~(u(i)=a,u(j)=b)\\
    0, & otherwise,
\end{cases}$$
and extend to $S_n[{\bf q}]\cup\{0\}$ by setting $\mathbf{v}\bullet_k(\mathbf{q}^{\alpha}w)=\mathbf{q}^{\alpha}(\mathbf{v}\bullet_kw)$ for $\mathbf{v}\in\mathcal{F}_n^{\mathbf{q}}\cup\{\mathbf{0}\}$. Operators $\mathbf{v}_{ab}$ with $a<b$ are called \demph{classical} (corresponding to classical covers), while those with $a>b$ are called \demph{quantum} (corresponding to quantum covers). 

We visualize a composition of operators in a diagrammatic way as follows: For $\mathbf{v}=\mathbf{v}_{a_rb_r}\cdots\mathbf{v}_{a_1b_1}\in \mathcal{F}^{\mathbf{q}}_n$, draw a rectangular window containing vertical dashed lines for each element of $\supp(\mathbf{v})$, labeled in increasing order from left to right. Now, a single classical (resp., quantum) operator $\mathbf{v}_{ab}$ is represented by a \textcolor{ForestGreen}{green} (resp., dotted \textcolor{BrickRed}{BrickRed}) horizontal line segment from the vertical line labeled $a$ to that labeled $b$. A composition of operators is depicted by stacking line segments from bottom to top as the composition is read from right to left. We illustrate $\mathbf{v}_{14}\mathbf{v}_{62}\mathbf{v}_{16}\in\mathcal{F}^{\mathbf{q}}_6$ in Figure~\ref{fig:op-quantum}.

\begin{figure}[H]
    \centering
    $$\scalebox{0.8}{\begin{tikzpicture}[scale=0.8]
        \draw[rounded corners] (-0.5, 0) rectangle (3.5, 2) {};
        \node (1) at (0,2.5) {$1$};
        \node (2) at (1,2.5) {$2$};
        \node (3) at (2,2.5)  {$4$};
        \node (4) at (3,2.5)  {$6$};
        \node (5) at (0,0.5) [circle, draw = black,fill=black, inner sep = 0.5mm] {};
        \node (6) at (3,0.5) [circle, draw = black,fill=black, inner sep = 0.5mm] {};
        \draw[ForestGreen] (5)--(6);
        \node (7) at (1,1) [circle, draw = black,fill=black, inner sep = 0.5mm] {};
        \node (8) at (3,1) [circle, draw = black,fill=black, inner sep = 0.5mm] {};
        \draw[decorate sep={0.5mm}{1mm},fill, BrickRed] (7)--(8);
        
        \node (9) at (0,1.5) [circle, draw = black,fill=black, inner sep = 0.5mm] {};
        \node (10) at (2,1.5) [circle, draw = black,fill=black, inner sep = 0.5mm] {};
        \draw[ForestGreen] (9)--(10);
        \draw[dashed,gray] (0,0)--(0,2);
        \draw[dashed,gray] (1,0)--(1,2);
        \draw[dashed,gray] (2,0)--(2,2);
        \draw[dashed,gray] (3,0)--(3,2);
    \end{tikzpicture}}$$
    \vspace{-30pt}
    \caption{Diagrammatic presentation of $\mathbf{v}_{14}\mathbf{v}_{62}\mathbf{v}_{16}\in\mathcal{F}^{\mathbf{q}}_6$}
    \label{fig:op-quantum}
\end{figure}
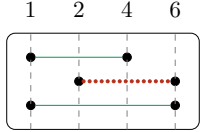

\begin{remark}\label{remark:q_k-power} 
Consider a non-zero composition of operators $\mathbf{v}=\mathbf{v}_{a_{r}b_{r}}\hdots\mathbf{v}_{a_1b_1}$. Then, there exist $n\in \mathbb{Z}_{>0}$, $u\in S_n$, and $k\in[n-1]$  such that $\mathbf{v}\bullet_ku=\mathbf{q}^{\alpha}w\neq 0$. We note the following two observations:
    \begin{enumerate}
        \item[(1)] The number of quantum operators involved in $\mathbf{v}$, i.e., the number of $i$ such that $\mathbf{v}_{a_ib_i}$ with $a_i>b_i$, equals the power of $q_k$ in $\mathbf{q}^{\alpha}$, since whenever we apply a quantum operator, the associated monomial $\mathbf{q_{ij}}$ contains $q_k$. 
        \item[(2)] The length of $\mathbf{q}^{\alpha}w$ can be computed using the information from the operators involved in $\mathbf{v}$ as
    \begin{align*}\label{eq:length with ops}
        \ell(\mathbf{q}^{\alpha}w)&=\ell(u) + r =\ell(u)+\#\{i~|~a_i>b_i,~1\le i\le r\}+\#\{i~|~a_i<b_i,~1\le i\le r\}.
    \end{align*}
    This follows from noticing that the definition of the length function is so that when we apply an operator $\mathbf{v}_{a_ib_i}$, the length increases by one, whether $\mathbf{v}_{a_ib_i}$ is classical or quantum. Therefore, it is enough to count how many operators are involved.
    \end{enumerate}
\end{remark}
We say that $\mathbf{v},\mathbf{v'}\in \mathcal{F}^{\mathbf{q}}_n$ are \demph{equivalent}, written $\mathbf{v}\equiv\mathbf{v'}$, if $\mathbf{v}\bullet_ku=\mathbf{v'}\bullet_ku$ for all $\mathbf{q}^\alpha u\in S_n[\mathbf{q}]$ and $k\in[n-1]$. We refer to  $\mathbf{v}\equiv \mathbf{0}$ as a \demph{zero equivalence} and to $\mathbf{v}\equiv \mathbf{v'} \not\equiv \mathbf{0}$ as a \demph{non-zero equivalence}.

\begin{remark}\label{rem:equivcomp}
    It is important to notice that this equivalence definition is different from the one introduced in~\cite{qkB1}. In~\cite{qkB1}, the authors say that two operators $\mathbf{v},\mathbf{v'}\in\mathcal{F}_n^\mathbf{q}$ are $(u,k)$-equivalent for $u\in S_n$ and $k\in [n-1]$ provided that $\mathbf{v}\bullet_ku=\mathbf{v'}\bullet_ku\neq 0$. Note that the equivalence $\equiv$ considered here is stronger than $(u,k)$-equivalence in the sense that if $\mathbf{v}\equiv\mathbf{v'}\not\equiv 0$, then it follows that $\mathbf{v}$ and $\mathbf{v'}$ are $(u,k)$-equivalent for all $u\in S_n$ and $k\in [n-1]$ for which $\mathbf{v}\bullet_ku\neq 0$ (equivalently $\mathbf{v'}\bullet_ku\neq 0$). As a consequence, the equivalence $\equiv$ has proven to be more challenging to analyze.
\end{remark}

As noted in the introduction, our main concern in this paper is a characterization of the set of relations $\mathbf{u}\equiv\mathbf{v}$ and $\mathbf{u}\equiv\mathbf{0}$ for $\mathbf{u},\mathbf{v}\in\mathcal{F}^{\mathbf{q}}_n$.

\subsection*{Technical results}
In this subsection, we address two initial concerns related to the equivalence relation introduced above:
\begin{itemize}
    \item[(a)] Is it possible to have two non-zero compositions of operators that are equivalent but involve a different number of operators? 
    \item[(b)] Is it possible to have two non-zero compositions of operators that are equivalent but whose support is not the same? 
\end{itemize}
The answer to both questions is \underline{no}, as shown in Lemma~\ref{lem:supp1} and Proposition~\ref{prop:supp2}, respectively. As a consequence of Proposition~\ref{prop:supp2}, we have that $\supp(\mathbf{v})$ is well-defined and independent of the choice of composition $\mathbf{v}\equiv \mathbf{v}_{a_{r}b_{r}}\hdots\mathbf{v}_{a_1b_1}$ used to define it. It is worth noting that, as we will see later via, for example, Theorem~\ref{thm:deg2zero} and Proposition~\ref{prop:min3zero1}, if one considers compositions of operators equivalent to zero, then the answers to both (a) and (b) are \underline{yes}.

We start with question (a) and consider the number of operators involved when we have a non-zero equivalence. Considering Remark~\ref{remark:q_k-power}~(2), we are immediately led to the following.

\begin{lemma}\label{lem:supp1}
    Let $\mathbf{v}=\mathbf{v}_{a_{r_1}b_{r_1}}\hdots\mathbf{v}_{a_1b_1}$ and $\mathbf{v'}= \mathbf{v}_{c_{r_2}d_{r_2}}\hdots\mathbf{v}_{c_1d_1}$ be two non-zero compositions such that $\mathbf{v}\equiv \mathbf{v'} \not\equiv \mathbf{0}$. Then $r_1=r_2$.
\end{lemma}

Next, we move to question (b) and consider the supports of equivalent non-zero compositions of operators. Unlike in the case of question (a), our argument for the answer to question (b) is not as direct.

\begin{prop}\label{prop:supp2}
    Let $\mathbf{v}=\mathbf{v}_{a_{r}b_{r}}\hdots\mathbf{v}_{a_1b_1}$ and $\mathbf{v'}= \mathbf{v}_{c_{r}d_{r}}\hdots\mathbf{v}_{c_1d_1}$ be two non-zero compositions such that $\mathbf{v}\equiv \mathbf{v'}\not\equiv \mathbf{0}$. Then, $\supp(\mathbf{v}) = \supp(\mathbf{v'})$. 
\end{prop}
\begin{proof}
    Since $\mathbf{v}\equiv \mathbf{v'}\not\equiv 0$, there exist $n\ge 3$, $u\in S_n$, $k\in[n-1]$, and  $\alpha\in\mathbb{Z}_{\geq 0}^{n-1}$ such that $\mathbf{v}\bullet_ku=\mathbf{q}^{\alpha}w= \mathbf{v'}\bullet_ku\neq 0$. By contradiction, assume that $\supp(\mathbf{v}) \neq \supp(\mathbf{v'})$. Without loss of generality, assume that there exists $c\in \supp(\mathbf{v'})$ such that $c\notin \supp(\mathbf{v})$.

    As $c\notin \supp(\mathbf{v})$, the position of $c$ in $\mathbf{v}$ is fixed by all the operators $\mathbf{v}_{a_{i}b_{i}}$. Thus, $$\text{if} \quad \mathbf{v}_{a_{i}b_{i}}\hdots\mathbf{v}_{a_1b_1}\bullet_ku=\mathbf{q^{\alpha^i}}u^i\quad  ~\text{for}~ 1\le i\le r~\text{and}~\alpha^i\in\mathbb{Z}_{\ge 0}^{n-1}\text{, \quad then}\quad u^{-1}(c)=(u^i)^{-1}(c).$$
    Moreover, setting $\mathbf{q}^{\alpha^0}=1$ and comparing the terms in $\mathbf{q}^\alpha$ as we apply the operators $\mathbf{v}_{a_ib_i}$ for $1\le i\le r$, only one of the following holds:
    \begin{itemize}
        \item $\mathbf{v}_{a_ib_i}$ is a classical operator, and so $\mathbf{q}^{\alpha^i}=\mathbf{q}^{\alpha^{i-1}}$.
        
        \item $\mathbf{v}_{a_ib_i}$ is a quantum operator with $c$ in a position strictly between the positions of the values $a_i$ and $b_i$. Then, we have that $\mathbf{q}^{\alpha^i}=\mathbf{q}_{jl}\mathbf{q}^{\alpha^{i-1}}$ with $j\le u^{-1}(c)-1<u^{-1}(c)<l$. Thus, the monomial $\mathbf{q}_{jl}$ contains a factor of $q_{\alpha_{u^{-1}(c)-1}}q_{\alpha_{u^{-1}(c)}}$.
  
        \item $\mathbf{v}_{a_ib_i}$ is a quantum operator with $c$ in a position strictly outside the positions between the values $a_i$ and $b_i$. Then, we have that $\mathbf{q}^{\alpha^i}=\mathbf{q}_{jl}\mathbf{q}^{\alpha^{i-1}}$ with $u^{-1}(c)-1<u^{-1}(c)<j$ or $l\le u^{-1}(c)-1<u^{-1}(c)$. Therefore, we conclude that the monomial $\mathbf{q}_{jl}$ does not contain a factor of either $q_{\alpha_{u^{-1}(c)-1}}$ or $q_{\alpha_{u^{-1}(c)}}$. Note that if $l=u^{-1}(c)-1$, then $q_{\alpha_{u^{-1}(c)-1}}$ does not appear in $\mathbf{q}_{jl}$ since the product goes up to $l-1$.
    \end{itemize}
     Consequently, it follows that $\alpha_{u^{-1}(c)}=\alpha_{u^{-1}(c)-1}$; note that if $u^{-1}(c)=1$, then $\alpha_1=0$. 
    
    Now, we consider $\mathbf{v'}$. Since $c\in \supp(\mathbf{v'})$ and $u^{-1}(c)=w^{-1}(c)$, it follows that $c$ must move from position $u^{-1}(c)$ and return to this same position during the operator-wise application of $\mathbf{v'}$. We show that $\alpha_{u^{-1}(c)}\neq \alpha_{u^{-1}(c)-1}$, providing a contradiction. To this end, consider the sequence of elements of $S_n[\mathbf{q}]$ obtained by sequentially applying the $\mathbf{v}_{c_id_i}$'s. For $1\le i\le r$, we write $$\mathbf{v}_{c_{i}d_{i}}\hdots\mathbf{v}_{c_1d_1}\bullet_ku=\mathbf{q}^{\beta^i}w^i,$$ 
   with $\beta^i\in \mathbb{Z}_{\ge 0}^{n-1}$. Then there must exist $1\le j_1<j_2\le r$ with $c\in \{c_{j_1},d_{j_1}\}\cap\{c_{j_2},d_{j_2}\}$ such that
    \begin{itemize}
        \item $(w^i)^{-1}(c)=u^{-1}(c)$ for $1\le i<j_1$ and $j_2\le i\le r$, and
        \item $(w^{j_1})^{-1}(c)\neq u^{-1}(c)\neq (w^{j_2-1})^{-1}(c)$.
    \end{itemize}
    Ongoing, we assume that $u^{-1}(c)\le k$, the other case following via a similar argument. We claim that there exists $j^*$ such that $j_1\le j^*\le j_2$, $c_{j^*}>d_{j^*}$, and $u^{-1}(c)=(w^{j^*-1})^{-1}(c_{j^*})$, i.e., there exists a quantum operator $\mathbf{v}_{c_{j^*}d_{j^*}}$ among $\mathbf{v}_{c_{j_1}d_{j_1}},\hdots,\mathbf{v}_{c_{j_2}d_{j_2}}$ for which $c_{j^*}$ is in position $u^{-1}(c)$ of $w^{j^*-1}$. Otherwise, for $j_1-1\le i\le j_2$, the values $w^i(u^{-1}(c))$ must be weakly increasing; but then $$w^{j_1-1}(u^{-1}(c))=c<w^{j_1}(u^{-1}(c))>c=w^{j_2}(u^{-1}(c)),$$ providing a contradiction. Now, setting $q^{\beta^{0}}=1$, note that for $1\le i\le r$, since $u^{-1}(c)\le k$, it is not possible to have $\mathbf{q}^{\beta^i}=\mathbf{q}_{jl}\mathbf{q}^{\beta^i{-1}}$ with $\mathbf{q}_{jl}$ containing a factor of $q_{u^{-1}(c)-1}$ and no factor of $q_{u^{-1}(c)}$. Thus, since $\mathbf{q}^{\beta^{j^*}}=\mathbf{q}_{j,l}\mathbf{q}^{\beta^{j^*-1}}$ with $j=u^{-1}(c)$, it follows that $\beta^{j^*}_{u^{-1}(c)-1}\neq \beta^{j^*}_{u^{-1}(c)}$ and, consequently, $\alpha_{u^{-1}(c)-1}\neq \alpha_{u^{-1}(c)}$, which is the desired contradiction, and the result follows. 
    \end{proof}

\begin{remark} 
It is important to notice that Proposition~\ref{prop:supp2} does not say that the values $c_1,d_1,\ldots, c_r,d_r$ are a reordering of the values $a_1,b_1,\ldots, a_r,b_r$. For instance, as we will see later, $\mathbf{v}_{13}\mathbf{v}_{34}\mathbf{v}_{23}\equiv \mathbf{v}_{23}\mathbf{v}_{12}\mathbf{v}_{24}$ with $$\supp(\mathbf{v}_{13}\mathbf{v}_{34}\mathbf{v}_{23})=\supp(\mathbf{v}_{23}\mathbf{v}_{12}\mathbf{v}_{24})=\{1,2,3,4\},$$ 
and the indices of $\mathbf{v}_{23}\mathbf{v}_{12}\mathbf{v}_{24}$ are not a reordering of the indices of $\mathbf{v}_{13}\mathbf{v}_{34}\mathbf{v}_{23}$. 
\end{remark}

\section{Equivalence Preserving Transformations}\label{sec:ops}

In this section, we show that certain transformations acting on compositions of operators preserve both zero and non-zero equivalence. Some of these transformations were already considered in~\cite{qkB1}. However, as explained in Remark~\ref{rem:equivcomp}, our equivalence relation is more universal, and the results presented here are in terms of it.

\subsection{Flattening}
 We start by considering the transformation on compositions of operators corresponding to an order-preserving map on the indices which takes the support of the composition to $[m]$ for some $m\in\mathbb{Z}_{>0}$. For that, we need the following auxiliary maps and notation introduced in~\cite{qkB1}. 

Given a subset $S\subseteq [n]$, with $|S|=m\leq n$, we denote by $\phi_S$ the unique order-preserving bijection from $S$ to $[m]$. 
For $s\in [n]$, define the \demph{truncation map} $\tau_s:[n]\to[n-1]\cup\{0\}$ by $$\tau_s(j)=\begin{cases}
    j, & \text{if}~j<s \\
    j-1, & \text{if}~j\ge s.
\end{cases}$$ 
Given $\mathbf{v}=\mathbf{v}_{a_rb_r}\cdots\mathbf{v}_{a_1b_1}$ with $S=\supp(\mathbf{v})$ and $p\notin S$, we define 
$$
\tau_p(\mathbf{v}_{a_rb_r}\cdots\mathbf{v}_{a_1b_1})=\mathbf{v}_{\tau_p(a_r)\tau_p(b_r)}\cdots\mathbf{v}_{\tau_p(a_1)\tau_p(b_1)}.
$$ 
Note that since $p\notin S$, the values $\tau_p(a_i)$ and $\tau_p(b_i)$ are non-zero for $1\le i\le r$, and distinct values of $S$ are mapped to distinct values by $\tau_p$. 
If $\max S=n$ and $[n]\setminus S=\{p_1<p_2<\hdots<p_j\}$, then we define the \demph{flattening} of $\mathbf{v}$ to be 
$$\underline{\mathbf{v}}=\tau_{p_1}(\tau_{p_2}(\cdots\tau_{p_j}(\mathbf{v})\cdots)).
$$

Evidently, we have the following result, which states that the flattening of $\mathbf{v}$ is the composition of operators obtained by reducing the indices to be $1,2,\ldots,|\supp(\mathbf{v})|$ in an order-preserving way.

\begin{lemma}
    Let $\mathbf{v}=\mathbf{v}_{a_rb_r}\cdots \mathbf{v}_{a_1b_1}$ with $S=\supp(\mathbf{v})$ and $|S|=m$. Then, $\supp(\underline{\mathbf{v}})=[m]$ and 
    $$\underline{\mathbf{v}}=\mathbf{v}_{\phi_S(a_r)\phi_S(b_r)}\cdots \mathbf{v}_{\phi_S(a_1)\phi_S(b_1)}.$$
\end{lemma}

\begin{Ex}
    For $\mathbf{v}=\mathbf{v}_{35}\mathbf{v}_{13}\mathbf{v}_{37}$, 
    $$\underline{\mathbf{v}}=\tau_2(\tau_4(\tau_6(\mathbf{v})))=\tau_2(\tau_4(\tau_6(\mathbf{v}_{35}\mathbf{v}_{13}\mathbf{v}_{37}))) =\tau_2(\tau_4(\mathbf{v}_{34}\mathbf{v}_{13}\mathbf{v}_{36})) =\tau_2(\mathbf{v}_{34}\mathbf{v}_{13}\mathbf{v}_{35})=\mathbf{v}_{23}\mathbf{v}_{12}\mathbf{v}_{24}.$$ 
\end{Ex}

The following result states that the flattening operation preserves zero and non-zero equivalence relations. 
\begin{theorem}~\label{thm:flat}
Let $\mathbf{v}=\mathbf{v}_{a_{r}b_{r}}\hdots\mathbf{v}_{a_1b_1}$ and $\mathbf{v}'= \mathbf{v}_{c_{r}d_{r}}\hdots\mathbf{v}_{c_1d_1}$ be two non-zero compositions such that $\supp(\mathbf{v}) = \supp(\mathbf{v}')=S$. Then,
    \begin{enumerate}
        \item[\textup{(a)}] $\mathbf{v}\equiv 0$ if and only if $\underline{\mathbf{v}'}\equiv 0$.
        \item[\textup{(b)}] $\mathbf{v}\equiv \mathbf{v}'$ if and only if $\underline{\mathbf{v}}\equiv\underline{\mathbf{v}'}$.
    \end{enumerate}
\end{theorem}

To prove Theorem~\ref{thm:flat}, we need to introduce notation that resembles the flattening operation at the level of the permutations and the sequences involved in the monomials of $\mathbf{q}$. For $u\in S_n$ and $r\in [n]$, let $\memph{u/r}\in S_{n-1}$ be the permutation obtained by omitting the $r^{\text{th}}$ value of $u$ and applying $\tau_{u(r)}$ to the remaining values. 
Similarly, for $\alpha=(\alpha_1,\alpha_2,\hdots,\alpha_{n-1})\in \mathbb{Z}_{\ge 0}^{n-1}$, define the sequence 
$$\memph{\alpha/s}=(\alpha_1,\alpha_2,\hdots\alpha_{s-1},\alpha_{s+1}\hdots,\alpha_{n-1}),$$ 
obtained by omitting the $s^{\text{th}}$ component of $\alpha$. Going in the other direction, for $r,s\in [n+1]$, $v\in S_n$, and $\alpha=(\alpha_1,\hdots,\alpha_{n-1})\in \mathbb{Z}_{\ge 0}^{n-1}$ define the permutation $\memph{\epsilon_{r,s}(v)}:=u\in S_{n+1}$ as the unique permutation such that $u(r)=s$ and $u/r=v$. Similarly, define the sequence $$\memph{\epsilon_{r,s}(\alpha)}:=(\alpha_1,\hdots,\alpha_{r-1},s,\alpha_r,\hdots,\alpha_{n-1})$$ obtained by introducing $s$ in the $r^{\text{th}}$ position.

We illustrate these notions in the following example.
\begin{Ex}
    For $u=(3,2,4,1)$, $u/2=(\tau_2(3),\tau_2(4),\tau_2(1))=(2,3,1)$ and $\epsilon_{23}(u)=(4,3,2,5,1)$. For $\alpha=(2,2,3,1)$, $\alpha/2=(2,3,1)$ and $\epsilon_{2,4}(\alpha)=(2,4,2,3,1)$.
\end{Ex}

Next, we establish a series of lemmas describing properties of $\tau_p$ and $\epsilon_{r,s}$. 
\begin{lemma}\label{lem:tauop} 
Let $a,b,p\in [n]$.
    \begin{enumerate}
        \item[\textup{(a)}] If $a<b$ and $p\neq a,b$, then $\tau_p(a)<\tau_p(b)$.
        \item[\textup{(b)}] If $\tau_p(a)<\tau_p(b)$, then $a<b$.
        \item[\textup{(c)}] For $a<b$ and $p\neq ab$, $\tau_p$ maps elements of $(a,b)\backslash\{p\}$ to elements of $(\tau_p(a),\tau_p(b))$ and elements of $[n]\backslash ([a,b]\cup \{p\})$ to elements of $[n-1]\backslash [\tau_p(a),\tau_p(b)]$.
    \end{enumerate}
\end{lemma}
\begin{proof} 
\begin{enumerate}
        \item[\textup{(a)}] We look at the three possible cases. If $p<a<b$, then $\tau_p(a)=a-1<b-1=\tau_p(b)$; if $a<p<b$, then $\tau_p(a)=a\le p-1<\tau_p(b)=b-1$; and if $a<b<p$, then $\tau_p(a)=a<b=\tau_p(b)$. In all three cases, the order is preserved. 
        \item[\textup{(b)}] We again look at the possible cases. If $\tau_p(a)<\tau_p(b)<p-1$, then $a=\tau_p(a)<\tau_p(b)=b$; while if $p-1<\tau_p(a)<\tau_p(b)$, then $a-1=\tau_p(a)<\tau_p(b)=b-1$. If $\tau_p(a)<\tau_p(b)=p-1$, then $a=\tau_p(a)<\tau_p(b)=b$ whenever $b\neq p$, and $a=\tau_p(a)<\tau_p(b)=b-1=p-1<b=p$ otherwise. Finally, if $\tau_p(a)\le p-1<\tau_p(b)$, then $a=\tau_p(a)<\tau_p(b)=b-1<b$ whenever $a\neq p$, and $a-1=p-1=\tau_p(a)<\tau_p(b)=b-1$ otherwise. Thus, in all cases, $a<b$, as desired.
        \item[\textup{(c)}] This follows immediately from part (a).
\end{enumerate}
\end{proof}
\begin{remark}\label{rk:tau_p map}
Lemma~\ref{lem:tauop} tells us that $\tau_p$ is ``almost'' order-preserving in the sense that it only breaks at $p-1$ and $p$, where these two values can become the same $\tau_p(p-1)=p-1=\tau_p(p)$. 
\end{remark}

Next, for a non-zero composition of operators $\mathbf{v}$ and $p\notin\supp(\mathbf{v})$, we establish a relationship between the actions of $\mathbf{v}$ and $\tau_p(\mathbf{v})$.

\begin{lemma}\label{lem:nztrunc}
Let $\mathbf{v}=\mathbf{v}_{a_{r}b_{r}}\hdots\mathbf{v}_{a_1b_1}$ be a non-zero composition of operators with $\supp(\mathbf{v})\subsetneq [n]$. Consider $u\in S_n$, $k\in [n-1]$, and $\alpha \in\mathbb{Z}_{\ge 0}^{n-1}$ such that $\mathbf{v}\bullet_ku=\mathbf{q}^\alpha w\in S_n[\mathbf{q}]$. Then, for all $p\notin \supp(\mathbf{v})$, we have $\tau_p(\mathbf{v})\bullet_{\tau_t(k)}u/t=\mathbf{q}^{\alpha/t}w/t$ where $t=u^{-1}(p)$. 
\end{lemma}
\begin{proof}
    We start by noticing that $w^{-1}(p)=t=u^{-1}(p)$, since $p\notin \supp(\mathbf{v})$, and that $\tau_p(k)>0$, since $\mathbf{v}\bullet_ku\neq 0$.

    We analyze three cases: (a) when $\mathbf{v}$ is a single classical operator, (b) when $\mathbf{v}$ is a single quantum operator, and (c) the general case. 
\bigskip

 \noindent
    \underline{Case (a):} Suppose $\mathbf{v}=\mathbf{v}_{ab}$ with $a<b$. Then, $i=u^{-1}(a)$ and $j= u^{-1}(b)$ satisfy $i\le k<j$. By Proposition~\ref{prop:covercond}, 
    $\{u(x)~|~i<x<j\}\subseteq [n]\backslash [a,b]$. Letting $i'=(u/t)^{-1}(\tau_p(a))$ and $j'=(u/t)^{-1}(\tau_p(b))$, we claim that $i'\le \tau_t(k)<j'$. To see this, note that
    \begin{itemize}
        \item if $t<i$, then $i'=i-1\le k-1=\tau_t(k)<j-1=j'$;
        \item if $i<t\le k$, then $i'=i\le k-1=\tau_t(k)<j-1=j'$;
        \item if $k<t<j$, then $i'=i\le k=\tau_t(k)<j-1=j'$; and
        \item if $j<t$, then $i'=i\le k=\tau_t(k)<j=j'$.
    \end{itemize}
    Moreover, by Lemma~\ref{lem:tauop}~(c),
    $\{(u/t)(x)~|~i'<x<j'\}=\{\tau_p(u(x))~|~i<x<j,~x\neq t\}\subseteq [n-1]\backslash [\tau_p(a),\tau_p(b)]$. 
    Finally, by Proposition~\ref{prop:covercond}, $\tau_p(\mathbf{v}_{ab})\bullet_{\tau_t(k)}u/t=(u/t)s_{i'j'}$, and it is straightforward to verify that $(u/t)s_{i'j'}=w/t$.
    \bigskip

    \noindent
    \underline{Case (b):} Suppose $\mathbf{v}=\mathbf{v}_{ab}$ with $a>b$. Then, $i=u^{-1}(a)$ and $j= u^{-1}(b)$ satisfy $i\le k<j$, and we denote $\mathbf{q_{ij}}=\mathbf{q}^{\alpha}$.
    By Proposition~\ref{prop:covercond}, $\{u(x)~|~i<x<j\}\subseteq (ab)$. 
    Letting $i'=(u/t)^{-1}(\tau_p(a))$ and $j'=(u/t)^{-1}(\tau_p(b))$, we have that $i'\le \tau_t(k)<j'$ by the same argument as in Case (a). Moreover, in this case, by Lemma~\ref{lem:tauop}~(c), $\{(u/t)(x)~|~i'<x<j'\}=\{\tau_p(u(x))~|~i<x<j,~x\neq t\} \subseteq (\tau_p(a),\tau_p(b))$.
    Finally, by Proposition~\ref{prop:covercond},  $\tau_p(\mathbf{v}_{ab})\bullet_{\tau_t(k)}u/t=\mathbf{q_{i'j'}}(u/t)s_{i',j'}$, and it is straightforward to verify that $(u/t)s_{i',j'}=w/t$ and $\mathbf{q_{i'j'}}=\mathbf{q}^{\alpha/t}$.
    \bigskip

    \noindent
    \underline{Case (c):} For the general case, we proceed by induction on $r$. Cases (a) and (b) cover the base case. Assume that
$$(\mathbf{v}_{a_{r-1}b_{r-1}}\hdots\mathbf{v}_{a_1b_1})\bullet_ku=\mathbf{q}^{\alpha^{r-1}}w_{r-1} \quad \text{ and } \quad \mathbf{v}_{a_{r}b_{r}}\bullet_kw_{r-1}=\mathbf{q}^{\beta^r}w_r.$$ Then, applying the induction hypothesis twice, we find that
\begin{align*}
\tau_p(\mathbf{v}_{a_{r}b_{r}}\hdots\mathbf{v}_{a_1b_1})\bullet_{\tau_t(k)}u&=\tau_p(\mathbf{v}_{a_{r}b_{r}})\bullet_{\tau_t(k)} \left(\tau_p(\mathbf{v}_{a_{r-1}b_{r-1}}\hdots\mathbf{v}_{a_1b_1})\bullet_{\tau_t(k)}u \right) \\
    &=\tau_p(\mathbf{v}_{a_{r}b_{r}})\bullet_{\tau_t(k)}{q}^{\alpha^{\textcolor{purple}{r-1}/t}}w_{r-1}/t\\
    &=\mathbf{q}^{\beta^r/t}\mathbf{q}^{\alpha^{r-1}/t}w_{r}/t\\
    &=\mathbf{q}^{\beta^r/t}\mathbf{q}^{\alpha^{r-1}/t}w/t.
\end{align*}
The result then follows by noticing that $\mathbf{q}^{\beta^r}\mathbf{q}^{\alpha^{r-1}}=\mathbf{q}^{\alpha}$, and so $\mathbf{q}^{\beta^r/t}\mathbf{q}^{\alpha^{r-1}/t}=\mathbf{q}^{\alpha/t}$. \qedhere
\end{proof}

\begin{Ex}
    Consider $\mathbf{v}=\mathbf{v}_{61}\mathbf{v}_{78}\mathbf{v}_{53}$ with $\supp(\mathbf{v})\subsetneq [8]$, $u=(7,6,2,5,3,4,1,8)\in S_8$, and $k=4\in[7]$. Then, 
$\mathbf{v}_{61}\mathbf{v}_{78}\mathbf{v}_{53}\bullet_k u=\mathbf{q}^{\alpha}w$ with $w=(8,1,2,3,5,4,6,7)$ and $\alpha=(0,1,1,2,1,1)$. 

Taking $p=2\notin\supp(\mathbf{v})$ and $t=u^{-1}(p)=3$, we have that $w/t=(7,1,2,4,3,5,6)$ and $\alpha/t=(0,1,2,1,1)$. The reader can check that, as stated in the result,
$$\tau_2(\mathbf{v}_{61}\mathbf{v}_{78}\mathbf{v}_{53})\bullet_{\tau_3(4)} u/3=\mathbf{q}^{\alpha/t}w/t.$$
\end{Ex}

\subsubsection*{Zero equivalences}
In this subsection, we establish part (a) of Theorem~\ref{thm:flat}. To do so, we require the following proposition.

\begin{prop}\label{prop:tauzequiv}
    Let $\mathbf{v}=\mathbf{v}_{a_{r}b_{r}}\hdots\mathbf{v}_{a_1b_1}$ and let $p\notin\supp(\mathbf{v})$. Then, $\mathbf{v}\equiv \mathbf{0}$ if and only if $\tau_p(\mathbf{v})\equiv \mathbf{0}$.
\end{prop}
\begin{proof}
\Ovalbox{$\Longrightarrow$} By a contrapositive argument, suppose $\tau_p(\mathbf{v})\not\equiv \mathbf{0}$. Then there exist $n\in\mathbb{Z}_{>0}$, $u\in S_n$, and $k\in [n-1]$ such that $\tau_p(\mathbf{v})\bullet_ku\neq 0$. One can verify directly that $\mathbf{v}\bullet_k\epsilon_{n+1,p}(u)\neq 0$, and so $\mathbf{v}\not\equiv \mathbf{0}$.

\Ovalbox{$\Longleftarrow$} We consider the contrapositive again.  If $\mathbf{v}\not\equiv \mathbf{0}$, then there exist $n\in\mathbb{Z}_{>0}$, $u\in S_n$, and $k\in [n-1]$ such that $\mathbf{v}\bullet_ku\neq 0$. Assume that $\mathbf{v}\bullet_ku=\mathbf{q}^{\alpha}w$ for some $\alpha\in\mathbb{Z}_{\ge 0}^{n-1}$ and $w\in S_n$. Applying Lemma~\ref{lem:nztrunc}, it follows that $\tau_p(\mathbf{v})\bullet_{\tau_t(k)}u/t=\mathbf{q}^{\alpha/t}w/t\neq 0$, and so $\tau_p(\mathbf{v})\not\equiv \mathbf{0}$.
\end{proof}

We now prove Theorem~\ref{thm:flat}~(a).

\begin{proof}[Proof of Theorem~\ref{thm:flat}~$(a)$] Let $S=\supp(\mathbf{v})$ with $|S|=m$, $n=\max S$, and $N=n-m$. We prove the result by induction on $N$. For $N=0$, $\mathbf{v}=\underline{\mathbf{v}}$ and the result follows immediately. Proposition~\ref{prop:tauzequiv} covers the case $N=1$. Assume then that the result holds for $N-1\ge 1$. Let $[n]\backslash S=\{p_1<\hdots<p_N\}$ and $\mathbf{\widehat{v}}=\tau_{p_N}(\mathbf{v})$. Note that $\widehat{S}=\supp(\mathbf{\widehat{v}})$ satisfies $|\widehat{S}|=m$, $\max \widehat{S}=n-1$, and $[n-1]\backslash \widehat{S}=\{p_1<\hdots<p_{N-1}\}$. Thus,
\begin{itemize}
    \item $\mathbf{v}\equiv \mathbf{0}$ if and only if $\mathbf{\widehat{v}}\equiv \mathbf{0}$ by Proposition~\ref{prop:tauzequiv} and
    \item $\mathbf{\widehat{v}}\equiv \mathbf{0}$ if and only if $\underline{\mathbf{\widehat{v}}}=\underline{\mathbf{v}}\equiv \mathbf{0}$ by our induction hypothesis.
\end{itemize}
The result follows.
\end{proof}

\subsubsection*{Non-zero equivalences}
In this subsection, we finish the proof of Theorem~\ref{thm:flat} by establishing part (b). To do so, we require a result analogous to Proposition~\ref{prop:tauzequiv}.

\begin{theorem}\label{thm:taunzequiv}
    Let $\mathbf{v}=\mathbf{v}_{a_{r}b_{r}}\hdots\mathbf{v}_{a_1b_1}$ and $\mathbf{v}'=\mathbf{v}_{c_{r}d_{r}}\hdots\mathbf{v}_{c_1d_1}$ with $p\notin \supp(\mathbf{v})=\supp(\mathbf{v}')$. Then $\mathbf{v}\equiv \mathbf{v}'$ if and only if $\tau_p(\mathbf{v})\equiv \tau_p(\mathbf{v}')$.
\end{theorem}

The proof of this result is more elaborate than that of Proposition~\ref{prop:tauzequiv}. Consequently, we outline the approach below.
\begin{itemize}
    \item[] (Step 1) We introduce the notion of $L_{<p,t}(u)$ for a given permutation $u\in S_n$ and $p,t\in [n]$ which enumerates the entries in $u$ to the right of position $t$ whose values are strictly smaller than $p$. In addition, we establish useful properties of $L_{<p,t}(u)$ (Lemma~\ref{lem:L1}).
    \item[] (Step 2)  We prove one technical result identifying cases where knowing that $\tau_p(\mathbf{v})=0$ for a particular choice of $u$ and $k$ implies that $\mathbf{v}$ is also zero for those $u$ and $k$ (Proposition~\ref{prop:tau0im0}).
    \item[] (Step 3) We prove that $\tau_p(\mathbf{v})\equiv \tau_p(\mathbf{v}') \ \Longrightarrow \ \mathbf{v} \equiv \mathbf{v}'$ (Proposition~\ref{prop:taunzequiv1}).
    \item[] (Step 4) We prove that $\mathbf{v} \equiv \mathbf{v}' \ \Longrightarrow \ \tau_p(\mathbf{v})\equiv \tau_p(\mathbf{v}')$ (Proposition~\ref{prop:taunzequiv2}).
\end{itemize}

Jumping into (Step 1), given a permutation $u\in S_n$ and $p\in [n]$, let $t=u^{-1}(p)$ and define
$$\memph{L_{<p,t}(u)}:=\#\{i~|~i<t,~u(i)<p\}.$$ 
The following result summarizes some of the properties of $L_{<p,t}(u)$.
\begin{lemma}\label{lem:L1}
    For $n\in \mathbb{Z}_{>1}$, let $u\in S_n$, $k\in [n-1]$, and $\mathbf{v}=\mathbf{v}_{a_{r}b_{r}}\hdots\mathbf{v}_{a_1b_1}$ with $\supp(v)\subsetneq [n]$. Suppose that $p\notin \supp(\mathbf{v})$, $u^{-1}(p)=t$, and $\tau_p(\mathbf{v})\bullet_{\tau_t(k)}u/t\neq 0$. Letting $u_0=u$, for $i\in[r]$, define $$\tau_p(\mathbf{v}_{a_{i}b_{i}}\hdots\mathbf{v}_{a_1b_1})\bullet_{\tau_t(k)}u/t = \mathbf{q}^{\beta^i}w_i$$ 
    where $w_i\in S_{n-1}$, and $u_i=\epsilon_{t,p}(w_i)$.
    \begin{enumerate}
        \item[\textup{(a)}] For $i\in[r]$, $L_{<p,t}(u_i/t)=L_{<p,t}(u_{i-1}/t)+1$ if and only if $u_{i-1}^{-1}(a_i)<t<u_{i-1}^{-1}(b_i)$ and $a_i>p>b_i$.
        
        \item[\textup{(b)}] For $i\in[r]$, $L_{<p,t}(u_i/t)=L_{<p,t}(u_{i-1}/t)-1$ if and only if $u_{i-1}^{-1}(a_i)<t<u_{i-1}^{-1}(b_i)$ and $a_i<p<b_i$.
        
        \item[\textup{(c)}] $L_{<p,t}(w/t)-L_{<p,t}(u/t)\le\#\{i\in [r]~|~u_{i-1}^{-1}(a_i)<t<u_{i-1}^{-1}(b_i)~\text{and}~a_i>p>b_i\}$, with equality if and only if there exists no $j\in [r]$ such that $u^{-1}_{j-1}(a_j)<t<u^{-1}_{j-1}(b_j)$ and $a_j<p<b_j$.

        \item[\textup{(d)}] If  $\mathbf{v}\bullet_ku=\mathbf{q}^{\alpha}w\in S_n[\mathbf{q}]$ with $\alpha=(\alpha_1,\hdots,\alpha_{n-1})\in\mathbb{Z}_{\ge 0}^{n-1}$ and $w\in S_n$, then $$\alpha_{t-1}=L_{<p,t}(w)-L_{<p,t}(u)=L_{<p,t}(w/t)-L_{<p,t}(u/t).$$ 
    \end{enumerate}
\end{lemma}
\begin{proof}
    \begin{enumerate}
        \item[\textup{(a)}] For $i\in [r]$, note that $L_{<p,t}(u_i/t)=L_{<p,t}(u_{i-1}/t)+1$ if and only if 
        \begin{equation}\label{eq:Lpt1}
            \tau_p(b_i)<p\le \tau_p(a_i)~\text{and}~(u_{i-1}/t)^{-1}(\tau_p(a_i))<t\le (u_{i-1}/t)^{-1}(\tau_p(b_i)).
        \end{equation}
        Since $p\notin\supp(\mathbf{v})$, (\ref{eq:Lpt1}) holds if and only if $b_i<p< a_i~\text{and}~u_{i-1}^{-1}(a_i)<t< u_{i-1}^{-1}(b_i)$, and the result follows.
        
        \item[\textup{(b)}] This follows similarly to part (a) with $\tau_p(b_i)\ge p>\tau_p(a_i)$ in place of $\tau_p(b_i)<p\le \tau_p(a_i)$.
        
        \item[\textup{(c)}] Since $|L_{<p,t}(u_i/t)-L_{<p,t}(u_{i-1}/t)|\le 1$ for $i\in [r]$, this property follows from parts (a) and (b).

        \item[\textup{(d)}] The second equality follows from the fact that $L_{<p,t}(w)=L_{<p,t}(w/t)$ and $L_{<p,t}(u)=L_{<p,t}(u/t)$. For the first equality, letting $u'_0=u$, for $i\in[r]$, define
        $$q^{\alpha^i}u'_i=(\mathbf{v}_{a_{i}b_{i}}\hdots\mathbf{v}_{a_1b_1})\bullet_ku$$ 
        where $u'_i\in S_n$. Note that $u'_i=u_i$ for $i\in [r]$ and $$\alpha_{t-1}=\#\{j\in[r]~|~a_j>b_j~\text{and}~u_{j-1}^{-1}(a_j)<t<u_{j-1}^{-1}(b_j)\}.$$ Since $u_i(t)=p$ for all $0\le i\le r$ and $\mathbf{v}\bullet_ku\neq0$, by Proposition~\ref{prop:covercond},
        $$\#\{j\in[r]~|~a_j>b_j~\text{and}~u_{j-1}^{-1}(a_j)<t<u_{j-1}^{-1}(b_j)\}=\#\{j\in[r]~|~a_j>p>b_j~\text{and}~u_{j-1}^{-1}(a_j)<t<u_{j-1}^{-1}(b_j)\}$$ 
        and there exists no $j\in [r]$ for which $a_j<p<b_j$ and $u_{j-1}^{-1}(a_j)<t<u_{j-1}^{-1}(b_j)$. Thus, by part (c),
        $$L_{<p,t}(w/t)-L_{<p,t}(u/t)=\#\{i\in [r]~|~u_{i-1}^{-1}(a_i)<t<u_{i-1}^{-1}(b_i)~\text{and}~a_i>p>b_i\}=\alpha_{t-1},$$
        as desired. \qedhere
    \end{enumerate} 
\end{proof}

Moving to (Step 2), we establish Proposition~\ref{prop:tau0im0} below, which will prove to be useful later in the paper. We start with the case of a single operator, which we treat separately in Lemma~\ref{lem:0tau1s} below.

\begin{lemma}\label{lem:0tau1s}
    Consider $\mathbf{v}_{ab}$ with $a,b\in [n]$ and $a\neq b$, and let $u\in S_n$. Let $p\in [n]\setminus \{a,b\}$ and $t=u^{-1}(p)$. If $\tau_p(\mathbf{v}_{ab})\bullet_{\tau_t(k)}u/t=0$ for some $k\in [n-1]$, then $\mathbf{v}_{ab}\bullet_ku=0$.
\end{lemma}

\begin{proof}
    Let $i=u^{-1}(a)$, $j=u^{-1}(b)$, $i'=(u/t)^{-1}(\tau_p(a))=\tau_t(i)$, and $j'=(u/t)^{-1}(\tau_p(b))=\tau_t(j)$. We analyze the possible cases, grouping them into two types based on the associated arguments.
    \bigskip

    \noindent
    \underline{Case (a):} $j'<i'$, $\tau_t(k)<i'<j'$, or $i'<j'\le \tau_t(k)$. For $j'<i'$, $\tau_t(k)<i'<j'$, or $i'<j'<\tau_t(k)$, applying Lemma~\ref{lem:tauop}~(b), we have that $j<i$, $k<i<j$, or $i<j<k$, respectively. Thus, $\mathbf{v}_{ab}\bullet_ku=0$. Now, for $i'<j'=\tau_t(k)$, since $u^{-1}(p)=t\neq u^{-1}(a),u^{-1}(b)$, applying Lemma~\ref{lem:tauop}~(b), either $i<j=k$ if $k\neq t$ or $i<j=t-1<k=t$ otherwise. In either case, $\mathbf{v}_{ab}\bullet_ku=0$, as desired.
    \bigskip

    \noindent
    \underline{Case (b):} $i'\le \tau_t(k)<j'$, and $a<b$ or $a>b$. We consider the case $a<b$ as the other one follows similarly.
    Note that since $u^{-1}(p)=t\neq u^{-1}(a),u^{-1}(b)$, applying Lemma~\ref{lem:tauop}~(b), either 
    \begin{itemize}
        \item $i'<\tau_t(k)<j'$ so that $i<k<j$,
        \item $k\neq t$ and $i'=\tau_t(k)<j'$ so that $i=k<j$, or
        \item $k=t$ and $i'=\tau_t(k)<j'$ so that $i=t-1<t=k<j$.
    \end{itemize}
    Thus, $i\le k<j$. Now, applying Lemma~\ref{lem:tauop}~(a), we have that $\tau_p(a)<\tau_p(b)$. Consequently, applying Proposition~\ref{prop:covercond}, there exists $l'\in(i',j')$ such that $(u/t)(l')\in (\tau_p(a),\tau_p(b))$. Considering Lemma~\ref{lem:tauop}~(b), this implies that there exists $i<l<j$ such that $\tau_t(l)=l'$ and $u(l)\in (a,b)$. Hence, applying Proposition~\ref{prop:covercond} once again, it follows that $\mathbf{v}_{ab}\bullet_ku=0$, as desired. \qedhere
\end{proof}
Now, we state and prove the result corresponding to (Step 2).
\begin{prop}\label{prop:tau0im0}
    Let $\mathbf{v}=\mathbf{v}_{a_{r}b_{r}}\hdots\mathbf{v}_{a_1b_1}$ and let $p\notin \supp(\mathbf{v})$. Suppose that $u\in S_n$ and $k\in [n-1]$ are such that $\tau_p(\mathbf{v})\bullet_{\tau_t(k)}u/t=0$ where $t=u^{-1}(p)$. Then $\mathbf{v}\bullet_ku=0$.
\end{prop}
\begin{proof}
We define $i^*\in[r]$ to be the least integer for which $\tau_p(\mathbf{v}_{a_{i^*}b_{i^*}}\cdots\mathbf{v}_{a_1b_1})\bullet_{\tau_t(k)}u/t=0$. If $i^*=1$, then the result follows upon applying Lemma~\ref{lem:0tau1s}. So, assume that $i^*>1$. In this case, $\tau_p(\mathbf{v}_{a_{i^*-1}b_{i^*-1}}\cdots\mathbf{v}_{a_1b_1})\bullet_{\tau_t(k)}u/t\neq 0$ and $\tau_p(\mathbf{v}_{a_{i^*}b_{i^*}}\cdots\mathbf{v}_{a_1b_1})\bullet_{\tau_t(k)}u/t=0$. Note that if $i^*>1$ and $\mathbf{v}_{a_{i^*-1}b_{i^*-1}}\cdots\mathbf{v}_{a_1b_1}\bullet_ku=0$, then the result follows directly. Consequently, we also assume that $\mathbf{v}_{a_{i^*-1}b_{i^*-1}}\cdots\mathbf{v}_{a_1b_1}\bullet_ku=\mathbf{q}^{\alpha}w\neq 0$ for some $\alpha\in\mathbb{Z}_{\ge 0}^{n-1}$ and $w\in S_n$. Applying Lemma~\ref{lem:nztrunc}, we have that  $\tau_p(\mathbf{v}_{a_{i^*-1}b_{i^*-1}}\cdots\mathbf{v}_{a_1b_1})\bullet_{\tau_t(k)}u/t= \mathbf{q}^{\alpha/t}w/t$ and
    \begin{align*}
0&=\tau_p(\mathbf{v}_{a_{i^*}b_{i^*}}\cdots\mathbf{v}_{a_1b_1})\bullet_{\tau_t(k)}u/t
=\tau_p(\mathbf{v}_{a_{i^*}b_{i^*}})\bullet_{\tau_t(k)}\left(\tau_p(\mathbf{v}_{a_{i^*-1}b_{i^*-1}}\cdots\mathbf{v}_{a_1b_1})\bullet_{\tau_t(k)}u/t\right) \\   &=\tau_p(\mathbf{v}_{a_{i^*}b_{i^*}})\bullet_{\tau_t(k)}\mathbf{q}^{\alpha/t}w/t=\mathbf{q}^{\alpha/t}\tau_p(\mathbf{v}_{a_{i^*}b_{i^*}})\bullet_{\tau_t(k)}w/t;
    \end{align*}
    that is, $\tau_p(\mathbf{v}_{a_{i^*}b_{i^*}})\bullet_{\tau_t(k)}w/t=0$. Thus, applying Lemma~\ref{lem:0tau1s} once again, it follows that $$\mathbf{v}\bullet_ku=\mathbf{v}_{a_{i^*}b_{i^*}}\cdots\mathbf{v}_{a_1b_1}\bullet_ku=\mathbf{v}_{a_{i^*}b_{i^*}}\bullet_kw=0,$$ as desired.
\end{proof}

Next, for (Step 3), we show that $\tau_p(\mathbf{v})\equiv \tau_p(\mathbf{v'})$ implies $\mathbf{v}\equiv \mathbf{v'}$.

\begin{prop}\label{prop:taunzequiv1}
    Let $\mathbf{v}=\mathbf{v}_{a_{r}b_{r}}\hdots\mathbf{v}_{a_1b_1}$ and $\mathbf{v}'=\mathbf{v}_{c_{r}d_{r}}\hdots\mathbf{v}_{c_1d_1}$ with $S=\supp(\mathbf{v})=\supp(\mathbf{v}')$ and take $p\notin S$. If $\tau_p(\mathbf{v})\equiv \tau_p(\mathbf{v}')$, then $\mathbf{v}\equiv \mathbf{v}'$.
\end{prop}
\begin{proof}
    For $n\ge \max(S)$, consider $u\in S_n$ and $k\in [n-1]$, and let $t=u^{-1}(p)$. Then, it suffices to show that $\mathbf{v}\bullet_ku=\mathbf{v}'\bullet_ku$. 
    
    If $\tau_p(\mathbf{v})\bullet_{\tau_t(k)}u/t=\tau_p(\mathbf{v}')\bullet_{\tau_t(k)}u/t=0$, then, by Proposition~\ref{prop:tau0im0}, we have that $\mathbf{v}\bullet_ku=\mathbf{v}'\bullet_ku=0$.
    Thus, two cases remain to consider.
    \bigskip

    \noindent
 \underline{Case (a):}  $\tau_p(\mathbf{v})\bullet_{\tau_t(k)}u/t=\tau_p(\mathbf{v}')\bullet_{\tau_t(k)}u/t\in S_{n-1}[\mathbf{q}]$ and $\mathbf{v}\bullet_ku,\mathbf{v}'\bullet_ku\neq 0$. Let $\mathbf{v}\bullet_ku=\mathbf{q}^{\alpha}w_1$ and $\mathbf{v}'\bullet_ku=\mathbf{q}^{\beta}w_2$ with $\alpha=(\alpha_1,\hdots,\alpha_{n-1}),\beta=(\beta_1,\hdots,\beta_{n-1})\in \mathbb{Z}_{\ge 0}^{n-1}$. Applying Lemma~\ref{lem:nztrunc}, it follows that $\tau_p(\mathbf{v})\bullet_{\tau_t(k)}u/t=\mathbf{q^{\alpha/t}}w_1/t=\mathbf{q}^{\beta/t}w_2/t=\tau_p(\mathbf{v}')\bullet_{\tau_t(k)}u/t$ which implies that $w_1=\epsilon_{t,p}(w_1/t)=\epsilon_{t,p}(w_2/t)=w_2$ and $\alpha^i=\beta^i$ for $i\neq t$. We claim that $\alpha_t=\alpha_{t-1}$ and $\beta_t=\beta_{t-1}$, from which it follows that $\alpha=\beta$. As the arguments are the same, we establish the claim for $\alpha_t$. Letting $u_0=u$ and $\alpha_0=(0,\hdots,0)\in\mathbb{Z}_{\ge 0}^{n-1}$, for $i\in [r]$, we define $$\mathbf{q}^{\alpha^i}u_i=(\mathbf{v}_{a_{i}b_{i}}\hdots\mathbf{v}_{a_1b_1})\bullet_ku$$ with $u_i\in S_n$. Note that if $a_i>b_i$ for some $i\in [r]$ with $u_{i-1}^{-1}(a_i)=i^*$ and $u_{i-1}^{-1}(b_i)=j^*$, then, since $p\notin S$, either
    \begin{itemize}
        \item $t<i^*<j^*$ or $i^*<j^*<t$, in which cases neither $q_{t-1}$ nor $q_t$ are factors of $\dfrac{\mathbf{q}^{\alpha^i}}{\mathbf{q}^{\alpha_{i-1}}}$, or
        \item $i^*<t<j^*$, in which case both $q_{t-1}$ and $q_t$ are factors of $\dfrac{\mathbf{q}^{\alpha^i}}{\mathbf{q}^{\alpha_{i-1}}}$.
    \end{itemize}
    Thus, the claim follows. Consequently, $\mathbf{v}\bullet_ku=\mathbf{v}'\bullet_ku$, as desired.
        \bigskip

    \noindent
\underline{Case (b):}  $\tau_p(\mathbf{v})\bullet_{\tau_t(k)}u/t=\tau_p(\mathbf{v}')\bullet_{\tau_t(k)}u/t=q^{\alpha/t}w/t\in S_{n-1}[\mathbf{q}]$ for $\alpha=(\alpha_1,\hdots,\alpha_{n-1})\in \mathbb{Z}_{\ge 0}^{n-1}$ and $w\in S_n$, and either  $\mathbf{v}\bullet_ku=0$ or $\mathbf{v}'\bullet_ku=0$. We claim that $\mathbf{v}\bullet_ku=\mathbf{v}'\bullet_ku=0$. Without loss of generality, assume that $\mathbf{v}'\bullet_ku=0$. For a contradiction, assume that $\mathbf{v}\bullet_ku=\mathbf{q}^{\alpha}w\in S_n[\mathbf{q}]$. Applying Lemma~\ref{lem:L1}~(d), it follows that 
    \begin{equation}\label{eq:1}
        \alpha_{t-1}=L_{<p,t}(w)-L_{<p,t}(u)=L_{<p,t}(w/t)-L_{<p,t}(u/t).
    \end{equation}
    Now, define $i^*\in [r]$ to be the least integer for which $\mathbf{v}_{c_{i^*}d_{i^*}}\cdots\mathbf{v}_{c_1d_1}\bullet_ku=0$. In particular, $i^*=1$ if $\mathbf{v}_{c_1d_1}\bullet_ku=0$; otherwise, $i^*>1$ is such that $\mathbf{v}_{c_{i^*-1}d_{i^*-1}}\hdots\mathbf{v}_{c_1d_1}\bullet_ku=\mathbf{q}^{\beta}u_{i^*-1}\in S_n[\mathbf{q}]$ and $\mathbf{v}_{c_{i^*}d_{i^*}}\hdots\mathbf{v}_{c_1d_1}\bullet_ku=0$. There are two cases, and in both, letting $u_0=u$, for $j\in[r]$, we define
$$\tau_p(\mathbf{v}_{a_{j}b_{j}}\hdots\mathbf{v}_{a_1b_1})\bullet_{\tau_t(k)}u/t = \mathbf{q}^{\beta^j}w_j$$ 
     where $w_j\in S_{n-1}$, and $u_j=\epsilon_{t,p}(w_j)$.

    \begin{itemize}
\item[]    \underline{Case (b.1):} $c_{i^*}>d_{i^*}$. Since $\tau_p(\mathbf{v}_{c_{i^*}d_{i^*}}\hdots\mathbf{v}_{c_1d_1})\bullet_{\tau_t(k)}u/t\neq 0$, it follows that $$\mathbf{v}_{c_{i^*}d_{i^*}}\bullet_k(\mathbf{v}_{c_{i^*-1}d_{i^*-1}}\hdots\mathbf{v}_{c_1d_1}\bullet_ku)=0$$ as a result of the addition of $p$ at position $t$ in $u$. Consequently, considering Proposition~\ref{prop:covercond}, we must have that $(u_{i^*-1})^{-1}(c_{i^*})<t<(u_{i^*-1})^{-1}(d_{i^*})$ and either  $p>c_{i^*}$ or $p<d_{i^*}$. Note that, in either case, applying Lemma~\ref{lem:L1}~(c), we have that 
    \begin{align*}
        \alpha_{t-1}&=\#\{j\in [r]~|~\tau_p(c_j)>\tau_p(d_j),~(u_{j-1}/t)^{-1}(\tau_p(c_j))<t\le (u_{j-1}/t)^{-1}(\tau_p(d_j))\}\\
        &=\#\{j\in [r]~|~c_j>d_j,~u^{-1}_{j-1}(c_j)<t<u^{-1}_{j-1}(d_j)\}\\
        &\ge \#\{j\in [r]~|~c_j>p>d_j,~u^{-1}_{j-1}(c_j)<t<u^{-1}_{j-1}(d_j)\}+1\\
        &>L_{<p,t}(w/t)-L_{<p,t}(u/t),
    \end{align*}
    where we are thinking of $\alpha$ as that in $\tau_p(\mathbf{v}')\bullet_{\tau_t(k)}u/t=\mathbf{q}^{\alpha/t}w/t$. Considering (\ref{eq:1}), this is a contradiction.

\item[]    \underline{Case (b.2):} $c_i^*<d_i^*$. As a consequence of the addition of $p$ at position $t$ in $u$, we have that 
$$\mathbf{v}_{c_{i^*}d_{i^*}}\bullet_k(\mathbf{v}_{c_{i^*-1}d_{i^*-1}}\hdots\mathbf{v}_{c_1d_1}\bullet_ku)=0.$$
Therefore, by Proposition~\ref{prop:covercond}, we must have that $(u_{i^*-1})^{-1}(c_{i^*})<t<(u_{i^*-1})^{-1}(d_{i^*})$ and $c_{i^*}<p<d_{i^*}$. Thus, applying Lemma~\ref{lem:L1}~(c), it follows that 
    \begin{align*}
        \alpha_{t-1}&=\#\{j\in [r]~|~\tau_p(c_j)>\tau_p(d_j),~(u_{j-1}/t)^{-1}(\tau_p(c_j))<t\le (u_{j-1}/t)^{-1}(\tau_p(d_j))\}\\
        &=\#\{j\in [r]~|~c_j>d_j,~u^{-1}_{j-1}(c_j)<t<u^{-1}_{j-1}(d_j)\}\\
        &\ge \#\{j\in [r]~|~c_j>p>d_j,~u^{-1}_{j-1}(c_j)<t<u^{-1}_{j-1}(d_j)\}\\
        &>L_{<p,t}(w/t)-L_{<p,t}(u/t).
    \end{align*}
As before, considering (\ref{eq:1}), this is a contradiction. \qedhere
    \end{itemize}
    
\end{proof}

Finally, for (Step 4), we show that $\mathbf{v}\equiv \mathbf{v}'$ implies $\tau_p(\mathbf{v})\equiv \tau_p(\mathbf{v}')$.

\begin{prop}\label{prop:taunzequiv2}
    Let $\mathbf{v}=\mathbf{v}_{a_{r}b_{r}}\hdots\mathbf{v}_{a_1b_1}$ and $\mathbf{v}'=\mathbf{v}_{c_{r}d_{r}}\hdots\mathbf{v}_{c_1d_1}$ with $S=\supp(\mathbf{v})=\supp(\mathbf{v}')$, and consider $p\notin S$. If $\mathbf{v}\equiv \mathbf{v}'$, then $\tau_p(\mathbf{v})\equiv \tau_p(\mathbf{v}')$.
\end{prop}
\begin{proof}
    The result is immediate for $p>\max S$, and so we assume that $p<\max S$. For $n\ge \max S$, consider $u\in S_{n-1}$ and $k\in [n-2]$, and define $u^*=\epsilon_{n,p}(u)$. If $\mathbf{v}\bullet_ku^*=\mathbf{v}'\bullet_ku^*=\mathbf{q}^{\alpha}w\neq 0$, then, applying Lemma~\ref{lem:nztrunc}, we have that $$\tau_p(\mathbf{v})\bullet_ku=\tau_p(\mathbf{v})\bullet_{\tau_n(k)}u^*/n=\mathbf{q}^{\alpha/n}w/n=\tau_p(\mathbf{v}')\bullet_{\tau_n(k)}u^*/n=\tau_p(\mathbf{v}')\bullet_ku.$$ On the other hand, if $\mathbf{v}\bullet_ku^*=\mathbf{v}'\bullet_ku^*=0$, then we claim that $\tau_p(\mathbf{v})\bullet_ku=\tau_p(\mathbf{v}')\bullet_ku=0$. We show that $\tau_p(\mathbf{v})\bullet_ku=0$ as $\tau_p(\mathbf{v}')\bullet_ku=0$ follows by the same argument. 
    
    Define $i^*\in [r]$ to be the least integer for which $\mathbf{v}_{a_{i^*}b_{i^*}}\cdots\mathbf{v}_{a_1b_1}\bullet_ku^*=0$. In particular, $i^*=1$ if $\mathbf{v}_{a_1b_1}\bullet_ku^*=0$; otherwise, $i^*>1$ is such that $\mathbf{v}_{a_{i^*-1}b_{i^*-1}}\hdots\mathbf{v}_{a_1b_1}\bullet_ku^*=\mathbf{q}^{\beta}w\in S_n[\mathbf{q}]$ and $\mathbf{v}_{a_{i^*}b_{i^*}}\hdots\mathbf{v}_{a_1b_1}\bullet_ku^*=0$. Applying Lemma~\ref{lem:nztrunc}, we have that $\tau_p(\mathbf{v}_{a_{i^*-1}b_{i^*-1}}\hdots\mathbf{v}_{a_1b_1})\bullet_{k}u=\mathbf{q}^{\beta/n}w/n\in S_{n-1}[\mathbf{q}]$. We analyze the possible cases, grouping them into two types based on the associated arguments.
    \bigskip

    \noindent
    \underline{Case (a):} $w^{-1}(a_{i^*})>w^{-1}(b_{i^*})$, $w^{-1}(a_{i^*})<w^{-1}(b_{i^*})\le k$, or $k<w^{-1}(a_{i^*})<w^{-1}(b_{i^*})$. Since, $w^{-1}(a_{i^*})=(w/n)^{-1}(\tau_p(a_{i^*}))$ and $w^{-1}(b_{i^*})=(w/n)^{-1}(\tau_p(b_{i^*}))$, we must have that $\tau_p(\mathbf{v}_{a_{i^*}b_{i^*}})\bullet_kw/n=0$, i.e., $\tau_p(\mathbf{v})\bullet_ku=0$.
    \bigskip

    \noindent
    \underline{Case (b):} $w^{-1}(a_{i^*})\le k<w^{-1}(b_{i^*})$, and $a_{i^*}<b_{i^*}$ or $a_{i^*}>b_{i^*}$. We consider the case $a_{i^*}<b_{i^*}$ as the other case follows similarly. 
    Considering Proposition~\ref{prop:covercond}, there exists $l\in (w^{-1}(a_{i^*}),w^{-1}(b_{i^*}))$ such that $c=w(l)\in (a_{i^*},b_{i^*})$. Since $p\neq a_{i^*},b_{i^*}, c$ by construction of $u^*$, applying Lemma~\ref{lem:tauop}~(a), we have that $\tau_p(a_{i^*})<\tau_p(c)<\tau_p(b_{i^*})$. Thus, $$w^{-1}(a_{i^*})=\tau_n(w^{-1}(a_{i^*}))=(w/n)^{-1}(\tau_p(a_{i^*}))<l=\tau_n(l)<w^{-1}(b_{i^*})=\tau_n(w^{-1}(b_{i^*}))=(w/n)^{-1}(\tau_p(b_{i^*}))$$ and $\tau_p(c)=\tau_p(w(l))\in (\tau_p(a_{i^*}),\tau_p(b_{i^*}))$; that is, $(w/n)^{-1}(\tau_p(a_{i^*}))<l<(w/n)^{-1}(\tau_p(b_{i^*}))$ and $(w/n)(l)\in (\tau_p(a_{i^*}),\tau_p(b_{i^*}))$. Consequently, by Proposition~\ref{prop:covercond}, $\tau_p(\mathbf{v}_{a_{i^*}b_{i^*}})\bullet_kw/n=0$, i.e., $\tau_p(\mathbf{v})\bullet_ku=0$.\qedhere

\end{proof}

All steps complete, Theorem~\ref{thm:taunzequiv} follows. We can now finish the proof of Theorem~\ref{thm:flat}.

\begin{proof}[Proof of Theorem~\ref{thm:flat}~$(b)$]
Let $S=\supp(\mathbf{v})=\supp(\mathbf{v}')$ with $|S|=m$, $n=\max S$, and $N=n-m$. We proceed by induction on $N$. For $N=0$, $\mathbf{v}=\underline{\mathbf{v}}$ and $\mathbf{v}'=\underline{\mathbf{v}'}$, so that the result follows immediately. Theorem~\ref{thm:taunzequiv} covers the case $N=1$.
Assume then that the result holds for $N-1\ge 1$. Let $[n]\backslash S=\{p_1<\hdots<p_N\}$, $\mathbf{\widehat{v}_1}=\tau_{p_N}(\mathbf{v})$, and $\mathbf{\widehat{v}_2}=\tau_{p_N}(\mathbf{v}')$. Note that $\widehat{S}=\supp(\tau_{p_N}(\mathbf{v}))=\supp(\tau_{p_N}(\mathbf{v}'))$ satisfies $|\widehat{S}|=m$, $\max \widehat{S}=n-1$, and $[n-1]\backslash \widehat{S}=\{p_1<\hdots<p_{N-1}\}$. Thus, it follows that
\begin{enumerate}
    \item[(1)] $\mathbf{v}\equiv \mathbf{v}'$ if and only if $\mathbf{\widehat{v}_1}\equiv \mathbf{\widehat{v}_2}$ by Theorem~\ref{thm:taunzequiv} and
    \item[(2)] $\mathbf{\widehat{v}_1}\equiv \mathbf{\widehat{v}_2}$ if and only if $\underline{\mathbf{v}}=\underline{\mathbf{\widehat{v}_1}}\equiv \underline{\mathbf{\widehat{v}_2}}=\underline{\mathbf{v}'}$ by our induction hypothesis.
\end{enumerate}
The result follows.
\end{proof}

Before moving to the next equivalence-preserving transformation of interest, we introduce two further results concerning flattening. Both results will be crucial in our study of the remaining equivalence-preserving transformations.

\begin{prop}\label{prop:zeroscheck}
    Let $S=\supp(\mathbf{v})$ with $|S|=m$. Then $\mathbf{v}\equiv 0$ if and only if $\underline{\mathbf{v}}\bullet_ku=0$ for all $u\in S_m$ and $k\in [m-1]$.
\end{prop}
\begin{proof}
    By Theorem~\ref{thm:flat}~(a), it suffices to show that $\underline{\mathbf{v}}\equiv 0$ if and only if $\underline{\mathbf{v}}\bullet_ku=0$ for all $u\in S_m$ and $k\in [m-1]$. The forward direction is immediate. For the backward direction, assume that $\underline{\mathbf{v}}\not\equiv 0$. Then there exist $N$, $u\in S_N$, and $k\in [N-1]$ such that $\underline{\mathbf{v}}\bullet_ku\neq 0$. We take $N$ to be the smallest possible. Note that $N>m$ by assumption, so that $N\notin \supp(\underline{\mathbf{v}})=[m]$. Thus, letting $t=u^{-1}(N)$ and applying Lemma~\ref{lem:nztrunc}, it follows that $$\underline{\mathbf{v}}\bullet_{\tau_t(k)}u/t=\tau_N(\underline{\mathbf{v}})\bullet_{\tau_t(k)}u/t\neq 0,$$ where $u/t\in S_{N-1}$, contradicting the minimality of $N$. The result follows.
\end{proof}

The proof of the following non-zero analogue of Proposition~\ref{prop:zeroscheck} is almost identical to that of Proposition~\ref{prop:taunzequiv1}, and so we omit the details.

\begin{prop}\label{prop:nzeroscheck}
    Let $\mathbf{v}=\mathbf{v}_{a_rb_r}\hdots\mathbf{v}_{a_1b_1}$, $\mathbf{v}'=\mathbf{v}_{c_rd_r}\hdots\mathbf{v}_{c_1d_1}$, and $S=\supp(\mathbf{v})=\supp(\mathbf{v}')$ with $|S|=m$. Then $\mathbf{v}\equiv \mathbf{v}'$ if and only if $\underline{\mathbf{v}}\bullet_ku=\underline{\mathbf{v}'}\bullet_ku$ for all $u\in S_m$ and $k\in [m-1]$.
\end{prop}

\subsection{Cyclic Shift}
In this subsection, we consider the transformation 
corresponding to a cyclic shift of the indices.

For $n>0$ and $S=\{t_1<\cdots<t_m\}\subset [n]$, let $\memph{\mathfrak{o}_S}:S\to S$ be the map defined by 
$$\mathfrak{o}(t_i)=\begin{cases}
    t_{i+1}, & 1\le i<m\\
    t_1, & i=m.
\end{cases}$$
Given an operator $\mathbf{v}=\mathbf{v}_{a_rb_r}\cdots\mathbf{v}_{a_1b_1}$ with $S=\supp(\mathbf{v})$, we define $\memph{\mathfrak{o}(\mathbf{v})}:=\mathbf{v}_{\mathfrak{o}_S(a_r)\mathfrak{o}_S(b_r)}\cdots\mathbf{v}_{\mathfrak{o}_S(a_1)\mathfrak{o}_S(b_1)}$. We refer to this transformation as the \demph{cyclic shift}.

In the particular case of $\supp(\mathbf{v}) =[n]$, we simply denote $\memph{\mathfrak{o}_n}:=\mathfrak{o}_{[n]}$. Note that one can think of $\mathfrak{o}_n$ as the $n$-cycle $(2,3,4\ldots,n,1)$ acting on the indices of the operators in $\mathbf{v}$. Moreover, given $u\in S_n$ and $1\leq i<j\leq n$, we define $\mathfrak{o}_n(u)\in S_n$ by $$\mathfrak{o}_n(u)(i)=\begin{cases}
    u(i)+1 & u(i)\neq n \\
    1 & u(i)=1
\end{cases}$$
and
$$\memph{\mathbf{q}^{\mathfrak{o}(u,n,i,j)}}=
\begin{cases}
    \mathbf{q}_{ij}^{-1}, & \text{if}~u(i)=n, \\
    \mathbf{q}_{ij}, & \text{if}~u(j)=n, \\
    1, & \text{otherwise}.
\end{cases}
$$

\begin{Ex}
    For the operator $\mathbf{v}=\mathbf{v}_{28}\mathbf{v}_{83}\mathbf{v}_{25}$, we have $\mathfrak{o}(\mathbf{v})=\mathbf{v}_{32}\mathbf{v}_{25}\mathbf{v}_{38}$. The diagrams of $\mathbf{v}$ \textup(left\textup) and $\mathfrak{o}(\mathbf{v})$ \textup(right\textup) are illustrated in Figure~\ref{fig:cycshift1}. For $u=(2,4,3,1)\in S_4$, we have $\mathfrak{o}_4(u)=(3,1,4,2)$, $\mathbf{q}^{\mathfrak{o}(u,4,1,3)}=1$, $\mathbf{q}^{\mathfrak{o}(u,4,1,2)}=\mathbf{q}_{12}$, and $\mathbf{q}^{\mathfrak{o}(u,4,2,4)}=\mathbf{q}^{-1}_{24}$.
    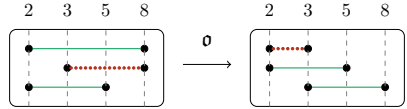
\begin{figure}[H]
    \centering
    \scalebox{0.8}{\begin{tikzpicture}[scale=0.8]
    \node at (0,0) {\scalebox{0.8}{\begin{tikzpicture}[scale=0.8]
        \draw[rounded corners] (-0.5, 0) rectangle (3.5, 2) {};
        
        \node (1) at (0,2.5) {$2$};
        \node (2) at (1,2.5) {$3$};
        \node (3) at (2,2.5)  {$5$};
        \node (4) at (3,2.5)  {$8$};

        \node (5) at (0,1.5) [circle, draw = black,fill=black, inner sep = 0.5mm] {};
        \node (6) at (3,1.5) [circle, draw = black,fill=black, inner sep = 0.5mm] {};
        \draw[ForestGreen] (5)--(6);

        \node (7) at (1,1) [circle, draw = black,fill=black, inner sep = 0.5mm] {};
        \node (8) at (3,1) [circle, draw = black,fill=black, inner sep = 0.5mm] {};
        \draw[decorate sep={0.5mm}{1mm},fill, BrickRed] (7)--(8);
        
        \node (9) at (0,0.5) [circle, draw = black,fill=black, inner sep = 0.5mm] {};
        \node (10) at (2,0.5) [circle, draw = black,fill=black, inner sep = 0.5mm] {};
        \draw[ForestGreen] (9)--(10);
        
        \draw[dashed,gray] (0,0)--(0,2);
        \draw[dashed,gray] (1,0)--(1,2);
        \draw[dashed,gray] (2,0)--(2,2);
        \draw[dashed,gray] (3,0)--(3,2);
    \end{tikzpicture}}};
    \draw[->] (2,-0.25)--(3,-0.25);
    \node at (2.5, 0.25) {$\mathfrak{o}$};
    \node at (5,0) {\scalebox{0.8}{\begin{tikzpicture}[scale=0.8]
        \draw[rounded corners] (-0.5, 0) rectangle (3.5, 2) {};
        
        \node (1) at (0,2.5) {$2$};
        \node (2) at (1,2.5) {$3$};
        \node (3) at (2,2.5)  {$5$};
        \node (4) at (3,2.5)  {$8$};

        \node (5) at (0,1.5) [circle, draw = black,fill=black, inner sep = 0.5mm] {};
        \node (6) at (1,1.5) [circle, draw = black,fill=black, inner sep = 0.5mm] {};
        \draw[decorate sep={0.5mm}{1mm},fill, BrickRed] (5)--(6);

        \node (7) at (0,1) [circle, draw = black,fill=black, inner sep = 0.5mm] {};
        \node (8) at (2,1) [circle, draw = black,fill=black, inner sep = 0.5mm] {};
        \draw[ForestGreen] (7)--(8);
        
        \node (9) at (1,0.5) [circle, draw = black,fill=black, inner sep = 0.5mm] {};
        \node (10) at (3,0.5) [circle, draw = black,fill=black, inner sep = 0.5mm] {};
        \draw[ForestGreen] (9)--(10);
        
        \draw[dashed,gray] (0,0)--(0,2);
        \draw[dashed,gray] (1,0)--(1,2);
        \draw[dashed,gray] (2,0)--(2,2);
        \draw[dashed,gray] (3,0)--(3,2);
        
    \end{tikzpicture}}};
    \end{tikzpicture}}
    \caption{Operator}
    \label{fig:cycshift1}
\end{figure}
\end{Ex}

Below, we establish that zero and non-zero equivalences are preserved by the cyclic shift, constituting the main result of this subsection.
\begin{theorem}\label{thm:cycshift}~
    \begin{enumerate}
        \item[\textup{(a)}] $\mathbf{v}\equiv \mathbf{0}$ if and only if $\mathfrak{o}(\mathbf{v})\equiv \mathbf{0}$.
        \item[\textup{(b)}] $\mathbf{v}\equiv\mathbf{v}'$ if and only if $\mathfrak{o}(\mathbf{v})\equiv \mathfrak{o}(\mathbf{v}')$.
    \end{enumerate}
\end{theorem}

We start by noticing that $\underline{\mathfrak{o}(\mathbf{v})}=\mathfrak{o}(\underline{\mathbf{v}})$, and so by Theorem~\ref{thm:flat}, it suffices to consider the case when $\supp(\mathbf{v})= \supp(\mathbf{v}')=[n]$ for some $n\in\mathbb{Z}_{>0}$. Moreover, by Propositions~\ref{prop:zeroscheck} and~\ref{prop:nzeroscheck}, we need only check the equivalence of the actions for all $u\in S_n$ and $k\in[n-1]$.

With the above in mind, Theorem~\ref{thm:cycshift}~(a) follows by induction, as shown in~\cite{qkB1}. The base case of a single operator is included below for later use.

\begin{lemma}~\cite[Corollary 5.11 (2)]{qkB1}\label{lem:osingle}
    Let $a,b\in[n]$ with $a\neq b$, $u\in S_n$, and $k\in[n-1]$. Then $$\mathbf{v}_{ab}\bullet_ku\neq 0 \qquad \text{ if and only if } \qquad \mathbf{v}_{\mathfrak{o}_n(a)\mathfrak{o}_n(b)}\bullet_k\mathfrak{o}_n(u)\neq 0.$$ 
    Moreover, 
    $$\mathbf{v}_{ab}\bullet_ku=\mathbf{q}^\alpha w\qquad \text{ if and only if } \qquad \mathbf{v}_{\mathfrak{o}_n(a)\mathfrak{o}_n(b)}\bullet_k\mathfrak{o}_n(u)=\mathbf{q}^{\mathfrak{o}(u,n,u^{-1}(a),u^{-1}(b))}\mathbf{q}^\alpha\mathfrak{o}_n(w).$$
\end{lemma}

Now, to finish the proof of Theorem~\ref{thm:cycshift}, it remains to establish Theorem~\ref{thm:cycshift}~(b). We break the proof of this result into two steps:
\begin{enumerate}
    \item[] (Step 1) We show that for an operator $\mathbf{v}$ with $\supp(\mathbf{v})=[n]$, $u\in S_n$, and $k\in [n-1]$,
    $$
    \mathbf{v}\bullet_ku=\mathbf{q}^\alpha w \qquad \text{ if and only if } \qquad \mathfrak{o}_n(\mathbf{v})\bullet_k\mathfrak{o}_n(u)={\bf q}^\beta\mathbf{q}^\alpha\mathfrak{o}_n(w),
    $$ for some $\beta\in\mathbb{Z}^{n-1}$ (Proposition~\ref{prop:cyclic_shift_non-zero}).
    \item[] (Step 2) We show that the aforementioned $\beta$ depends only on $u$ and $w$ (Lemma~\ref{lem:osameq}).
\end{enumerate}
Once both steps are completed, it follows that operators $\mathbf{v},\mathbf{v}'$ with support $[n]$ act equally on all elements of $S_n$ for all choices of $k\in [n-1]$ if and only if $\mathfrak{o}_n(\mathbf{v})$ and $\mathfrak{o}_n(\mathbf{v}')$ do. As noted above, this is enough to establish Theorem~\ref{thm:cycshift}~(b).

\begin{prop}\label{prop:cyclic_shift_non-zero}
   Let $\mathbf{v}=\mathbf{v}_{a_rb_r}\cdots\mathbf{v}_{a_1b_1}$ with $\supp(\mathbf{v})\subseteq[n]$. Consider $u\in S_n$ and $k\in [n-1]$ such that $\mathbf{v}\bullet_ku\neq 0$. Take $u_0=u$ and, for $1\le i<r$, define $u_i\in S_n$ as $\mathbf{v}_{a_ib_i}\cdots\mathbf{v}_{a_1b_1}\bullet_ku_0=\mathbf{q}^{\alpha^i}u_i$ with $\alpha^i\in\mathbb{Z}^{n-1}_{\ge 0}$.
   Then,
   $$\mathbf{v}\bullet_ku=\mathbf{q}^\alpha w \qquad \text{ if and only if } \qquad \mathfrak{o}_n(\mathbf{v})\bullet_k\mathfrak{o}_n(u)=\left(\prod_{i=1}^r\mathbf{q}^{\mathfrak{o}(u_{i-1},n,u_{i-1}^{-1}(a_i),u_{i-1}^{-1}(b_i))}\right)\mathbf{q}^\alpha\mathfrak{o}_n(w). $$
\end{prop}

\begin{proof}
    We proceed by induction on $r$. The base case is covered by Lemma~\ref{lem:osingle}. Assume the result holds for $r-1\ge 0$. Moreover, assume that $\mathbf{v}_{a_rb_r}\bullet_ku_{r-1}=\mathbf{q}^\beta w$; note that, by definition, we have $\beta+\alpha^{r-1}=\alpha$. Then, applying our induction hypothesis along with Lemma~\ref{lem:osingle}, we have that $\mathbf{v}\bullet_ku=w$ if and only if
    \begin{align*}
    \mathfrak{o}(\mathbf{v})\bullet_k\mathfrak{o}(u_0)&=\left(\prod_{i=1}^{r-1}\mathbf{q}^{\mathfrak{o}(u_{i-1},n,u_{i-1}^{-1}(a_i),u_{i-1}^{-1}(b_i))}\right)\mathbf{q}^{\alpha^{r-1}}\mathfrak{o}_n(\mathbf{v}_{a_rb_r})\bullet_k\mathfrak{o}_n(u_{r-1})\\
    &=\left(\prod_{i=1}^{r-1}\mathbf{q}^{\mathfrak{o}(u_{i-1},n,u_{i-1}^{-1}(a_i),u_{i-1}^{-1}(b_i))}\right)\mathbf{q}^{\alpha^{r-1}}\mathbf{q}^{\mathfrak{o}(u_{r-1},n,u_{r-1}^{-1}(a_r),u_{i-1}^{-1}(b_r))}\mathbf{q}^\beta\mathfrak{o}_n(u)\\
    &=\left(\prod_{i=1}^{r}\mathbf{q}^{\mathfrak{o}(u_{i-1},n,u_{i-1}^{-1}(a_i),u_{i-1}^{-1}(b_i))}\right)\mathbf{q}^\alpha\mathfrak{o}_n(u),
\end{align*}
as desired. Thus, the result follows by induction.
\end{proof}

Finally, we look at the $\mathbf{q}$-monomial appearing in Proposition~\ref{prop:cyclic_shift_non-zero}.  
\begin{lemma}\label{lem:osameq}
With the same assumptions and notation as in Proposition~\ref{prop:cyclic_shift_non-zero}, together with $u^{-1}(n)=i$ and $w^{-1}(n)=j$, we have that
$$\prod_{i=1}^{r}\mathbf{q}^{\mathfrak{o}(u_{i-1},n,(u_{i-1})^{-1}(a_i),(u_{i-1})^{-1}(b_i))}=\begin{cases}
        \mathbf{q}_{ji}, & \text{ if } j<i \\
        \mathbf{q}_{ij}^{-1}, & \text{ if } i<j \\
        1, & \text{ if } i=j.
    \end{cases}$$ 
\end{lemma}

\begin{proof}
    We proceed by induction on $r$. The base case is immediate. Assume the result holds for $r-1\ge 0$. There are three cases.
    \bigskip

    \noindent
    \underline{Case (a):} $n\notin \{a_r,b_r\}$. In this case $\mathbf{q}^{\mathfrak{o}(u_{r-1},n,(u_{r-1})^{-1}(a_r),(u_{r-1})^{-1}(b_r))}=1$ so that $$\prod_{i=1}^{r}\mathbf{q}^{\mathfrak{o}(u_{i-1},n,(u_{i-1})^{-1}(a_i),(u_{i-1})^{-1}(b_i))}=\prod_{i=1}^{r-1}\mathbf{q}^{\mathfrak{o}(u_{i-1},n,(u_{i-1})^{-1}(a_i),(u_{i-1})^{-1}(b_i))}.$$ Applying the induction hypothesis, we have that $$\prod_{i=1}^{r}\mathbf{q}^{\mathfrak{o}(u_{i-1},n,(u_{i-1})^{-1}(a_i),(u_{i-1})^{-1}(b_i))}=\begin{cases}
        \mathbf{q}_{u_{r-1}^{-1}(n)u^{-1}(n)}, & u_{r-1}^{-1}(n)<u^{-1}(n) \\
        \mathbf{q}_{u^{-1}(n)u_{r-1}^{-1}(n)}^{-1}, & u^{-1}(n)<u_{r-1}^{-1}(n) \\
        1, & u^{-1}(n)=u_{r-1}^{-1}(n)
    \end{cases}=\begin{cases}
        \mathbf{q}_{u_{r}^{-1}(n)u^{-1}(n)}, & u_{r}^{-1}(n)<u^{-1}(n) \\
        \mathbf{q}_{u^{-1}(n)u_{r}^{-1}(n)}^{-1}, & u^{-1}(n)<u_{r}^{-1}(n) \\
        1, & u^{-1}(n)=u_{r}^{-1}(n),
    \end{cases}$$ as desired, where the final equality holds since $u_{r}^{-1}(n)=u_{r-1}^{-1}(n)$.
    \bigskip

    \noindent
    \underline{Case (b):} $n=a_r$. In this case $\mathbf{q}^{\mathfrak{o}(u_{r-1},n,(u_{r-1})^{-1}(a_r),(u_{r-1})^{-1}(b_r))}=\mathbf{q}^{-1}_{u_{r-1}^{-1}(n)u_{r}^{-1}(n)}$. Applying our induction hypothesis, we have that 
    \begin{align*}
        \prod_{i=1}^{r}\mathbf{q}^{\mathfrak{o}(u_{i-1},n,(u_{i-1})^{-1}(a_i),(u_{i-1})^{-1}(b_i))}&=\mathbf{q}^{-1}_{u_{r-1}^{-1}(n)u_{r}^{-1}(n)}\prod_{i=1}^{r-1}\mathbf{q}^{\mathfrak{o}(u_{i-1},n,(u_{i-1})^{-1}(a_i),(u_{i-1})^{-1}(b_i))} \\
        &~\\
        &=\mathbf{q}^{-1}_{u_{r-1}^{-1}(n)u_{r}^{-1}(n)}\begin{cases}
        \mathbf{q}_{u_{r-1}^{-1}(n)u^{-1}(n)}, & u_{r-1}^{-1}(n)<u^{-1}(n) \\
        \mathbf{q}_{u^{-1}(n)u_{r-1}^{-1}(n)}^{-1}, & u^{-1}(n)<u_{r-1}^{-1}(n) \\
        1, & u^{-1}(n)=u_{r-1}^{-1}(n)
    \end{cases} \\
    &~\\
    &=\begin{cases}
        \mathbf{q}_{u_{r}^{-1}(n)u^{-1}(n)}, & u_{r-1}^{-1}(n)<u_{r}^{-1}(n)<u^{-1}(n) \\
        \mathbf{q}^{-1}_{u^{-1}(n)u_{r}^{-1}(n)}, & u_{r-1}^{-1}(n),u^{-1}(n)<u_{r}^{-1}(n)~\text{and}~u_{r-1}^{-1}(n)\neq u^{-1}(n) \\
        1, & u_{r-1}^{-1}(n)<u_{r}^{-1}(n)=u^{-1}(n) \\
        \mathbf{q}^{-1}_{u_{r-1}^{-1}(n)u_{r}^{-1}(n)}, & u^{-1}(n)=u_{r-1}^{-1}(n)<u^{-1}_r(n)
    \end{cases} \\
    &~\\
    &=\begin{cases}
        \mathbf{q}_{u_{r}^{-1}(n)u^{-1}(n)}, & u_{r}^{-1}(n)<u^{-1}(n) \\
        \mathbf{q}^{-1}_{u^{-1}(n)u_{r}^{-1}(n)}, & u^{-1}(n)<u_{r}^{-1}(n) \\
        1, & u^{-1}(n)=u_{r}^{-1}(n),
    \end{cases}
    \end{align*}
    as desired.
    \bigskip

    \noindent
    \underline{Case (c):} $n=b_r$. In this case $\mathbf{q}^{\mathfrak{o}(u_{r-1},n,(u_{r-1})^{-1}(a_r),(u_{r-1})^{-1}(b_r))}=\mathbf{q}_{u_{r}^{-1}(n)u_{r-1}^{-1}(n)}$. Applying our induction hypothesis, we have that  
    \begin{align*}
        \prod_{i=1}^{r}\mathbf{q}^{\mathfrak{o}(u_{i-1},n,(u_{i-1})^{-1}(a_i),(u_{i-1})^{-1}(b_i))}&=\mathbf{q}_{u_{r}^{-1}(n)u_{r-1}^{-1}(n)}\prod_{i=1}^{r-1}\mathbf{q}^{\mathfrak{o}(u_{i-1},(u_{i-1})^{-1}(a_i),(u_{i-1})^{-1}(b_i))} \\
        &~\\
        &=\mathbf{q}_{u_{r}^{-1}(n)u_{r-1}^{-1}(n)}\begin{cases}
        \mathbf{q}_{u_{r-1}^{-1}(n)u^{-1}(n)}, & u_{r-1}^{-1}(n)<u^{-1}(n) \\
        \mathbf{q}_{u^{-1}(n)u_{r-1}^{-1}(n)}^{-1}, & u^{-1}(n)<u_{r-1}^{-1}(n) \\
        1, & u^{-1}(n)=u_{r-1}^{-1}(n)
    \end{cases} \\
    &~\\
    &=\begin{cases}
        \mathbf{q}_{u_{r}^{-1}(n)u^{-1}(n)}, & u_r^{-1}(n)<u_{r-1}^{-1}(n),u^{-1}(n)~\text{and}~u_{r-1}^{-1}(n)\neq u^{-1}(n) \\
        1, & u_{r}^{-1}(n)=u^{-1}(n)<u_{r-1}^{-1}(n) \\
        \mathbf{q}^{-1}_{u^{-1}(n)u_{r}^{-1}(n)}, & u^{-1}(n)<u_{r}^{-1}(n)<u_{r-1}^{-1}(n) \\
        \mathbf{q}_{u_{r}^{-1}(n)u_{r-1}^{-1}(n)}, & u_r^{-1}(n)<u^{-1}(n)=u_{r-1}^{-1}(n)
    \end{cases} \\
    &~\\
    &=\begin{cases}
        \mathbf{q}_{u_{r}^{-1}(n)u^{-1}(n)}, & u_{r}^{-1}(n)<u^{-1}(n) \\
        \mathbf{q}^{-1}_{u^{-1}(n)u_{r}^{-1}(n)}, & u^{-1}(n)<u_{r}^{-1}(n) \\
        1, & u^{-1}(n)=u_{r}^{-1}(n),
    \end{cases}
    \end{align*}
    as desired.
\end{proof}

\subsection{Horizontal Flip}

In this subsection, we consider the transformation 
corresponding to ``flipping'' the associated diagram horizontally.

For $S=\{t_1<\hdots<t_m\}\subseteq [n]$, let $\memph{\mathfrak{h}_S}:S\to S$ be the map defined by $$\mathfrak{h}_S(t_i)=t_{m-i+1}.$$ 
We extend this definition to permutations by defining $\memph{\mathfrak{h}_n}:S_n\to S_n$ as $\mathfrak{h}(u)=w_0uw_0$ for $u\in S_n$, where $w_0$ is the longest element of $S_n$, i.e., $\mathfrak{h}_n(u)(i)=n-u(n-i+1)+1$ for $i\in [n]$. Given an operator $\mathbf{v}=\mathbf{v}_{a_rb_r}\cdots\mathbf{v}_{a_1b_1}$ with $\supp(\mathbf{v})=S$, we define $\memph{\mathfrak{h}(\mathbf{v})}=\mathbf{v}_{\mathfrak{h}_S(b_r)\mathfrak{h}_S(a_r)}\cdots\mathbf{v}_{\mathfrak{h}_S(b_1)\mathfrak{h}_S(a_1)}$ and refer to this transformation as the \demph{horizontal flip}. Finally, define $\mathfrak{h}(\alpha)$ as the composition obtained from $\alpha=(\alpha_1,\hdots,\alpha_{n-1})\in\mathbb{Z}_{\ge 0}^{n-1}$ by reversing the entries, $\mathfrak{h}(\alpha)=(\alpha_{n-1},\hdots,\alpha_1)$.

\begin{Ex}
    For the operator $\mathbf{v}=\mathbf{v}_{28}\mathbf{v}_{83}\mathbf{v}_{25}$, we have $\mathfrak{h}(\mathbf{v})=\mathbf{v}_{28}\mathbf{v}_{52}\mathbf{v}_{38}$. The diagrams of $\mathbf{v}$ \textup(left\textup) and $\mathfrak{h}(\mathbf{v})$ \textup(right\textup) are illustrated in Figure~\ref{fig:horflip}. For $u=(2,4,3,1)\in S_4$ and $\alpha=(2,0,2,5)\in\mathbb{Z}^4_{\ge 0}$, we have $\mathfrak{h}_4(u)=(4,2,1,3)$ and $\mathfrak{h}(\alpha)=(5,2,0,2)$.
    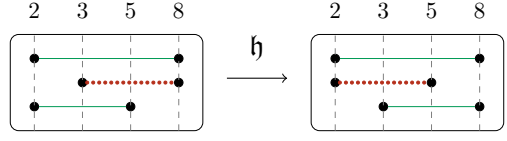
\begin{figure}[H]
        \centering
        $$\begin{tikzpicture}[scale=0.8]
    \node at (0,0) {\scalebox{0.8}{\begin{tikzpicture}[scale=0.8]
        \draw[rounded corners] (-0.5, 0) rectangle (3.5, 2) {};
        
        \node (1) at (0,2.5) {$2$};
        \node (2) at (1,2.5) {$3$};
        \node (3) at (2,2.5)  {$5$};
        \node (4) at (3,2.5)  {$8$};

        \node (5) at (0,1.5) [circle, draw = black,fill=black, inner sep = 0.5mm] {};
        \node (6) at (3,1.5) [circle, draw = black,fill=black, inner sep = 0.5mm] {};
        \draw[ForestGreen] (5)--(6);

        \node (7) at (1,1) [circle, draw = black,fill=black, inner sep = 0.5mm] {};
        \node (8) at (3,1) [circle, draw = black,fill=black, inner sep = 0.5mm] {};
        \draw[decorate sep={0.5mm}{1mm},fill, BrickRed] (7)--(8);
        
        \node (9) at (0,0.5) [circle, draw = black,fill=black, inner sep = 0.5mm] {};
        \node (10) at (2,0.5) [circle, draw = black,fill=black, inner sep = 0.5mm] {};
        \draw[ForestGreen] (9)--(10);
        
        \draw[dashed,gray] (0,0)--(0,2);
        \draw[dashed,gray] (1,0)--(1,2);
        \draw[dashed,gray] (2,0)--(2,2);
        \draw[dashed,gray] (3,0)--(3,2);
        
    \end{tikzpicture}}};
    \draw[->] (2,-0.25)--(3,-0.25);
    \node at (2.5, 0.25) {$\mathfrak{h}$};
    \node at (5,0) {\scalebox{0.8}{\begin{tikzpicture}[scale=0.8]
        \draw[rounded corners] (-0.5, 0) rectangle (3.5, 2) {};
        
        \node (1) at (0,2.5) {$2$};
        \node (2) at (1,2.5) {$3$};
        \node (3) at (2,2.5)  {$5$};
        \node (4) at (3,2.5)  {$8$};

        \node (5) at (0,1.5) [circle, draw = black,fill=black, inner sep = 0.5mm] {};
        \node (6) at (3,1.5) [circle, draw = black,fill=black, inner sep = 0.5mm] {};
        \draw[ForestGreen] (5)--(6);

        \node (7) at (0,1) [circle, draw = black,fill=black, inner sep = 0.5mm] {};
        \node (8) at (2,1) [circle, draw = black,fill=black, inner sep = 0.5mm] {};
        \draw[decorate sep={0.5mm}{1mm},fill, BrickRed] (7)--(8);
        
        \node (9) at (1,0.5) [circle, draw = black,fill=black, inner sep = 0.5mm] {};
        \node (10) at (3,0.5) [circle, draw = black,fill=black, inner sep = 0.5mm] {};
        \draw[ForestGreen] (9)--(10);
        
        \draw[dashed,gray] (0,0)--(0,2);
        \draw[dashed,gray] (1,0)--(1,2);
        \draw[dashed,gray] (2,0)--(2,2);
        \draw[dashed,gray] (3,0)--(3,2);
        
    \end{tikzpicture}}};
    \end{tikzpicture}$$
        \caption{Horizontal Flip}
        \label{fig:horflip}
    \end{figure}
\end{Ex}

Below, we establish that zero and non-zero equivalences are preserved by the horizontal flip.
\begin{theorem}\label{thm:hshift}~
\begin{enumerate}
    \item[\textup{(a)}] $\mathbf{v}\equiv \mathbf{0}$ if and only if $\mathfrak{h}(\mathbf{v})\equiv \mathbf{0}$.
    \item[\textup{(b)}] $\mathbf{v}\equiv \mathbf{v}$ if and only if $\mathfrak{h}(\mathbf{v})\equiv \mathfrak{h}(\mathbf{v}')$.
\end{enumerate}
\end{theorem}

As in the case of the cyclic shift, it suffices to establish Theorem~\ref{thm:hshift} in the case where the operators have support $[n]$, and we only need to check the equivalences for all $u\in S_n$ and $k\in [n-1]$. Moreover, as in Theorem~\ref{thm:cycshift}~(a), Theorem~\ref{thm:hshift}~(a) follows by induction, as shown in~\cite{qkB1}. The base case of a single operator is included below for later use.

\begin{lemma}\label{lem:hf1}\cite[Corollary 5.11 (3)]{qkB1}
    Let $a,b\in[n]$ with $a\neq b$, $u\in S_n$, and $k\in[n-1]$. Then, 
    $$\mathbf{v}_{ab}\bullet_ku\neq 0 \qquad \text{ if and only if }\qquad \mathfrak{h}_{[n]}(\mathbf{v}_{ab})\bullet_{n-k}\mathfrak{h}_n(u)\neq 0.$$ 
    Moreover,
    $$\mathbf{v}_{ab}\bullet_ku=\mathbf{q}^{\alpha}w \qquad \text{ if and only if } \qquad \mathfrak{h}_{[n]}(\mathbf{v}_{ab})\bullet_{n-k}\mathfrak{h}_n(u)=\mathbf{q}^{\mathfrak{h}(\alpha)}\mathfrak{h}_n(w).$$
\end{lemma}

Now, to finish the proof of Theorem~\ref{thm:hshift}, it remains to establish Theorem~\ref{thm:hshift}~(b). To do so, we show that for an operator $\mathbf{v}$ with $\supp(\mathbf{v})=[n]$, $u\in S_n$, and $k\in [n-1]$,  $\mathbf{v}\bullet_ku=\mathbf{q}^{\alpha}w$ if and only if $\mathfrak{h}_n(\mathbf{v})\bullet_{n-k}\mathfrak{h}_n(u)=\mathbf{q}^{\mathfrak{h}(\alpha)}\mathfrak{h}_n(w)$. From this it follows that operators $\mathbf{v},\mathbf{v}'$ with support $[n]$ have equal action on all elements of $S_n$ for all choices of $k\in [n-1]$ if and only if $\mathfrak{h}_n(\mathbf{v})$ and $\mathfrak{h}_n(\mathbf{v}')$ do.

\begin{prop}\label{prop:hfr}
Let $\mathbf{v}=\mathbf{v}_{a_rb_r}\cdots\mathbf{v}_{a_1b_1}$ with $\supp(\mathbf{v})\subseteq [n]$, $u\in S_n$, and $k\in [n-1]$ be such that $\mathbf{v}\bullet_ku\neq 0$.  Then, 
$$\mathbf{v}\bullet_ku=\mathbf{q}^{\alpha}w \qquad \text{ if and only if } \qquad \mathfrak{h}_{[n]}(\mathbf{v})\bullet_{n-k}\mathfrak{h}_n(u)=\mathbf{q}^{\mathfrak{h}(\alpha)} \mathfrak{h}_n(w).$$ 
\end{prop}
    \begin{proof}
        We proceed by induction on $r$. The base case is covered by Lemma~\ref{lem:hf1}. Assume the result holds for $r-1\ge 0$. Since $\mathbf{v}\bullet_ku=\mathbf{v}_{a_rb_r}\cdots\mathbf{v}_{a_1b_1}\bullet_ku\neq 0$, it follows that $\mathbf{v}_{a_{r-1}b_{r-1}}\cdots\mathbf{v}_{a_1b_1}\bullet_ku\neq 0$. Suppose that $\mathbf{v}_{a_{r-1}b_{r-1}}\cdots\mathbf{v}_{a_1b_1}\bullet_ku=\mathbf{q}^{\beta}u^*$ for $\beta\in\mathbb{Z}_{\ge 0}^{n-1}$ and $u^*\in S_n$. Then, applying our induction hypothesis along with Lemma~\ref{lem:hf1}, we have that $$\mathbf{v}\bullet_ku=\mathbf{v}_{a_rb_r}\bullet_k\mathbf{q}^{\beta}u^*=\mathbf{q}^{\beta}\mathbf{v}_{a_rb_r}\bullet_ku^*=\mathbf{q}^{\beta}\mathbf{q}^{\gamma}w=\mathbf{q}^{\beta+\gamma}w=\mathbf{q}^{\alpha}w$$ if and only if 
    \begin{align*}
        \mathfrak{h}_{[n]}(\mathbf{v})\bullet_{n-k}\mathfrak{h}_n(u)&=\mathfrak{h}_{[n]}(\mathbf{v}_{a_rb_r})\bullet_{n-k}\mathbf{q}^{\mathfrak{h}(\beta)}\mathfrak{h}_n(u^*)=\mathbf{q}^{\mathfrak{h}(\beta)}\mathfrak{h}_{[n]}(\mathbf{v}_{a_rb_r})\bullet_{n-k}\mathfrak{h}_n(u^*)=\mathbf{q}^{\mathfrak{h}(\beta)}\mathbf{q}^{\mathfrak{h}(\gamma)}\mathfrak{h}_n(w)\\
        &=\mathbf{q}^{\mathfrak{h}(\beta) + \mathfrak{h}(\gamma)}\mathfrak{h}_n(w).
    \end{align*}
    Since $\alpha=\beta+\gamma$ if and only if $\mathfrak{h}(\alpha)=\mathfrak{h}(\beta) + \mathfrak{h}(\gamma)$ as vectors in $\mathbb{Z}^{n-1}_{\ge 0}$, we have $\mathbf{q}^{\mathfrak{h}(\beta) +\mathfrak{h}(\gamma)}\mathfrak{h}_n(w)=\mathbf{q}^{\mathfrak{h}(\alpha)}\mathfrak{h}_n(w)$ and the result follows.
\end{proof}

\subsection{Vertical Flip}

In this subsection, we consider the transformation 
corresponding to reversing the order of the operators, which essentially ``flips'' the associated diagram vertically. 

In this case, the transformation is defined directly for the operators. For an operator $\mathbf{v}=\mathbf{v}_{a_rb_r}\cdots\mathbf{v}_{a_1b_1}$, we define the transformation $\memph{\rho}$ by $\rho(\mathbf{v})=\mathbf{v}_{a_1b_1}\cdots\mathbf{v}_{a_rb_r}$ and refer to this transformation as the \demph{vertical flip}. In addition, we need the corresponding definition for permutations and compositions.
Given a permutation $u\in S_n$, define $\memph{\rho_n(u)}:S_n\to S_n$ by $\rho_n(u):=uw_0$ where $w_0$ is the longest element of $S_n$, i.e., $\rho_n(u)$ is the permutation for which $\rho_n(u)(i)=u(n-i+1)$ for $i\in [n]$. 

\begin{Ex}
    For the operator $\mathbf{v}=\mathbf{v}_{28}\mathbf{v}_{83}\mathbf{v}_{25}$, we have $\rho(\mathbf{v})=\mathbf{v}_{25}\mathbf{v}_{83}\mathbf{v}_{28}$. The diagrams of $\mathbf{v}$ \textup(left\textup) and $\rho(\mathbf{v})$ \textup(right\textup) are illustrated in Figure~\ref{fig:verflip}. For $u=(2,4,3,1)\in S_4$, we have $\rho_n(u)=(1,3,4,2)$.
    \begin{figure}[H]
        \centering
        \begin{tikzpicture}[scale=0.8]
    \node at (0,0) {\scalebox{0.8}{\begin{tikzpicture}[scale=0.8]
        \draw[rounded corners] (-0.5, 0) rectangle (3.5, 2) {};
        
        \node (1) at (0,2.5) {$1$};
        \node (2) at (1,2.5) {$2$};
        \node (3) at (2,2.5)  {$3$};
        \node (4) at (3,2.5)  {$4$};

        \node (5) at (0,1.5) [circle, draw = black,fill=black, inner sep = 0.5mm] {};
        \node (6) at (3,1.5) [circle, draw = black,fill=black, inner sep = 0.5mm] {};
        \draw[ForestGreen] (5)--(6);

        \node (7) at (1,1) [circle, draw = black,fill=black, inner sep = 0.5mm] {};
        \node (8) at (3,1) [circle, draw = black,fill=black, inner sep = 0.5mm] {};
        \draw[decorate sep={0.5mm}{1mm},fill, BrickRed] (7)--(8);
        
        \node (9) at (0,0.5) [circle, draw = black,fill=black, inner sep = 0.5mm] {};
        \node (10) at (2,0.5) [circle, draw = black,fill=black, inner sep = 0.5mm] {};
        \draw[ForestGreen] (9)--(10);
        
        \draw[dashed,gray] (0,0)--(0,2);
        \draw[dashed,gray] (1,0)--(1,2);
        \draw[dashed,gray] (2,0)--(2,2);
        \draw[dashed,gray] (3,0)--(3,2);
    \end{tikzpicture}}};
    \draw[->] (2,-0.25)--(3,-0.25);
    \node at (2.5, 0.25) {$\rho$};
    \node at (5,0) {\scalebox{0.8}{\begin{tikzpicture}[scale=0.8]
        \draw[rounded corners] (-0.5, 0) rectangle (3.5, 2) {};
        
        \node (1) at (0,2.5) {$1$};
        \node (2) at (1,2.5) {$2$};
        \node (3) at (2,2.5)  {$3$};
        \node (4) at (3,2.5)  {$4$};

        \node (5) at (0,0.5) [circle, draw = black,fill=black, inner sep = 0.5mm] {};
        \node (6) at (3,0.5) [circle, draw = black,fill=black, inner sep = 0.5mm] {};
        \draw[ForestGreen] (5)--(6);

        \node (7) at (1,1) [circle, draw = black,fill=black, inner sep = 0.5mm] {};
        \node (8) at (3,1) [circle, draw = black,fill=black, inner sep = 0.5mm] {};
        \draw[decorate sep={0.5mm}{1mm},fill, BrickRed] (7)--(8);
        
        \node (9) at (0,1.5) [circle, draw = black,fill=black, inner sep = 0.5mm] {};
        \node (10) at (2,1.5) [circle, draw = black,fill=black, inner sep = 0.5mm] {};
        \draw[ForestGreen] (9)--(10);
        
        \draw[dashed,gray] (0,0)--(0,2);
        \draw[dashed,gray] (1,0)--(1,2);
        \draw[dashed,gray] (2,0)--(2,2);
        \draw[dashed,gray] (3,0)--(3,2);
    \end{tikzpicture}}};
    \end{tikzpicture}
        \caption{Vertical Flip}
        \label{fig:verflip}
    \end{figure}
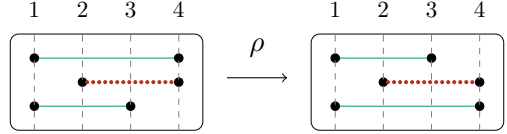
\end{Ex}

Below, we establish that the zero and non-zero equivalences are preserved by the vertical flip.

\begin{theorem}\label{thm:vshift}~
\begin{enumerate}
    \item[\textup{(a)}] $\mathbf{v}\equiv \mathbf{0}$ if and only if $\rho(\mathbf{v})\equiv \mathbf{0}$.
    \item[\textup{(b)}] $\mathbf{v}\equiv \mathbf{v}'$ if and only if $\rho(\mathbf{v})\equiv \rho(\mathbf{v}')$.
\end{enumerate}
\end{theorem}

As in the previous cases, it suffices to establish this result in the case where the operators have support $[n]$, and we only need to check the equivalences for all $u\in S_n$ and $k\in [n-1]$. Moreover, as in the cases of Theorems~\ref{thm:cycshift}~(a) and~\ref{thm:hshift}~(a), Theorem~\ref{thm:vshift}~(a) follows by induction, as shown in~\cite{qkB1}. The base case of a single operator is included below for later use.

\begin{lemma}~\cite[Corollary 5.11 (4)]{qkB1}\label{lem:vf1}
    Let $a, b\in [n]$ with $a\neq b$, $u,w\in S_n$, and $k\in [n-1]$. Then, $\rho(\mathbf{v}_{ab}) = \mathbf{v}_{ab}$, and
    $$\mathbf{v}_{ab}\bullet_ku=\mathbf{q}^{\alpha}w \qquad \text{ if and only if } \qquad \mathbf{v}_{ab}\bullet_{n-k}\rho_n(w)=\mathbf{q}^{\mathfrak{h}(\alpha)}\rho_n(u).$$
\end{lemma}

Now, to finish the proof of Theorem~\ref{thm:vshift}, it remains to establish Theorem~\ref{thm:vshift}~(b). To do so, we show that for an operator $\mathbf{v}$ with $\supp(\mathbf{v})=[n]$, $u\in S_n$, and $k\in [n-1]$,  $\mathbf{v}\bullet_ku=\mathbf{q}^{\alpha}w$ if and only if $\rho(\mathbf{v})\bullet_{n-k}\rho_n(w)=\mathbf{q}^{\mathfrak{h}(\alpha)}\rho_n(u)$. From this it follows that operators $\mathbf{v},\mathbf{v}'$ with support $[n]$ have equal action on all elements of $S_n$ for all choices of $k\in [n-1]$ if and only if $\rho(\mathbf{v})$ and $\rho(\mathbf{v}')$ do. 

\begin{prop}\label{prop:vflip}
 Let $\mathbf{v}=\mathbf{v}_{a_rb_r}\cdots\mathbf{v}_{a_1b_1}$, $u\in S_n$, and $k\in [n-1]$. Then $\mathbf{v}\bullet_ku=\mathbf{q}^{\alpha}w$ if and only if $\rho(\mathbf{v})\bullet_{n-k}\rho_n(w)=\mathbf{q}^{\mathfrak{h}(\alpha)}\rho_n(u)$.
\end{prop}
\begin{proof}
    We proceed by induction on $r$. The base case is covered by Lemma~\ref{lem:vf1}. Assume the result holds for $r-1\ge 1$. Since $\mathbf{v}\bullet_ku=\mathbf{v}_{a_rb_r}\cdots\mathbf{v}_{a_1b_1}\bullet_ku\neq 0$, it follows that $\mathbf{v}_{a_{r-1}b_{r-1}}\cdots\mathbf{v}_{a_1b_1}\bullet_ku\neq 0$. Suppose that $\mathbf{v}_{a_{r-1}b_{r-1}}\cdots\mathbf{v}_{a_1b_1}\bullet_ku=\mathbf{q}^{\beta}u^*$ with $\beta\in\mathbb{Z}^{n-1}_{\ge 0}$ and $u^*\in S_n$. Then, applying our inductive hypothesis along with Lemma~\ref{lem:vf1}, we have that
    \begin{align*}
\mathbf{v}\bullet_ku&=\mathbf{v}_{a_rb_r}\cdots\mathbf{v}_{a_1b_1}\bullet_ku=\mathbf{v}_{a_r,b_r}\bullet_k(\mathbf{v}_{a_{r-1}b_{r-1}}\cdots\mathbf{v}_{a_1b_1})\bullet_ku=\mathbf{v}_{a_r,b_r}\bullet_k\mathbf{q}^{\beta}u^*=\mathbf{q}^{\beta}\mathbf{v}_{a_r,b_r}\bullet_ku^*\\
&=\mathbf{q}^{\beta}\mathbf{q}^{\gamma}w=\mathbf{q}^{\beta+\gamma}w=\mathbf{q}^{\alpha}w
    \end{align*}
    if and only if 
    \begin{align*}
        \rho(\mathbf{v})\bullet_{n-k}\rho_n(w)&=\mathbf{v}_{a_1,b_1}\cdots\mathbf{v}_{a_{r-1}b_{r-1}}\mathbf{v}_{a_rb_r}\bullet_{n-k}\rho_n(w)=(\mathbf{v}_{a_1,b_1}\cdots\mathbf{v}_{a_{r-1}b_{r-1}})\bullet_{n-k}(\mathbf{v}_{a_rb_r}\bullet_{n-k}\rho_n(w))\\
        &=\mathbf{v}_{a_1,b_1}\cdots\mathbf{v}_{a_{r-1}b_{r-1}}\bullet_{n-k}\mathbf{q}^{\mathfrak{h}(\gamma)}\rho_n(u^*)=\mathbf{q}^{\mathfrak{h}(\gamma)}
        \mathbf{v}_{a_1,b_1}\cdots\mathbf{v}_{a_{r-1}b_{r-1}}\bullet_{n-k}\rho_n(u^*)\\
        &=\mathbf{q}^{\mathfrak{h}(\gamma)}\mathbf{q}^{\mathfrak{h}(\beta)}\rho_n(u)
        =\mathbf{q}^{\mathfrak{h}(\gamma)+\mathfrak{h}(\beta)}\rho_n(u)
        =\mathbf{q}^{\mathfrak{h}(\alpha)}\rho_n(u),
    \end{align*}
    as desired.
    
\end{proof}

\section{Low-Degree Equivalences}\label{sec:ldequiv}

The main goal of this section is to determine all the zero equivalences of degrees two, three, and four, and all the non-zero equivalences of degrees two and three. 
Since the equivalence of two compositions of operators is preserved under further applications of other operators, we introduce the following notion.
\begin{definition}
    We say that an equivalence is \demph{minimal} if it is not a consequence of an equivalence of a lower degree \textup(i.e., involving fewer operators\textup). 
\end{definition}
The rest of this section is organized by the degree of the composition of operators.

\subsection{Degree 2 Equivalences}
We start by analyzing the zero equivalences. By 
Theorem~\ref{thm:cycshift}~(a), it suffices to consider only those compositions for which the first operator is quantum. In Figure~\ref{fig:deg2rels} we illustrate the set of all such diagrams.

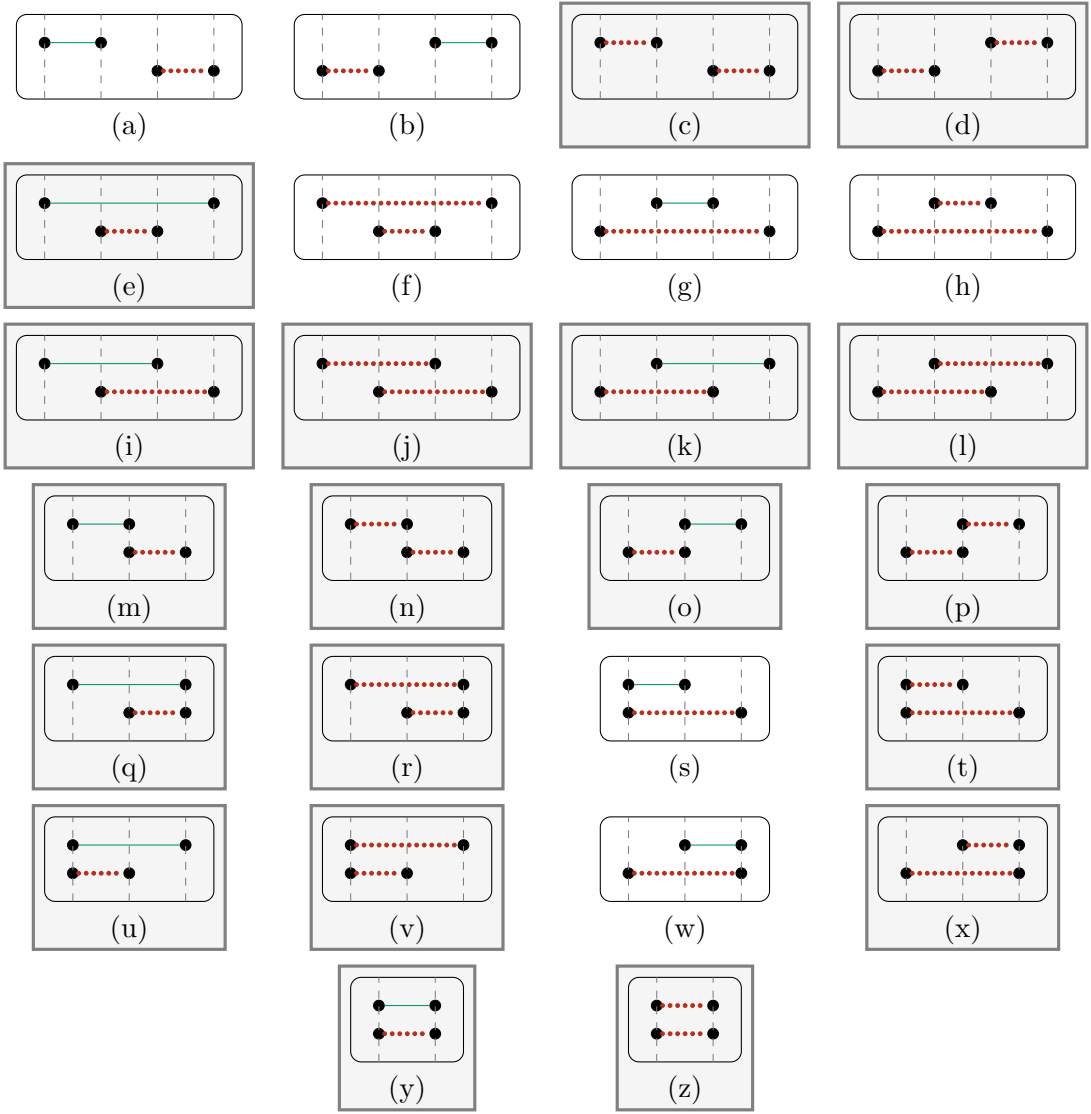
\begin{figure}[H]
    \centering
$$\begin{array}{cccc}
\begin{tikzpicture}[scale=0.75]
        \draw[rounded corners] (-0.5, 0) rectangle (3.5, 1.5) {};
        
        \node (1) at (0,1) [circle, draw = black,fill=black, inner sep = 0.5mm] {};
        \node (2) at (1,1) [circle, draw = black,fill=black, inner sep = 0.5mm] {};
        \draw[ForestGreen](1)--(2);
        
        \node (3) at (2,0.5) [circle, draw = black,fill=black, inner sep = 0.5mm] {};
        \node (4) at (3,0.5) [circle, draw = black,fill=black, inner sep = 0.5mm] {};
        \draw[decorate sep={0.5mm}{1mm},fill, BrickRed] (3)--(4);
        
        \draw[dashed,gray] (0,0)--(0,1.5);
        \draw[dashed,gray] (1,0)--(1,1.5);
        \draw[dashed,gray] (2,0)--(2,1.5);
        \draw[dashed,gray] (3,0)--(3,1.5);
        \node at (1.5,-0.5) {(a)};
    \end{tikzpicture} &
    \begin{tikzpicture}[scale=0.75]
     
        \draw[rounded corners] (-0.5, 0) rectangle (3.5, 1.5) {};
        \node (1) at (2,1) [circle, draw = black,fill=black, inner sep = 0.5mm] {};
        \node (2) at (3,1) [circle, draw = black,fill=black, inner sep = 0.5mm] {};
        \draw[ForestGreen] (1)--(2);
        
        \node (3) at (0,0.5) [circle, draw = black,fill=black, inner sep = 0.5mm] {};
        \node (4) at (1,0.5) [circle, draw = black,fill=black, inner sep = 0.5mm] {};
        \draw[decorate sep={0.5mm}{1mm},fill, BrickRed] (3)--(4);
        
        \draw[dashed,gray] (0,0)--(0,1.5);
        \draw[dashed,gray] (1,0)--(1,1.5);
        \draw[dashed,gray] (2,0)--(2,1.5);
        \draw[dashed,gray] (3,0)--(3,1.5);
        \node at (1.5,-0.5) {(b)};
    \end{tikzpicture}&
\begin{tikzpicture}[scale=0.75]
 \draw[gray, very thick, fill=gray!8] (-0.7,-0.85) rectangle (3.7,1.7);
    \draw[rounded corners] (-0.5, 0) rectangle (3.5, 1.5) {};    
    \node (1) at (0,1) [circle, draw = black,fill=black, inner sep = 0.5mm] {};
        \node (2) at (1,1) [circle, draw = black,fill=black, inner sep = 0.5mm] {};
        \draw[decorate sep={0.5mm}{1mm},fill, BrickRed] (1)--(2);
        
        \node (3) at (2,0.5) [circle, draw = black,fill=black, inner sep = 0.5mm] {};
        \node (4) at (3,0.5) [circle, draw = black,fill=black, inner sep = 0.5mm] {};
        \draw[decorate sep={0.5mm}{1mm},fill, BrickRed] (3)--(4);
        
        \draw[dashed,gray] (0,0)--(0,1.5);
        \draw[dashed,gray] (1,0)--(1,1.5);
        \draw[dashed,gray] (2,0)--(2,1.5);
        \draw[dashed,gray] (3,0)--(3,1.5);
        \node at (1.5,-0.5) {(c)};
    
    \end{tikzpicture} 
    &
    \begin{tikzpicture}[scale=0.75]
    \draw[gray, very thick, fill=gray!8] (-0.7,-0.85) rectangle (3.7,1.7);

        \draw[rounded corners] (-0.5, 0) rectangle (3.5, 1.5) {};
        
        \node (1) at (2,1) [circle, draw = black,fill=black, inner sep = 0.5mm] {};
        \node (2) at (3,1) [circle, draw = black,fill=black, inner sep = 0.5mm] {};
        \draw[decorate sep={0.5mm}{1mm},fill, BrickRed] (1)--(2);
        
        \node (3) at (0,0.5) [circle, draw = black,fill=black, inner sep = 0.5mm] {};
        \node (4) at (1,0.5) [circle, draw = black,fill=black, inner sep = 0.5mm] {};
        \draw[decorate sep={0.5mm}{1mm},fill, BrickRed] (3)--(4);

        \draw[dashed,gray] (0,0)--(0,1.5);
        \draw[dashed,gray] (1,0)--(1,1.5);
        \draw[dashed,gray] (2,0)--(2,1.5);
        \draw[dashed,gray] (3,0)--(3,1.5);
        \node at (1.5,-0.5) {(d)};
    \end{tikzpicture}\\
        \begin{tikzpicture}[scale=0.75]
    \draw[gray, very thick, fill=gray!8] (-0.7,-0.85) rectangle (3.7,1.7);
        \draw[rounded corners] (-0.5, 0) rectangle (3.5, 1.5) {};
        
        \node (1) at (0,1) [circle, draw = black,fill=black, inner sep = 0.5mm] {};
        \node (2) at (3,1) [circle, draw = black,fill=black, inner sep = 0.5mm] {};
        \draw[ForestGreen] (1)--(2);
        
        \node (3) at (1,0.5) [circle, draw = black,fill=black, inner sep = 0.5mm] {};
        \node (4) at (2,0.5) [circle, draw = black,fill=black, inner sep = 0.5mm] {};
        \draw[decorate sep={0.5mm}{1mm},fill, BrickRed] (3)--(4);
        
        \draw[dashed,gray] (0,0)--(0,1.5);
        \draw[dashed,gray] (1,0)--(1,1.5);
        \draw[dashed,gray] (2,0)--(2,1.5);
        \draw[dashed,gray] (3,0)--(3,1.5);
        \node at (1.5,-0.5) {(e)};
    \end{tikzpicture} &
    \begin{tikzpicture}[scale=0.75]
        \draw[rounded corners] (-0.5, 0) rectangle (3.5, 1.5) {};
        
        \node (1) at (0,1) [circle, draw = black,fill=black, inner sep = 0.5mm] {};
        \node (2) at (3,1) [circle, draw = black,fill=black, inner sep = 0.5mm] {};
        \draw[decorate sep={0.5mm}{1mm},fill, BrickRed] (1)--(2);
        \node (3) at (1,0.5) [circle, draw = black,fill=black, inner sep = 0.5mm] {};
        \node (4) at (2,0.5) [circle, draw = black,fill=black, inner sep = 0.5mm] {};
        \draw[decorate sep={0.5mm}{1mm},fill, BrickRed] (3)--(4);
            
        \draw[dashed,gray] (0,0)--(0,1.5);
        \draw[dashed,gray] (1,0)--(1,1.5);
        \draw[dashed,gray] (2,0)--(2,1.5);
        \draw[dashed,gray] (3,0)--(3,1.5);
        \node at (1.5,-0.5) {(f)};
    \end{tikzpicture}  &
        \begin{tikzpicture}[scale=0.75]
        \draw[rounded corners] (-0.5, 0) rectangle (3.5, 1.5) {};
        
        \node (1) at (0,0.5) [circle, draw = black,fill=black, inner sep = 0.5mm] {};
        \node (2) at (3,0.5) [circle, draw = black,fill=black, inner sep = 0.5mm] {};
        \draw[decorate sep={0.5mm}{1mm},fill, BrickRed] (1)--(2);
        
        \node (3) at (1,1) [circle, draw = black,fill=black, inner sep = 0.5mm] {};
        \node (4) at (2,1) [circle, draw = black,fill=black, inner sep = 0.5mm] {};
        \draw[ForestGreen] (3)--(4);

        \draw[dashed,gray] (0,0)--(0,1.5);
        \draw[dashed,gray] (1,0)--(1,1.5);
        \draw[dashed,gray] (2,0)--(2,1.5);
        \draw[dashed,gray] (3,0)--(3,1.5);
        \node at (1.5,-0.5) {(g)};
    \end{tikzpicture}  
    &
    \begin{tikzpicture}[scale=0.75]
        \draw[rounded corners] (-0.5, 0) rectangle (3.5, 1.5) {};
        
        \node (1) at (0,0.5) [circle, draw = black,fill=black, inner sep = 0.5mm] {};
        \node (2) at (3,0.5) [circle, draw = black,fill=black, inner sep = 0.5mm] {};
        \draw[decorate sep={0.5mm}{1mm},fill, BrickRed] (1)--(2);
        
        \node (3) at (1,1) [circle, draw = black,fill=black, inner sep = 0.5mm] {};
        \node (4) at (2,1) [circle, draw = black,fill=black, inner sep = 0.5mm] {};
        \draw[decorate sep={0.5mm}{1mm},fill, BrickRed] (3)--(4);

        \draw[dashed,gray] (0,0)--(0,1.5);
        \draw[dashed,gray] (1,0)--(1,1.5);
        \draw[dashed,gray] (2,0)--(2,1.5);
        \draw[dashed,gray] (3,0)--(3,1.5);
        \node at (1.5,-0.5) {(h)};
    \end{tikzpicture} \\
    \begin{tikzpicture}[scale=0.75]
     \draw[gray, very thick, fill=gray!8] (-0.7,-0.85) rectangle (3.7,1.7);

        \draw[rounded corners] (-0.5, 0) rectangle (3.5, 1.5) {};
        
        \node (1) at (0,1) [circle, draw = black,fill=black, inner sep = 0.5mm] {};
        \node (2) at (2,1) [circle, draw = black,fill=black, inner sep = 0.5mm] {};
        \draw[ForestGreen] (1)--(2);
        
        \node (3) at (1,0.5) [circle, draw = black,fill=black, inner sep = 0.5mm] {};
        \node (4) at (3,0.5) [circle, draw = black,fill=black, inner sep = 0.5mm] {};
        \draw[decorate sep={0.5mm}{1mm},fill, BrickRed] (3)--(4);

        \draw[dashed,gray] (0,0)--(0,1.5);
        \draw[dashed,gray] (1,0)--(1,1.5);
        \draw[dashed,gray] (2,0)--(2,1.5);
        \draw[dashed,gray] (3,0)--(3,1.5);
        \node at (1.5,-0.5) {(i)};
    \end{tikzpicture} 
    &
    \begin{tikzpicture}[scale=0.75]
     \draw[gray, very thick, fill=gray!8] (-0.7,-0.85) rectangle (3.7,1.7);

        \draw[rounded corners] (-0.5, 0) rectangle (3.5, 1.5) {};
        \node (1) at (0,1) [circle, draw = black,fill=black, inner sep = 0.5mm] {};
        \node (2) at (2,1) [circle, draw = black,fill=black, inner sep = 0.5mm] {};
        \draw[decorate sep={0.5mm}{1mm},fill, BrickRed] (1)--(2);
        \node (3) at (1,0.5) [circle, draw = black,fill=black, inner sep = 0.5mm] {};
        \node (4) at (3,0.5) [circle, draw = black,fill=black, inner sep = 0.5mm] {};
        \draw[decorate sep={0.5mm}{1mm},fill, BrickRed] (3)--(4);
        \draw[dashed,gray] (0,0)--(0,1.5);
        \draw[dashed,gray] (1,0)--(1,1.5);
        \draw[dashed,gray] (2,0)--(2,1.5);
        \draw[dashed,gray] (3,0)--(3,1.5);
        \node at (1.5,-0.5) {(j)};
    \end{tikzpicture} &
    \begin{tikzpicture}[scale=0.75]
     \draw[gray, very thick, fill=gray!8] (-0.7,-0.85) rectangle (3.7,1.7);

        \draw[rounded corners] (-0.5, 0) rectangle (3.5, 1.5) {};
        
        \node (1) at (1,1) [circle, draw = black,fill=black, inner sep = 0.5mm] {};
        \node (2) at (3,1) [circle, draw = black,fill=black, inner sep = 0.5mm] {};
        \draw[ForestGreen](1)--(2);
        
        \node (3) at (0,0.5) [circle, draw = black,fill=black, inner sep = 0.5mm] {};
        \node (4) at (2,0.5) [circle, draw = black,fill=black, inner sep = 0.5mm] {};
        \draw[decorate sep={0.5mm}{1mm},fill, BrickRed] (3)--(4);

        \draw[dashed,gray] (0,0)--(0,1.5);
        \draw[dashed,gray] (1,0)--(1,1.5);
        \draw[dashed,gray] (2,0)--(2,1.5);
        \draw[dashed,gray] (3,0)--(3,1.5);
        \node at (1.5,-0.5) {(k)};
    \end{tikzpicture}
    &
    \begin{tikzpicture}[scale=0.75]
     \draw[gray, very thick, fill=gray!8] (-0.7,-0.85) rectangle (3.7,1.7);

        \draw[rounded corners] (-0.5, 0) rectangle (3.5, 1.5) {};
        
        \node (1) at (1,1) [circle, draw = black,fill=black, inner sep = 0.5mm] {};
        \node (2) at (3,1) [circle, draw = black,fill=black, inner sep = 0.5mm] {};
        \draw[decorate sep={0.5mm}{1mm},fill, BrickRed] (1)--(2);
        
        \node (3) at (0,0.5) [circle, draw = black,fill=black, inner sep = 0.5mm] {};
        \node (4) at (2,0.5) [circle, draw = black,fill=black, inner sep = 0.5mm] {};
        \draw[decorate sep={0.5mm}{1mm},fill, BrickRed] (3)--(4);

        \draw[dashed,gray] (0,0)--(0,1.5);
        \draw[dashed,gray] (1,0)--(1,1.5);
        \draw[dashed,gray] (2,0)--(2,1.5);
        \draw[dashed,gray] (3,0)--(3,1.5);
        \node at (1.5,-0.5) {(l)};
    \end{tikzpicture}    \\ 
    \begin{tikzpicture}[scale=0.75]
     \draw[gray, very thick, fill=gray!8] (-0.7,-0.85) rectangle (2.7,1.7);
        \draw[rounded corners] (-0.5, 0) rectangle (2.5, 1.5) {};
        
        \node (1) at (0,1) [circle, draw = black,fill=black, inner sep = 0.5mm] {};
        \node (2) at (1,1) [circle, draw = black,fill=black, inner sep = 0.5mm] {};
        \draw[ForestGreen] (1)--(2);
        
        \node (3) at (1,0.5) [circle, draw = black,fill=black, inner sep = 0.5mm] {};
        \node (4) at (2,0.5) [circle, draw = black,fill=black, inner sep = 0.5mm] {};
        \draw[decorate sep={0.5mm}{1mm},fill, BrickRed] (3)--(4);

        \draw[dashed,gray] (0,0)--(0,1.5);
        \draw[dashed,gray] (1,0)--(1,1.5);
        \draw[dashed,gray] (2,0)--(2,1.5);
        \node at (1,-0.5) {(m)};
    \end{tikzpicture} 
    &
    \begin{tikzpicture}[scale=0.75]
    \draw[gray, very thick, fill=gray!8] (-0.7,-0.85) rectangle (2.7,1.7);
        \draw[rounded corners] (-0.5, 0) rectangle (2.5, 1.5) {};
        
        \node (1) at (0,1) [circle, draw = black,fill=black, inner sep = 0.5mm] {};
        \node (2) at (1,1) [circle, draw = black,fill=black, inner sep = 0.5mm] {};
        \draw[decorate sep={0.5mm}{1mm},fill, BrickRed] (1)--(2);
        
        \node (3) at (1,0.5) [circle, draw = black,fill=black, inner sep = 0.5mm] {};
        \node (4) at (2,0.5) [circle, draw = black,fill=black, inner sep = 0.5mm] {};
        \draw[decorate sep={0.5mm}{1mm},fill, BrickRed] (3)--(4);

        \draw[dashed,gray] (0,0)--(0,1.5);
        \draw[dashed,gray] (1,0)--(1,1.5);
        \draw[dashed,gray] (2,0)--(2,1.5);
        \node at (1,-0.5) {(n)};
    \end{tikzpicture} &
    \begin{tikzpicture}[scale=0.75]
    \draw[gray, very thick, fill=gray!8] (-0.7,-0.85) rectangle (2.7,1.7);
        \draw[rounded corners] (-0.5, 0) rectangle (2.5, 1.5) {};
        
        \node (1) at (1,1) [circle, draw = black,fill=black, inner sep = 0.5mm] {};
        \node (2) at (2,1) [circle, draw = black,fill=black, inner sep = 0.5mm] {};
        \draw[ForestGreen](1)--(2);
        
        \node (3) at (0,0.5) [circle, draw = black,fill=black, inner sep = 0.5mm] {};
        \node (4) at (1,0.5) [circle, draw = black,fill=black, inner sep = 0.5mm] {};
        \draw[decorate sep={0.5mm}{1mm},fill, BrickRed] (3)--(4);

        \draw[dashed,gray] (0,0)--(0,1.5);
        \draw[dashed,gray] (1,0)--(1,1.5);
        \draw[dashed,gray] (2,0)--(2,1.5);
        \node at (1,-0.5) {(o)};
    \end{tikzpicture}
    &
    \begin{tikzpicture}[scale=0.75]
    \draw[gray, very thick, fill=gray!8] (-0.7,-0.85) rectangle (2.7,1.7);
        \draw[rounded corners] (-0.5, 0) rectangle (2.5, 1.5) {};
        
        \node (1) at (1,1) [circle, draw = black,fill=black, inner sep = 0.5mm] {};
        \node (2) at (2,1) [circle, draw = black,fill=black, inner sep = 0.5mm] {};
        \draw[decorate sep={0.5mm}{1mm},fill, BrickRed] (1)--(2);
        
        \node (3) at (0,0.5) [circle, draw = black,fill=black, inner sep = 0.5mm] {};
        \node (4) at (1,0.5) [circle, draw = black,fill=black, inner sep = 0.5mm] {};
        \draw[decorate sep={0.5mm}{1mm},fill, BrickRed] (3)--(4);

        \draw[dashed,gray] (0,0)--(0,1.5);
        \draw[dashed,gray] (1,0)--(1,1.5);
        \draw[dashed,gray] (2,0)--(2,1.5);
        \node at (1,-0.5) {(p)};
    \end{tikzpicture} \\ 
    \begin{tikzpicture}[scale=0.75]
    \draw[gray, very thick, fill=gray!8] (-0.7,-0.85) rectangle (2.7,1.7);
        \draw[rounded corners] (-0.5, 0) rectangle (2.5, 1.5) {};
        
        \node (1) at (0,1) [circle, draw = black,fill=black, inner sep = 0.5mm] {};
        \node (2) at (2,1) [circle, draw = black,fill=black, inner sep = 0.5mm] {};
        \draw[ForestGreen] (1)--(2);
        
        \node (3) at (1,0.5) [circle, draw = black,fill=black, inner sep = 0.5mm] {};
        \node (4) at (2,0.5) [circle, draw = black,fill=black, inner sep = 0.5mm] {};
        \draw[decorate sep={0.5mm}{1mm},fill, BrickRed] (3)--(4);

        \draw[dashed,gray] (0,0)--(0,1.5);
        \draw[dashed,gray] (1,0)--(1,1.5);
        \draw[dashed,gray] (2,0)--(2,1.5);
        \node at (1,-0.5) {(q)};
    \end{tikzpicture}
    & 
    \begin{tikzpicture}[scale=0.75]
    \draw[gray, very thick, fill=gray!8] (-0.7,-0.85) rectangle (2.7,1.7);
        \draw[rounded corners] (-0.5, 0) rectangle (2.5, 1.5) {};
        
        \node (1) at (0,1) [circle, draw = black,fill=black, inner sep = 0.5mm] {};
        \node (2) at (2,1) [circle, draw = black,fill=black, inner sep = 0.5mm] {};
        \draw[decorate sep={0.5mm}{1mm},fill, BrickRed] (1)--(2);
        
        \node (3) at (1,0.5) [circle, draw = black,fill=black, inner sep = 0.5mm] {};
        \node (4) at (2,0.5) [circle, draw = black,fill=black, inner sep = 0.5mm] {};
        \draw[decorate sep={0.5mm}{1mm},fill, BrickRed] (3)--(4);

        \draw[dashed,gray] (0,0)--(0,1.5);
        \draw[dashed,gray] (1,0)--(1,1.5);
        \draw[dashed,gray] (2,0)--(2,1.5);
        \node at (1,-0.5) {(r)};
    \end{tikzpicture} &
    \begin{tikzpicture}[scale=0.75]
        \draw[rounded corners] (-0.5, 0) rectangle (2.5, 1.5) {};
        
        \node (1) at (0,1) [circle, draw = black,fill=black, inner sep = 0.5mm] {};
        \node (2) at (1,1) [circle, draw = black,fill=black, inner sep = 0.5mm] {};
        \draw[ForestGreen] (1)--(2);
        
        \node (3) at (0,0.5) [circle, draw = black,fill=black, inner sep = 0.5mm] {};
        \node (4) at (2,0.5) [circle, draw = black,fill=black, inner sep = 0.5mm] {};
        \draw[decorate sep={0.5mm}{1mm},fill, BrickRed] (3)--(4);
        
        \draw[dashed,gray] (0,0)--(0,1.5);
        \draw[dashed,gray] (1,0)--(1,1.5);
        \draw[dashed,gray] (2,0)--(2,1.5);
        \node at (1,-0.5) {(s)};
    \end{tikzpicture}
    &
    \begin{tikzpicture}[scale=0.75]
    \draw[gray, very thick, fill=gray!8] (-0.7,-0.85) rectangle (2.7,1.7);
        \draw[rounded corners] (-0.5, 0) rectangle (2.5, 1.5) {};
        
        \node (1) at (0,1) [circle, draw = black,fill=black, inner sep = 0.5mm] {};
        \node (2) at (1,1) [circle, draw = black,fill=black, inner sep = 0.5mm] {};
        \draw[decorate sep={0.5mm}{1mm},fill, BrickRed] (1)--(2);
        
        \node (3) at (0,0.5) [circle, draw = black,fill=black, inner sep = 0.5mm] {};
        \node (4) at (2,0.5) [circle, draw = black,fill=black, inner sep = 0.5mm] {};
        \draw[decorate sep={0.5mm}{1mm},fill, BrickRed] (3)--(4);
        
        \draw[dashed,gray] (0,0)--(0,1.5);
        \draw[dashed,gray] (1,0)--(1,1.5);
        \draw[dashed,gray] (2,0)--(2,1.5);
        \node at (1,-0.5) {(t)};
    \end{tikzpicture} \\ 
    \begin{tikzpicture}[scale=0.75]
    \draw[gray, very thick, fill=gray!8] (-0.7,-0.85) rectangle (2.7,1.7);
        \draw[rounded corners] (-0.5, 0) rectangle (2.5, 1.5) {};
        
        \node (1) at (0,1) [circle, draw = black,fill=black, inner sep = 0.5mm] {};
        \node (2) at (2,1) [circle, draw = black,fill=black, inner sep = 0.5mm] {};
        \draw[ForestGreen](1)--(2);
        
        \node (3) at (0,0.5) [circle, draw = black,fill=black, inner sep = 0.5mm] {};
        \node (4) at (1,0.5) [circle, draw = black,fill=black, inner sep = 0.5mm] {};
        \draw[decorate sep={0.5mm}{1mm},fill, BrickRed] (3)--(4);
        
        \draw[dashed,gray] (0,0)--(0,1.5);
        \draw[dashed,gray] (1,0)--(1,1.5);
        \draw[dashed,gray] (2,0)--(2,1.5);
        \node at (1,-0.5) {(u)};
    \end{tikzpicture}
    &
    \begin{tikzpicture}[scale=0.75]
    \draw[gray, very thick, fill=gray!8] (-0.7,-0.85) rectangle (2.7,1.7);
        \draw[rounded corners] (-0.5, 0) rectangle (2.5, 1.5) {};
        \node (1) at (0,1) [circle, draw = black,fill=black, inner sep = 0.5mm] {};
        \node (2) at (2,1) [circle, draw = black,fill=black, inner sep = 0.5mm] {};
        \draw[decorate sep={0.5mm}{1mm},fill, BrickRed] (1)--(2);
        
        \node (3) at (0,0.5) [circle, draw = black,fill=black, inner sep = 0.5mm] {};
        \node (4) at (1,0.5) [circle, draw = black,fill=black, inner sep = 0.5mm] {};
        \draw[decorate sep={0.5mm}{1mm},fill, BrickRed] (3)--(4);
        
        \draw[dashed,gray] (0,0)--(0,1.5);
        \draw[dashed,gray] (1,0)--(1,1.5);
        \draw[dashed,gray] (2,0)--(2,1.5);
        \node at (1,-0.5) {(v)};
    \end{tikzpicture} & 
    \begin{tikzpicture}[scale=0.75]
        \draw[rounded corners] (-0.5, 0) rectangle (2.5, 1.5) {};
        
        \node (1) at (1,1) [circle, draw = black,fill=black, inner sep = 0.5mm] {};
        \node (2) at (2,1) [circle, draw = black,fill=black, inner sep = 0.5mm] {};
        \draw[ForestGreen] (1)--(2);
        
        \node (3) at (0,0.5) [circle, draw = black,fill=black, inner sep = 0.5mm] {};
        \node (4) at (2,0.5) [circle, draw = black,fill=black, inner sep = 0.5mm] {};
        \draw[decorate sep={0.5mm}{1mm},fill, BrickRed] (3)--(4);

        \draw[dashed,gray] (0,0)--(0,1.5);
        \draw[dashed,gray] (1,0)--(1,1.5);
        \draw[dashed,gray] (2,0)--(2,1.5);
        \node at (1,-0.5) {(w)};
    \end{tikzpicture}  
    &
\begin{tikzpicture}[scale=0.75]
\draw[gray, very thick, fill=gray!8] (-0.7,-0.85) rectangle (2.7,1.7);
        \draw[rounded corners] (-0.5, 0) rectangle (2.5, 1.5) {};
        
        \node (1) at (1,1) [circle, draw = black,fill=black, inner sep = 0.5mm] {};
        \node (2) at (2,1) [circle, draw = black,fill=black, inner sep = 0.5mm] {};
        \draw[decorate sep={0.5mm}{1mm},fill, BrickRed] (1)--(2);
        
        \node (3) at (0,0.5) [circle, draw = black,fill=black, inner sep = 0.5mm] {};
        \node (4) at (2,0.5) [circle, draw = black,fill=black, inner sep = 0.5mm] {};
        \draw[decorate sep={0.5mm}{1mm},fill, BrickRed] (3)--(4);

        \draw[dashed,gray] (0,0)--(0,1.5);
        \draw[dashed,gray] (1,0)--(1,1.5);
        \draw[dashed,gray] (2,0)--(2,1.5);
        \node at (1,-0.5) {(x)};
    \end{tikzpicture}  \\ 
        & \begin{tikzpicture}[scale=0.75]
        \draw[gray, very thick, fill=gray!8] (-0.7,-0.85) rectangle (1.7,1.7);
        \draw[rounded corners] (-0.5, 0) rectangle (1.5, 1.5) {};
        
        \node (1) at (0,1) [circle, draw = black,fill=black, inner sep = 0.5mm] {};
        \node (2) at (1,1) [circle, draw = black,fill=black, inner sep = 0.5mm] {};
        \draw[ForestGreen] (1)--(2);
        
        \node (3) at (0,0.5) [circle, draw = black,fill=black, inner sep = 0.5mm] {};
        \node (4) at (1,0.5) [circle, draw = black,fill=black, inner sep = 0.5mm] {};
        \draw[decorate sep={0.5mm}{1mm},fill, BrickRed] (3)--(4);
        
        \draw[dashed,gray] (0,0)--(0,1.5);
        \draw[dashed,gray] (1,0)--(1,1.5);
        \node at (0.5,-0.5) {(y)};
    \end{tikzpicture}
    &
    \begin{tikzpicture}[scale=0.75]
    \draw[gray, very thick, fill=gray!8] (-0.7,-0.85) rectangle (1.7,1.7);
        \draw[rounded corners] (-0.5, 0) rectangle (1.5, 1.5) {};
        
        \node (1) at (0,1) [circle, draw = black,fill=black, inner sep = 0.5mm] {};
        \node (2) at (1,1) [circle, draw = black,fill=black, inner sep = 0.5mm] {};
        \draw[decorate sep={0.5mm}{1mm},fill, BrickRed] (1)--(2);
        
        \node (3) at (0,0.5) [circle, draw = black,fill=black, inner sep = 0.5mm] {};
        \node (4) at (1,0.5) [circle, draw = black,fill=black, inner sep = 0.5mm] {};
        \draw[decorate sep={0.5mm}{1mm},fill, BrickRed] (3)--(4);
        
        \draw[dashed,gray] (0,0)--(0,1.5);
        \draw[dashed,gray] (1,0)--(1,1.5);
        \node at (0.5,-0.5) {(z)};
    \end{tikzpicture} & 
\end{array}$$
\caption{Compositions of two operators with the first operator being a quantum operator}
\label{fig:deg2rels}
\end{figure}

The following result establishes which of the compositions appearing in Figure~\ref{fig:deg2rels} are equivalent to zero. 
\begin{prop}\label{prop:deg2zeroqf}
        Let $a<b$ and $\mathbf{v}=\mathbf{v}_{cd}\mathbf{v}_{ba}$. Then $\mathbf{v}\equiv \mathbf{0}$ if and only if one of the following holds:
        \begin{enumerate}
            \item[\textup{(i)}] $a= d<b= c$; 
            \item[\textup{(ii)}] $d<c< a<b$ or $a<b<d<c$; 
            \item[\textup{(iii)}] $c<d=a<b$, $d<c= a<b$, $a<b=c<d$, or $a<b=d<c$; 
            \item[\textup{(iv)}] $c<a<d<b$, $d<a<c<b$, $a<c<b<d$, or $a<d<b<c$; 
            \item[\textup{(v)}] $c<a<b=d$, $d<a<b=c$, $a=d<c<b$, $a=c<b<d$, $a=d<b<c$, or $a<d<c=b$;
            \item[\textup{(vi)}] $c<a<b<d$. 
        \end{enumerate}
The set of operators satisfying one of conditions listed in (i)-(vi) corresponds to the diagrams highlighted in Figure~\ref{fig:deg2rels}, which correspond to the labels (c), (d), (e), (i), (j), (k), (l), (m), (n),  (o), (p), (q), (r), (t), (u), (v), (x), and (z).
\end{prop}
    
    \begin{proof}
        First, we show that the compositions (a), (c), (f), (g), (h), (s), (w), and (y) of Figure~\ref{fig:deg2rels} are non-zero. Note that to do so, given such an operator $\mathbf{v}$, it suffices to identify $n\in\mathbb{Z}_{>0}$, $k\in [n-1]$, and $u\in S_n$ such that $\mathbf{v}\bullet_ku\neq 0$. To this end, consider the permutation $w=71635|428$, for $n=8$ and $k=5$. Then, we have that
\begin{equation*}
    \begin{array}{clccl}
{(a)} & \mathbf{v}_{12}\mathbf{v}_{54}\bullet_kw = q_5\cdot 72634|518 
& \qquad &
{(j)} & \mathbf{v}_{62}\mathbf{v}_{54} \bullet_kw= q_3q_4q^2_5q_6 \cdot 71234|568\\
{(k)} & \mathbf{v}_{24}\mathbf{v}_{52}\bullet_kw= q_5q_6 \cdot  71634|258
& \qquad & 
{(m)} & \mathbf{v}_{45}\mathbf{v}_{54}\bullet_k = q_5 \cdot 71635|428 = w \\
{(q)} & \mathbf{v}_{34}\mathbf{v}_{62}\bullet_kw = q_3q_4q_5q_6 \cdot 71245|368
& \qquad & 
{(r)} & \mathbf{v}_{54}\mathbf{v}_{62}\bullet_kw= q_3q_4q^2_5q_6 \cdot 71234|568 \\
{(s)} & \mathbf{v}_{56}\mathbf{v}_{62}\bullet_kw = q_3q_4q_5q_6 \cdot 71236|458
& \qquad & 
{(y)} & \mathbf{v}_{78}\mathbf{v}_{54}\bullet_kw = q_5 \cdot 81634|527
    \end{array}
\end{equation*}

To finish the proof, we must show that the remaining operators of Figure~\ref{fig:deg2rels} are equivalent to \textbf{0}. As it turns out, the remaining operators are cyclic shifts of compositions of two classical operators. Consequently, the fact that such operators are equivalent to \textbf{0} follows from~\cite[Proposition 5.2]{qkB1}. For the sake of illustrating how one shows that a composition of operators is equivalent to \textbf{0}, we include the proof for relation (b).

Suppose $a<b<c<d$ and $\mathbf{v}=\mathbf{v}_{ba}\mathbf{v}_{dc}\not\equiv \mathbf{0}$. Then there must exist a permutation $w=(w_1,\hdots,w_n) \in S_n$ and $k\in [n-1]$ for which $\mathbf{v}\bullet_k w\neq 0$. Since $\mathbf{v}=\mathbf{v}_{ba}\mathbf{v}_{dc}$, it must be the case that $\mathbf{v}_{dc}\bullet_k w=q^{\alpha}w'\neq 0$ for some $q^{\alpha}w'=q^{\alpha}(w_1',\hdots,w_n') \in S_n[\mathbf{q}]$; that is, if $w_{i_1}=c$ and $w_{j_1}=d$, then $j_1\le k<i_1$ and $c<w_l<d$ for all $j_1<l<i_1$. Now, similarly, as  
$$
\mathbf{v}\bullet_k w=\mathbf{v}_{ba}\bullet_k(\mathbf{v}_{dc}\bullet_k w)=\mathbf{v}_{ba}\bullet_k w'\neq 0,
$$
it must be the case that if $w'_{i_2}=a$ and $w'_{j_2}=b$, then $j_2\le k<i_2$ and $a<w_l<b$ for all $j_2<l<i_2$. Since $a<b<c<d$, it follows that $w_{i_2}=w'_{i_2}=a$ and $w_{j_2}=w'_{j_2}=b$. Consequently, there are two cases: either
\begin{itemize}
    \item $j_1<j_2\le k<i_1$, which contradicts our assumption that $c<w_l<d$ for all $j_1<l<i_1$, or
    \item $j_2<j_1\le k<i_2$, which contradicts our assumption that $a<w_l<b$ for all $j_2<l<i_2$;
\end{itemize}
that is, in either case, we are led to a contradiction. The result follows.
    \end{proof}
    
 Applying the equivalence preserving operators of Section~\ref{sec:ops} to the compositions of operators in Proposition~\ref{prop:deg2zeroqf}, we obtain the following theorem, which characterizes the degree two compositions that are equivalent to zero.
 
    \begin{theorem}\label{thm:deg2zero}
        Let $\mathbf{v}=\mathbf{v}_{cd}\mathbf{v}_{ab}$. Then $\mathbf{v}\equiv \mathbf{0}$ if and only if one of the following holds:
        \begin{enumerate}
            \item[\textup{(i)}] $a=c<b=d$ or $b=d<a=c$;
            \item[\textup{(ii)}] $a<d<c<b$ or $c<b<a<d$;
            \item[\textup{(iii)}] $a<c<b=d$, $a<d<b=c$, $a=c<d<b$, $a=d<c<b$, $c<a<b=d$, $c<b<a=d$, $a=c<b<d$, $b=c<a<d$, $b=d<c<a$, $d<b<a=c$, $d=b<a<c$, or $b<d<c=a$;
            \item[\textup{(iv)}] $a<c<b<d$, $b<c<a<d$, $a<d<b<c$, $b<d<a<c$, $c<a<d<b$, $c<b<d<a$, $d<a<c<b$, or $d<b<c<a$;
            \item[\textup{(v)}] $a<b=d<c$, $b<a=c<d$, $b<a=d<c$, $d<c=a<b$, $c<d=b<a$, or $d<c=b<a$;
            \item[\textup{(vi)}] $d<c<b<a$ or $b<a<d<c$.
        \end{enumerate}
    \end{theorem}    

Next, we look at equivalences between non-zero compositions of two operators. The following result provides a set of non-zero equivalences that appear in~\cite[Proposition 5.2]{qkB1}. We illustrate them in  Figure~\ref{fig:deg2equivrels}, where we keep the same labeling as in Figure~\ref{fig:deg2rels} for the diagrams included in that list.

\begin{prop}\label{prop:deg2rel}
Let $a<b<c<d$. We have the following list of equivalences of degree two:
\begin{multicols}{2}
        \begin{enumerate}
            \item[\textup{(i)}] $\mathbf{v}_{dc}\mathbf{v}_{ab}\equiv \mathbf{v}_{ab}\mathbf{v}_{dc}$,
            \item[\textup{(ii)}] $\mathbf{v}_{da}\mathbf{v}_{cb}\equiv\mathbf{v}_{cb}\mathbf{v}_{da}$,
            \item[\textup{(iii)}] $\mathbf{v}_{da}\mathbf{v}_{bc}\equiv\mathbf{v}_{bc}\mathbf{v}_{da}$,
            \item[\textup{(iv)}] $\mathbf{v}_{cd}\mathbf{v}_{ba}\equiv\mathbf{v}_{ba}\mathbf{v}_{cd}$,
            \item[\textup{(v)}] $\mathbf{v}_{ab}\mathbf{v}_{cd}\equiv\mathbf{v}_{cd}\mathbf{v}_{ab}$, and
            \item[\textup{(vi)}] $\mathbf{v}_{ad}\mathbf{v}_{bc}\equiv\mathbf{v}_{bc}\mathbf{v}_{ad}$.
        \end{enumerate}
        \end{multicols}
\end{prop}

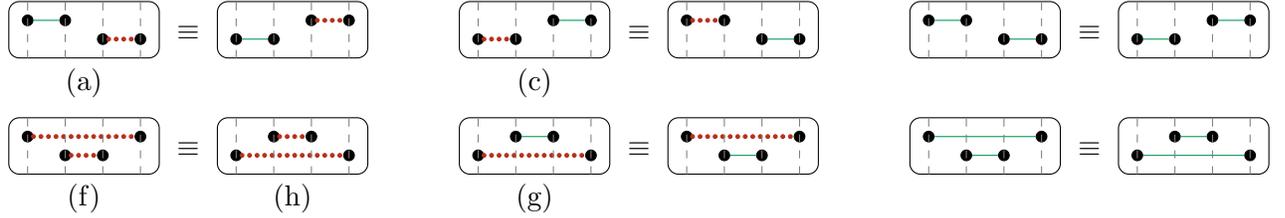
\begin{figure}[H]
    \centering
    $$
    \begin{array}{ccccccccccc}
    \begin{tikzpicture}[scale=0.5]
        \draw[rounded corners] (-0.5, 0) rectangle (3.5, 1.5) {};
        
        \node (1) at (0,1) [circle, draw = black,fill=black, inner sep = 0.5mm] {};
        \node (2) at (1,1) [circle, draw = black,fill=black, inner sep = 0.5mm] {};
        \draw[ForestGreen](1)--(2);
        
        \node (3) at (2,0.5) [circle, draw = black,fill=black, inner sep = 0.5mm] {};
        \node (4) at (3,0.5) [circle, draw = black,fill=black, inner sep = 0.5mm] {};
        \draw[decorate sep={0.5mm}{1mm},fill, BrickRed] (3)--(4);

        \draw[dashed,gray] (0,0)--(0,1.5);
        \draw[dashed,gray] (1,0)--(1,1.5);
        \draw[dashed,gray] (2,0)--(2,1.5);
        \draw[dashed,gray] (3,0)--(3,1.5);
        \node at (1.5,-0.65) {(a)};
    \end{tikzpicture} & 
    \hspace{-0.25cm} \begin{tikzpicture}
        \node at (0,0.85) {$\equiv$};
        \node at (0,0) {};
    \end{tikzpicture} &
    \hspace{-0.25cm} \begin{tikzpicture}[scale=0.5]
        \draw[rounded corners] (-0.5, 0) rectangle (3.5, 1.5) {};
        
        \node (1) at (0,0.5) [circle, draw = black,fill=black, inner sep = 0.5mm] {};
        \node (2) at (1,0.5) [circle, draw = black,fill=black, inner sep = 0.5mm] {};
        \draw[ForestGreen](1)--(2);
        
        \node (3) at (2,1) [circle, draw = black,fill=black, inner sep = 0.5mm] {};
        \node (4) at (3,1) [circle, draw = black,fill=black, inner sep = 0.5mm] {};
        \draw[decorate sep={0.5mm}{1mm},fill, BrickRed] (3)--(4);

        \draw[dashed,gray] (0,0)--(0,1.5);
        \draw[dashed,gray] (1,0)--(1,1.5);
        \draw[dashed,gray] (2,0)--(2,1.5);
        \draw[dashed,gray] (3,0)--(3,1.5);
        \node at (1.5,-0.65) {$\textcolor{white}{(a)}$};
    \end{tikzpicture} & \hspace{0.5cm} & 
    \begin{tikzpicture}[scale=0.5]
        \draw[rounded corners] (-0.5, 0) rectangle (3.5, 1.5) {};
        
        \node (1) at (2,1) [circle, draw = black,fill=black, inner sep = 0.5mm] {};
        \node (2) at (3,1) [circle, draw = black,fill=black, inner sep = 0.5mm] {};
        \draw[ForestGreen](1)--(2);
        
        \node (3) at (0,0.5) [circle, draw = black,fill=black, inner sep = 0.5mm] {};
        \node (4) at (1,0.5) [circle, draw = black,fill=black, inner sep = 0.5mm] {};
        \draw[decorate sep={0.5mm}{1mm},fill, BrickRed] (3)--(4);

        \draw[dashed,gray] (0,0)--(0,1.5);
        \draw[dashed,gray] (1,0)--(1,1.5);
        \draw[dashed,gray] (2,0)--(2,1.5);
        \draw[dashed,gray] (3,0)--(3,1.5);
        \node at (1.5,-0.65) {(c)};
    \end{tikzpicture} & 
    \hspace{-0.25cm} \begin{tikzpicture}
        \node at (0,0.85) {$\equiv$};
        \node at (0,0) {};
    \end{tikzpicture} & 
    \hspace{-0.25cm} \begin{tikzpicture}[scale=0.5]
        \draw[rounded corners] (-0.5, 0) rectangle (3.5, 1.5) {};
        
        \node (1) at (0,1) [circle, draw = black,fill=black, inner sep = 0.5mm] {};
        \node (2) at (1,1) [circle, draw = black,fill=black, inner sep = 0.5mm] {};
        \draw[decorate sep={0.5mm}{1mm},fill, BrickRed] (1)--(2);
        
        \node (3) at (2,0.5) [circle, draw = black,fill=black, inner sep = 0.5mm] {};
        \node (4) at (3,0.5) [circle, draw = black,fill=black, inner sep = 0.5mm] {};
        \draw[ForestGreen](3)--(4);

        \draw[dashed,gray] (0,0)--(0,1.5);
        \draw[dashed,gray] (1,0)--(1,1.5);
        \draw[dashed,gray] (2,0)--(2,1.5);
        \draw[dashed,gray] (3,0)--(3,1.5);
        \node at (1.5,-0.65) {$\textcolor{white}{(c)}$};
    \end{tikzpicture} 
    & \hspace{0.5cm} & 
        \begin{tikzpicture}[scale=0.5]
        \draw[rounded corners] (-0.5, 0) rectangle (3.5, 1.5) {};
        
        \node (1) at (0,1) [circle, draw = black,fill=black, inner sep = 0.5mm] {};
        \node (2) at (1,1) [circle, draw = black,fill=black, inner sep = 0.5mm] {};
        \draw[ForestGreen](1)--(2);
        
        \node (3) at (2,0.5) [circle, draw = black,fill=black, inner sep = 0.5mm] {};
        \node (4) at (3,0.5) [circle, draw = black,fill=black, inner sep = 0.5mm] {};
        \draw[ForestGreen](3)--(4);

        \draw[dashed,gray] (0,0)--(0,1.5);
        \draw[dashed,gray] (1,0)--(1,1.5);
        \draw[dashed,gray] (2,0)--(2,1.5);
        \draw[dashed,gray] (3,0)--(3,1.5);
        \node at (1.5,-0.65) {$\textcolor{white}{(*)}$};
    \end{tikzpicture} & 
    \hspace{-0.25cm} \begin{tikzpicture}
        \node at (0,0.85) {$\equiv$};
        \node at (0,0) {};
    \end{tikzpicture} & 
    \hspace{-0.25cm} \begin{tikzpicture}[scale=0.5]
        \draw[rounded corners] (-0.5, 0) rectangle (3.5, 1.5) {};
        
        \node (1) at (0,0.5) [circle, draw = black,fill=black, inner sep = 0.5mm] {};
        \node (2) at (1,0.5) [circle, draw = black,fill=black, inner sep = 0.5mm] {};
        \draw[ForestGreen](1)--(2);
        
        \node (3) at (2,1) [circle, draw = black,fill=black, inner sep = 0.5mm] {};
        \node (4) at (3,1) [circle, draw = black,fill=black, inner sep = 0.5mm] {};
        \draw[ForestGreen](3)--(4);

        \draw[dashed,gray] (0,0)--(0,1.5);
        \draw[dashed,gray] (1,0)--(1,1.5);
        \draw[dashed,gray] (2,0)--(2,1.5);
        \draw[dashed,gray] (3,0)--(3,1.5);
        \node at (1.5,-0.65) {$\textcolor{white}{(*)}$};
    \end{tikzpicture}
    \\ 
    \begin{tikzpicture}[scale=0.5]
        \draw[rounded corners] (-0.5, 0) rectangle (3.5, 1.5) {};
        \node (1) at (0,1) [circle, draw = black,fill=black, inner sep = 0.5mm] {};
        \node (2) at (3,1) [circle, draw = black,fill=black, inner sep = 0.5mm] {};
        \draw[decorate sep={0.5mm}{1mm},fill, BrickRed] (1)--(2);

        \node (3) at (1,0.5) [circle, draw = black,fill=black, inner sep = 0.5mm] {};
        \node (4) at (2,0.5) [circle, draw = black,fill=black, inner sep = 0.5mm] {};
        \draw[decorate sep={0.5mm}{1mm},fill, BrickRed] (3)--(4);
        \draw[dashed,gray] (0,0)--(0,1.5);
        \draw[dashed,gray] (1,0)--(1,1.5);
        \draw[dashed,gray] (2,0)--(2,1.5);
        \draw[dashed,gray] (3,0)--(3,1.5);
        \node at (1.5,-0.65) {(f)};
    \end{tikzpicture} &
    \hspace{-0.25cm} \begin{tikzpicture}
        \node at (0,0.85) {$\equiv$};
        \node at (0,0) {};
    \end{tikzpicture} & 
    \hspace{-0.25cm} \begin{tikzpicture}[scale=0.5]
        \draw[rounded corners] (-0.5, 0) rectangle (3.5, 1.5) {};
        
        \node (1) at (1,1) [circle, draw = black,fill=black, inner sep = 0.5mm] {};
        \node (2) at (2,1) [circle, draw = black,fill=black, inner sep = 0.5mm] {};
        \draw[decorate sep={0.5mm}{1mm},fill, BrickRed] (1)--(2);
        
        \node (3) at (0,0.5) [circle, draw = black,fill=black, inner sep = 0.5mm] {};
        \node (4) at (3,0.5) [circle, draw = black,fill=black, inner sep = 0.5mm] {};
        \draw[decorate sep={0.5mm}{1mm},fill, BrickRed] (3)--(4);

        \draw[dashed,gray] (0,0)--(0,1.5);
        \draw[dashed,gray] (1,0)--(1,1.5);
        \draw[dashed,gray] (2,0)--(2,1.5);
        \draw[dashed,gray] (3,0)--(3,1.5);
        \node at (1.5,-0.65) {(h)};
    \end{tikzpicture} & \hspace{0.5cm} &
    \begin{tikzpicture}[scale=0.5]
        \draw[rounded corners] (-0.5, 0) rectangle (3.5, 1.5) {};
        
        \node (1) at (1,1) [circle, draw = black,fill=black, inner sep = 0.5mm] {};
        \node (2) at (2,1) [circle, draw = black,fill=black, inner sep = 0.5mm] {};
        \draw[ForestGreen](1)--(2);
        
        \node (3) at (0,0.5) [circle, draw = black,fill=black, inner sep = 0.5mm] {};
        \node (4) at (3,0.5) [circle, draw = black,fill=black, inner sep = 0.5mm] {};
        \draw[decorate sep={0.5mm}{1mm},fill, BrickRed] (3)--(4);

        \draw[dashed,gray] (0,0)--(0,1.5);
        \draw[dashed,gray] (1,0)--(1,1.5);
        \draw[dashed,gray] (2,0)--(2,1.5);
        \draw[dashed,gray] (3,0)--(3,1.5);
        \node at (1.5,-0.65) {(g)};
    \end{tikzpicture} & 
    \hspace{-0.25cm} \begin{tikzpicture}
        \node at (0,0.85) {$\equiv$};
        \node at (0,0) {};
    \end{tikzpicture} & 
    \hspace{-0.25cm} \begin{tikzpicture}[scale=0.5]
        \draw[rounded corners] (-0.5, 0) rectangle (3.5, 1.5) {};
        
        \node (1) at (1,0.5) [circle, draw = black,fill=black, inner sep = 0.5mm] {};
        \node (2) at (2,0.5) [circle, draw = black,fill=black, inner sep = 0.5mm] {};
        \draw[ForestGreen](1)--(2);
        
        \node (3) at (0,1) [circle, draw = black,fill=black, inner sep = 0.5mm] {};
        \node (4) at (3,1) [circle, draw = black,fill=black, inner sep = 0.5mm] {};
        \draw[decorate sep={0.5mm}{1mm},fill, BrickRed] (3)--(4);

        \draw[dashed,gray] (0,0)--(0,1.5);
        \draw[dashed,gray] (1,0)--(1,1.5);
        \draw[dashed,gray] (2,0)--(2,1.5);
        \draw[dashed,gray] (3,0)--(3,1.5);
        \node at (1.5,-0.65) {$\textcolor{white}{(g)}$};
    \end{tikzpicture}
    &\hspace{0.5cm} &
\begin{tikzpicture}[scale=0.5]
        \draw[rounded corners] (-0.5, 0) rectangle (3.5, 1.5) {};
        
        \node (1) at (1,0.5) [circle, draw = black,fill=black, inner sep = 0.5mm] {};
        \node (2) at (2,0.5) [circle, draw = black,fill=black, inner sep = 0.5mm] {};
        \draw[ForestGreen](1)--(2);
        
        \node (3) at (0,1) [circle, draw = black,fill=black, inner sep = 0.5mm] {};
        \node (4) at (3,1) [circle, draw = black,fill=black, inner sep = 0.5mm] {};
        \draw[ForestGreen](3)--(4);

        \draw[dashed,gray] (0,0)--(0,1.5);
        \draw[dashed,gray] (1,0)--(1,1.5);
        \draw[dashed,gray] (2,0)--(2,1.5);
        \draw[dashed,gray] (3,0)--(3,1.5);
        \node at (1.5,-0.65) {$\textcolor{white}{(*)}$};
    \end{tikzpicture}
    & \hspace{-0.25cm}\begin{tikzpicture}
        \node at (0,0.85) {$\equiv$};
        \node at (0,0) {};
    \end{tikzpicture} & 
    \hspace{-0.25cm} \begin{tikzpicture}[scale=0.5]
        \draw[rounded corners] (-0.5, 0) rectangle (3.5, 1.5) {};
        
        \node (1) at (1,1) [circle, draw = black,fill=black, inner sep = 0.5mm] {};
        \node (2) at (2,1) [circle, draw = black,fill=black, inner sep = 0.5mm] {};
        \draw[ForestGreen](1)--(2);
        
        \node (3) at (0,0.5) [circle, draw = black,fill=black, inner sep = 0.5mm] {};
        \node (4) at (3,0.5) [circle, draw = black,fill=black, inner sep = 0.5mm] {};
        \draw[ForestGreen](3)--(4);

        \draw[dashed,gray] (0,0)--(0,1.5);
        \draw[dashed,gray] (1,0)--(1,1.5);
        \draw[dashed,gray] (2,0)--(2,1.5);
        \draw[dashed,gray] (3,0)--(3,1.5);
        \node at (1.5,-0.65) {$\textcolor{white}{(*)}$};
    \end{tikzpicture}
    \end{array}
    $$
    \caption{Equivalent non-zero compositions of degree two}
    \label{fig:deg2equivrels}
\end{figure} 

To provide an example of how one can approach establishing an equivalence between non-zero compositions of operators, we include the following lemma with proof, which corresponds to Proposition~\ref{prop:deg2rel}~(i).

\begin{lemma}
        If $a<b<c<d$, then $\mathbf{v}_{dc}\mathbf{v}_{ab}\equiv \mathbf{v}_{ab}\mathbf{v}_{dc}$.
    \end{lemma}
    \begin{proof}
        Let $\mathbf{v}=\mathbf{v}_{dc}\mathbf{v}_{ab}$. Assume that $\mathbf{v}\bullet_kw\neq 0$, $w^{-1}(a)=i_1$, $w^{-1}(b)=j_1$, $w^{-1}(c)=i_2$, and $w^{-1}(d)=j_2$. Then, since $\mathbf{v}\bullet_kw\neq 0$, it follows that $\mathbf{v}_{ab}\bullet_kw\neq 0$. Applying Proposition~\ref{prop:covercond}, $\mathbf{v}_{ab}\bullet_kw\neq 0$ if and only if  $i_1\le k<j_1$ and, for $l\in (a,b)$, $w^{-1}(l)<i_1$ or $w^{-1}(l)>j_1$.
        Since $c,d\notin (a,b)$, we get no additional restrictions on $i_1,j_1,i_2,$ and $j_2$. Note that in $\mathbf{v}_{ab}\bullet_kw$ we have $a$ in position $j_1>k$, $b$ in position $i_1\le k$, $c$ in position $i_2$, and $d$ in position $j_2$. Now, since $$\mathbf{v}\bullet_kw=\mathbf{v}_{dc}\bullet_k(\mathbf{v}_{ab}\bullet_kw)\neq 0,$$ applying Proposition~\ref{prop:covercond} once again, we find that $j_2\le k<i_2$ and, for $l\in\mathbb{Z}_{>0}\backslash [c,d]$, $w^{-1}(l)<j_2$ or $w^{-1}(l)>i_2$. 
        As $a,b\in \mathbb{Z}_{>0}\backslash [c,d]$, it follows that $i_1<j_2\le k<i_2<j_1$. Thus, all together, we have that $\mathbf{v}\bullet_kw\neq 0$ if and only if 
        \begin{itemize}
            \item $i_1<j_2\le k<i_2<j_1$,
            \item $w^{-1}(l)<i_1$ or $w^{-1}(l)>j_1$ for $l\in (a,b)$, and
            \item $w^{-1}(l)<j_2$ or $w^{-1}(l)>i_2$ for $l\in\mathbb{Z}_{>0}\backslash [c,d]$.
        \end{itemize}
        Moreover, $\mathbf{v}\bullet_kw$ has a coefficient of $q_{j_2i_2}$, $a$ in position $j_1$, $b$ in position $i_1$, $c$ in position $j_2$, $d$ in position $i_2$, and all other entries of $w$ fixed.

        A similar analysis for $\mathbf{v}^*=\mathbf{v}_{ab}\mathbf{v}_{dc}$ shows that
        $\mathbf{v}^*\bullet_kw\neq 0$ if and only if 
        \begin{itemize}
            \item $i_1<j_2\le k<i_2<j_1$,
            \item $w^{-1}(l)<j_2$ or $w^{-1}(l)>i_2$ for $l\in\mathbb{Z}_{>0}\backslash [c,d]$, and
            \item $w^{-1}(l)<i_1$ or $w^{-1}(l)>j_1$ for $l\in (a,b)$;
        \end{itemize}
        that is, the exact same conditions required for $\mathbf{v}\bullet_kw\neq 0$. 
        Moreover, as in $\mathbf{v}\bullet_kw$, one finds that $\mathbf{v}^*\bullet_kw$ has a coefficient of $q_{j_2i_2}$, $a$ in position $j_1$, $b$ in position $i_1$, $c$ in position $j_2$, $d$ in position $i_2$, and all other entries of $w$ fixed; that is, when nontrivial, $\mathbf{v}$ and $\mathbf{v}^*$ have the same effect. The result follows.
    \end{proof}

It turns out that these are all the equivalences involving non-zero compositions of degree two.
\begin{theorem}
    The list of equivalences of non-zero compositions of degree two presented in Proposition~\ref{prop:deg2rel} is complete; that is, there are no other equivalences of non-zero compositions of only two operators.
\end{theorem}

\begin{proof}
We prove this result by analyzing the remaining non-zero compositions of two operators and showing that their action cannot be the same for all permutations. 

Let $a_1<b_1$, $a_2<b_2<c_2$, and $a_3<b_3<c_3<d_3$. We need to show that no two operators of the following set are equivalent:
\begin{align*}
S=\big\{ &\mathbf{v}_{a_2b_2}\mathbf{v}_{b_2c_2},
\mathbf{v}_{b_2c_2}\mathbf{v}_{a_2b_2},
\mathbf{v}_{a_3b_3}\mathbf{v}_{c_3d_3},
\mathbf{v}_{a_3d_3}\mathbf{v}_{b_3c_3}, \mathbf{v}_{a_1b_1}\mathbf{v}_{b_1a_1}, \mathbf{v}_{a_1b_1}\mathbf{v}_{b_1a_1},
\mathbf{v}_{a_2b_2}\mathbf{v}_{c_2a_2},
\mathbf{v}_{b_2c_2}\mathbf{v}_{c_2a_2}, \\
& \mathbf{v}_{c_2a_2}\mathbf{v}_{b_2c_2}, \mathbf{v}_{c_2a_2}\mathbf{v}_{a_2b_2},
\mathbf{v}_{a_3b_3}\mathbf{v}_{d_3c_3},
\mathbf{v}_{b_3c_3}\mathbf{v}_{d_3a_3},
\mathbf{v}_{b_3a_3}\mathbf{v}_{c_3d_3},
\mathbf{v}_{d_3a_3}\mathbf{v}_{c_3b_3}\big\}.
\end{align*} 
Considering Lemma~\ref{lem:supp1} and Proposition~\ref{prop:supp2}, we split $S$ into several subsets according to their support and the number of quantum operators involved.
 \begin{equation*}
     \begin{array}{lcl}
S_1=\{\mathbf{v}_{a_2b_2}\mathbf{v}_{b_2c_2},\mathbf{v}_{b_2c_2}\mathbf{v}_{a_2b_2}\} & \qquad &
S_2=\{\mathbf{v}_{a_3b_3}\mathbf{v}_{c_3d_3},\mathbf{v}_{a_3d_3}\mathbf{v}_{b_3c_3}\} \\
S_3=\{\mathbf{v}_{a_1b_1}\mathbf{v}_{b_1a_1}, \mathbf{v}_{a_1b_1}\mathbf{v}_{b_1a_1}\} & \qquad &
S_4=\{\mathbf{v}_{a_2b_2}\mathbf{v}_{c_2a_2},\mathbf{v}_{b_2c_2}\mathbf{v}_{c_2a_2}, \mathbf{v}_{c_2a_2}\mathbf{v}_{b_2c_2},\mathbf{v}_{c_2a_2}\mathbf{v}_{a_2b_2}\} \\
S_5=\{\mathbf{v}_{a_3b_3}\mathbf{v}_{d_3c_3},\mathbf{v}_{b_3c_3}\mathbf{v}_{d_3a_3},\mathbf{v}_{b_3a_3}\mathbf{v}_{c_3d_3}\} & \qquad & 
S_6=\{\mathbf{v}_{d_3a_3}\mathbf{v}_{c_3b_3}\}.
    \end{array}
 \end{equation*}

We analyze each set $S_i$, for $i\in [5]$.
\bigskip

\noindent
\underline{Case 1:} 
Let $\mathbf{v}_1=\mathbf{v}_{a_2b_2}\mathbf{v}_{b_2c_2}$ and $\mathbf{v}_2=\mathbf{v}_{b_2c_2}\mathbf{v}_{a_2b_2}$ in $S_1$. If $\mathbf{v}_1\bullet_kw\neq 0$, then $w(i)=b_2$ implies $i\le k$. On the other hand, if $\mathbf{v}_2\bullet_kw\neq 0$, then $w(i)=b_2$ implies $i> k$. Thus, $\mathbf{v}_1\not\equiv\mathbf{v}_2$ and no two compositions of $S_1$ are equivalent.
\bigskip

\noindent
\underline{Case 2:} 
Let $\mathbf{v}_1=\mathbf{v}_{a_3b_3}\mathbf{v}_{c_3d_3}$ and $\mathbf{v}_2=\mathbf{v}_{a_3d_3}\mathbf{v}_{b_3c_3}$ in $S_2$. If $\mathbf{v}_1\bullet_kw\neq 0$, then $w(i)=c_3$ implies $i\le k$. On the other hand, if $\mathbf{v}_2\bullet_kw\neq 0$, then $w(i)=c_3$ implies $i> k$. Thus, $\mathbf{v}_1\not\equiv\mathbf{v}_2$ and no two compositions of $S_2$ are equivalent.
\bigskip

\noindent
\underline{Case 3:} 
Let $\mathbf{v}_1=\mathbf{v}_{a_1b_1}\mathbf{v}_{b_1a_1}$ and $\mathbf{v}_2=\mathbf{v}_{b_1a_1}\mathbf{v}_{a_1b_1}$ in $S_3$. Note that if $\mathbf{v}_1\bullet_kw\neq 0$, then $w(i)=a_1$ implies $i>k$. On the other hand, if $\mathbf{v}_2\bullet_kw\neq 0$, then $w(i)=a_1$ implies $i\le k$. Thus, $\mathbf{v}_1\not\equiv\mathbf{v}_2$ and no two compositions of $S_3$ are equivalent.
\bigskip

\noindent
\underline{Case 4:} 
Let $\mathbf{v}_1=\mathbf{v}_{a_2b_2}\mathbf{v}_{c_2a_2}$, $\mathbf{v}_2=\mathbf{v}_{b_2c_2}\mathbf{v}_{c_2a_2}$, $\mathbf{v}_3=\mathbf{v}_{c_2a_2}\mathbf{v}_{b_2c_2}$, and $\mathbf{v}_4=\mathbf{v}_{c_2a_2}\mathbf{v}_{a_2b_2}$ in $S_4$. If $\mathbf{v}_4\bullet_kw\neq 0$, then $w(i)=a_2$ implies $i\le k$. On the other hand, if $\mathbf{v}_1\bullet_kw\neq 0$, $\mathbf{v}_2\bullet_kw\neq 0$, or $\mathbf{v}_3\bullet_kw\neq 0$, then $w(i)=a_2$ implies $i>k$. Consequently, $\mathbf{v}_4\not\equiv \mathbf{v}_1,\mathbf{v}_2,\mathbf{v}_3$. Now, if $\mathbf{v}_3\bullet_kw\neq 0$, then $w(i)=c_2$ implies $i>k$. On the other hand, if $\mathbf{v}_1\bullet_kw\neq 0$ or $\mathbf{v}_2\bullet_kw\neq 0$, then $w(i)=c_2$ implies $i\le k$. Consequently, $\mathbf{v}_3\not\equiv \mathbf{v}_1,\mathbf{v}_2$. Finally, if $\mathbf{v}_1\bullet_kw\neq 0$, then $w(i)=b_2$ implies $i>k$. On the other hand, if $\mathbf{v}_2\bullet_kw\neq 0$, then $w(i)=b_2$ implies $i\le k$. Consequently, $\mathbf{v}_1\not\equiv\mathbf{v}_2$. Thus, no two compositions of $S_4$ are equivalent.
\bigskip

\noindent
\underline{Case 5:} 
Let $\mathbf{v}_1=\mathbf{v}_{a_3b_3}\mathbf{v}_{d_3c_3}$, $\mathbf{v}_2=\mathbf{v}_{b_3c_3}\mathbf{v}_{d_3a_3}$, and $\mathbf{v}_3=\mathbf{v}_{b_3a_3}\mathbf{v}_{c_3d_3}$ in $S_5$. If $\mathbf{v}_1\bullet_kw\neq 0$, then $w(i)=b_3$ implies $i>k$. On the other hand, if $\mathbf{v}_2\bullet_kw\neq 0$ or $\mathbf{v}_3\bullet_kw\neq 0$, then $w(i)=b_3$ implies $i\le k$. Consequently, $\mathbf{v}_1\not\equiv\mathbf{v}_2,\mathbf{v}_3$. Now, if $\mathbf{v}_2\bullet_kw\neq 0$, then $w(i)=d_3$ implies $d_3\le k$. On the other hand, if $\mathbf{v}_3\bullet_kw\neq 0$, then $w(i)=d_3$ implies $d_3>k$. Consequently, $\mathbf{v}_2\not\equiv\mathbf{v}_3$. Thus, no two compositions of $S_5$ are equivalent.
\end{proof}

Our next result shows that if adding an operator to a non-zero composition of operators with no shared support values results in a zero composition, then it is the result of a zero equivalence of degree two.

\begin{prop}\label{prop:dis0}
    Let $\mathbf{v}=\mathbf{v}_{a_nb_n}\cdots\mathbf{v}_{a_1b_1}\not\equiv \mathbf{0}$. Suppose that $a_{n+1},b_{n+1}\in\mathbb{Z}_{>0}$ satisfy $a_{n+1}\neq b_{n+1}$ and $a_{n+1},b_{n+1}\notin \{a_1,b_1,\hdots,a_n,b_n\}$. Then $\mathbf{v}_{a_{n+1}b_{n+1}}\mathbf{v}\equiv \mathbf{0}$ if and only if $\mathbf{v}_{a_{n+1}b_{n+1}}\mathbf{v}_{a_ib_i}\equiv\mathbf{0}$ for some $1\le i\le n$. 
\end{prop}
\begin{proof}
    Note that since $\{a_{n+1},b_{n+1}\}\cap\{a_i,b_i\}=\emptyset$ for all $i\in [n]$, combining Theorem~\ref{thm:deg2zero}~(ii),~(iv),~(vi) and Proposition~\ref{prop:deg2rel} we have that 
    $$(\star) \qquad \text{for each } i\in[n], \quad \mathbf{v}_{a_{n+1}b_{n+1}}\mathbf{v}_{a_ib_i}\equiv \mathbf{v}_{a_ib_i}\mathbf{v}_{a_{n+1}b_{n+1}}.$$

    The backward direction is an immediate consequence of $(\star)$. For the forward direction, assume that $\mathbf{v}_{a_{n+1}b_{n+1}}\mathbf{v}_{a_ib_i}\not\equiv\mathbf{0}$ for all $1\le i\le n$. We construct a permutation $w^*$ for which there exists $k^*\in\mathbb{Z}_{>0}$ such that $\mathbf{v}_{a_{n+1}b_{n+1}}\mathbf{v}\bullet_{k^*}w^*\neq 0$. Considering Theorem~\ref{thm:cycshift}~(a), without loss of generality, we may assume that $a_{n+1}>b_{n+1}$. Then, since $\mathbf{v}_{a_{n+1}b_{n+1}}\mathbf{v}_{a_ib_i}\not\equiv\mathbf{0}$ for all $1\le i\le n$, by Theorem~\ref{thm:deg2zero}~(ii),~(iv),~(vi), it must be the case that for $i\in [n]$,
    \begin{enumerate}
        \item[(1)] if $a_i>b_i$, then either $b_i<b_{n+1}<a_{n+1}<a_i$ or $b_{n+1}<b_i<a_i<a_{n+1}$; and
        \item[(2)] if $a_i<b_i$, then either $b_{n+1}<a_i<b_i<a_{n+1}$, $a_i<b_i<b_{n+1}<a_{n+1}$, or $b_{n+1}<a_{n+1}<a_i<b_i$;
    \end{enumerate}
    that is, 
    $$(\star\star) \qquad \text{for each } i\in [n] \text{, either } \quad a_i,b_i\in (b_{n+1},a_{n+1}) \quad \text{ or } \quad a_i,b_i\notin [b_{n+1},a_{n+1}].$$

    Since $\mathbf{v}\not\equiv 0$, if $|\mathrm{supp}(\mathbf{v})|=t$, then there exists $u\in S_t$ and $k\in\mathbb{Z}_{>0}$ such that $\mathbf{\underline{v}}\bullet_ku\neq 0$. Assume $u=(u'_1,\hdots ,u'_t)$. If $\phi:\mathrm{supp}(\mathbf{v})\to [t]$ denotes the unique order-preserving bijection, then let $(u_1,\hdots,u_t)=(\phi^{-1}(u_1'),\hdots,\phi^{-1}(u_t'))$. Letting $N=\max\{a_1,b_1,\hdots,a_{n+1},b_{n+1}\}$, define $w,w^*\in S_N$ as follows. Set $w=(w_1,\hdots,w_N)$ where $w_i=u_i$ for $i\in [t]$ and the values $w_i$ for $i>t$ are the values $[N]\backslash\{a_1,b_1,\hdots,a_n,b_n\}$ listed in increasing order. Note that since $\mathbf{\underline{v}}\bullet_ku\neq 0$, evidently, $\mathbf{v}\bullet_kw\neq 0$. Now, for $w^*$, let $i^*=1+\max\{i~|~i\le k,~u_i\notin [b_{n+1},a_{n+1}]\}$ if such $i\le k$ exists, or $i^*=1$, otherwise, and $j^*=1+\min\{j~|~j> k,~u_j\notin [b_{n+1},a_{n+1}]\}$ if such a value $j>k$ exists, or $j^*=t+2$ otherwise. Then, we define $w^*=(w'_1,\hdots, w'_N)$ where $w'_i=w_i$ for $1\le i<i^*$, $w'_{i^*}=a_{n+1}$, $w'_i=w_{i-1}$ for $i^*<i<j^*$, $w'_{j^*}=b_{n+1}$, and $w'_i=w_i$ for $t+2<i\le N$. Note that, by construction, $w'_i=u_{i-1}\in [b_{n+1},a_{n+1}]$ for $i^*+1\le i\le j^*-1$. Ongoing, we assume that $i^*>1$ and $j^*<t+2$; the arguments in the other cases are similar, but simpler. In this case, we have $$w^*=(u_1,\hdots,u_{i^*-1},a_{n+1},u_{i^*},\hdots,u_{j^*-2},b_{n+1},u_{j^*-1},\hdots,u_t,w_{t+3},\hdots,w_N).$$ We claim that taking $k^*=k+1$ we have $\mathbf{v}_{a_{n+1}b_{n+1}}\mathbf{v}\bullet_{k^*}w^*\neq 0$.
    Assume otherwise. By construction, we have that 
    \begin{align*}
        \mathbf{v}_{a_{n+1}b_{n+1}}\bullet_{k^*}w^*&=q_{i^*,j^*}(u_1,\hdots,u_{i^*-1},b_{n+1},u_{i^*},\hdots,u_{j^*-1},a_{n+1},u_{j^*},\hdots,u_t,w_{t+3},\hdots,w_N)\\
        &=q_{i^*,j^*}w'\neq 0.
    \end{align*}
    Thus, since $\mathbf{v}_{a_{n+1}b_{n+1}}\mathbf{v}\equiv \mathbf{v}\mathbf{v}_{a_{n+1}b_{n+1}}$ by $(\star)$, it follows that $\mathbf{v}\bullet_{k^*}w'=0$. Let $m\in [n]$ be least such that $\mathbf{v}_{a_mb_m}\cdots\mathbf{v}_{a_1b_1}\bullet_{k^*}w'= 0$. By construction, $\mathbf{v}_{a_mb_m}\cdots\mathbf{v}_{a_1b_1}\bullet_{k^*}w'= 0$ must be a consequence of the positions of $b_{n+1}$ and $a_{n+1}$ in $w'$. In particular, letting $\mathbf{q^{\alpha}}\widehat{w}=\mathbf{v}_{a_{m-1}b_{m-1}}\cdots\mathbf{v}_{a_1b_1}\bullet_{k^*}w^*\neq 0$, since the positions of $a_{n+1}$ and $b_{n+1}$ are fixed by $\mathbf{v}_{a_ib_i}$ for all $i\in [n]$, applying Proposition~\ref{prop:covercond} there are two cases.
    \bigskip

    \noindent
    \underline{Case 1:} $a_m<b_m$ and either
    \begin{itemize}
        \item $a_m<a_{n+1}<b_m$ with $\widehat{w}^{-1}(a_m)<j^*<\widehat{w}^{-1}(b_m)$ or
        \item $a_m<b_{n+1}<b_m$ with $\widehat{w}^{-1}(a_m)<i^*<\widehat{w}^{-1}(b_m)$.
    \end{itemize}
     By (2) above, $a_{n+1},b_{n+1}\notin [a_m,b_m]$, and so neither option is possible.
    \bigskip

    \noindent
    \underline{Case 2:} $a_m>b_m$ and either 
    \begin{itemize}
        \item $a_{n+1}\notin [b_m,a_m]$ with $\widehat{w}^{-1}(a_m)<j^*<\widehat{w}^{-1}(b_m)$ or
        \item $b_{n+1}\notin [b_m,a_m]$ with $\widehat{w}^{-1}(a_m)<i^*<\widehat{w}^{-1}(b_m)$.
    \end{itemize}
    By (2) above, in either case, we must have $[b_m,a_m]\subset (b_{n+1},a_{n+1})$. Assume that $a_{n+1}\notin [b_m,a_m]$ with $\widehat{w}^{-1}(a_m)<j^*<\widehat{w}^{-1}(b_m)$, as the other case follows via a similar argument. Since $w^*(j^*+1)=u_{j^*-1}\notin [b_{n+1},a_{n+1}]$, it follows that $w^*(j^*+1)\notin [b_m,a_m]$ and, considering $(\star\star)$, we have $\widehat{w}(j^*+1)\notin [b_m,a_m]$. Consequently, $\widehat{w}^{-1}(a_m)<j^*+1<\widehat{w}^{-1}(b_m)$. Letting $\mathbf{q^{\beta}}w''=\mathbf{v}_{a_{m-1}b_{m-1}}\cdots\mathbf{v}_{a_1b_1}\bullet_kw$, this implies that $(w'')^{-1}(b_m)<j^*-1<(w'')^{-1}(a_m)$ where $w''(j^*-1)\notin [b_m,a_m]$; but then $\mathbf{v}_{a_mb_m}\bullet_kw''=0$, i.e., $\mathbf{v}\bullet_kw=0$, which is a contradiction. Thus, $\mathbf{v}_{a_{n+1}b_{n+1}}\mathbf{v}\bullet_{k^*}w^*\neq 0$, as desired, and the result follows.
\end{proof}

We finish this subsection with a result concerning compositions equivalent to ${\bf 0}$, their support, and implications related to minimality. The main motivation for these results is computational. As mentioned in the introduction, our goal is to fully characterize the minimal zero and non-zero equivalences of all degrees, and computer testing is an essential tool for that. However, due to the sizes of the samples and the equalities that need to be checked, having results that allow us to restrict the sizes of supports of compositions that need to be considered is very helpful. 

\begin{lemma}\label{lem:help1}
    Let $n>1$ and $\mathbf{v}=\mathbf{v}_{a_nb_n}\cdots\mathbf{v}_{a_1b_1}$ with either
    \begin{enumerate}
        \item[$(a)$] $|\supp(\mathbf{v})|=2n$ or
        \item[$(b)$] $|\supp(\mathbf{v})|=2n-1$.
    \end{enumerate}
    If $\mathbf{v}\equiv \mathbf{0}$, then $\mathbf{v}\equiv\mathbf{v}_{a'_nb'_n}\cdots\mathbf{v}_{a'_1b'_1}$ where $\mathbf{v}_{a'_{i+1}b'_{i+1}}\mathbf{v}_{a'_ib'_i}\equiv \mathbf{0}$ for some $1\le i<n$, i.e., $\mathbf{v}\equiv \mathbf{0}$ as a consequence of a degree-two zero equivalence.
\end{lemma}
\begin{proof}
    We treat each case separately by induction on $n$.
    \begin{enumerate}
        \item[$(a)$] The base case is immediate. Assume that the result holds for $n-1\ge 1$. Note that $|\supp(\mathbf{v})|=2n$ implies $\{a_i,b_i\}\cap\{a_j,b_j\}=\emptyset$ for all $i,j\in [n]$ such that $i\neq j$. Consequently, combining Theorem~\ref{thm:deg2zero}~(ii),~(iv),~(vi) and Proposition~\ref{prop:deg2rel} we have that 
    $$\textup{$(\star)\qquad$ $\mathbf{v}_{a_ib_i}\mathbf{v}_{a_jb_j}\equiv \mathbf{v}_{a_jb_j}\mathbf{v}_{a_ib_i}$ for all $i,j\in [n]$ such that $i\neq j$.}$$ 
    Now, if $\mathbf{v}=\mathbf{v}_{a_nb_n}\cdots\mathbf{v}_{a_1b_1}\equiv 0$, then either
    \begin{itemize}
        \item $\mathbf{v}_{a_{n-1}b_{n-1}}\cdots\mathbf{v}_{a_1b_1}\equiv \mathbf{0}$, and the result follows from the induction hypothesis; or
        \item $\mathbf{v}_{a_{n-1}b_{n-1}}\cdots\mathbf{v}_{a_1b_1}\not\equiv \mathbf{0}$, and the result follows by Proposition~\ref{prop:dis0} along with $(\star)$.
    \end{itemize}
    
    \item[$(b)$] As in (a), the base case is immediate. Assume the result holds for $n-1\ge 1$. Note that $|\supp(\mathbf{v})|=2n-1$ implies that
    \begin{itemize}
        \item there exists unique pair $i^*, j^*\in [n]$ such that $i^*< j^*$ and $|\{a_{i^*},b_{i^*}\}\cap\{a_{j^*},b_{j^*}\}|=1$, and
        \item for all other pairs $i,j\in [n]$ such that $i\neq j$, we have $\{a_i,b_i\}\cap\{a_j,b_j\}=\emptyset$.
    \end{itemize}
    Consequently, if $i\notin \{i^*,j^*\}$, then $\{a_i,b_i\}\cap\{a_j,b_j\}=\emptyset$ for all $j\in [n]$ such that $j\neq i$. Thus, combining Theorem~\ref{thm:deg2zero}~(ii),~(iv),~(vi) and Proposition~\ref{prop:deg2rel} we have that 
    $$\text{$(\star\star)$ if $i\notin \{i^*,j^*\}$, then $\mathbf{v}_{a_ib_i}\mathbf{v}_{a_jb_j}\equiv \mathbf{v}_{a_jb_j}\mathbf{v}_{a_ib_i}$ for all $j\in [n]$ such that $i\neq j$.}$$ Consequently, without loss of generality, we may assume that $i^*=1$ and $j^*=2$. Now, if $\mathbf{v}=\mathbf{v}_{a_nb_n}\cdots\mathbf{v}_{a_1b_1}\equiv 0$, then either
    \begin{itemize}
        \item $\mathbf{v}_{a_{n-1}b_{n-1}}\cdots\mathbf{v}_{a_1b_1}\equiv \mathbf{0}$, and the result follows from the inductive hypothesis; or
        \item $\mathbf{v}_{a_{n-1}b_{n-1}}\cdots\mathbf{v}_{a_1b_1}\not\equiv \mathbf{0}$, and the result follows by Proposition~\ref{prop:dis0} along with $(\star\star)$. \qedhere
    \end{itemize}
    \end{enumerate}
\end{proof}

\subsection{Degree 3 Equivalences}\label{sec:order3}

In this section, we determine all minimal equivalences of degree three. We follow the same approach as for degree two equivalences. Consequently, we start with zero equivalences of degree three and, considering Theorem~\ref{thm:cycshift}~(a), focus on those for which the first operator is quantum.

Utilizing~\cite[Proposition 5.2]{qkB1} along with properties of the cyclic shift transformation, we are led to the following.

\begin{prop}\label{prop:min3zero1}
    Let $a<b<c<d$. Then the following are minimal zero equivalences of degree three:
    \begin{multicols}{2}
        \begin{enumerate}
            \item[\textup{(i)}] $\mathbf{v}_{ca}\mathbf{v}_{bc}\mathbf{v}_{ca}\equiv \mathbf{0}$; 
            \item[\textup{(ii)}] $\mathbf{v}_{ca}\mathbf{v}_{ab}\mathbf{v}_{ca}\equiv \mathbf{0}$; 
            \item[\textup{(iii)}] $\mathbf{v}_{bd}\mathbf{v}_{ab}\mathbf{v}_{ca}\equiv \mathbf{0}$; 
            \item[\textup{(iv)}] $\mathbf{v}_{ca}\mathbf{v}_{bc}\mathbf{v}_{db}\equiv \mathbf{0}$; 
            \item[\textup{(v)}]  $\mathbf{v}_{db}\mathbf{v}_{bc}\mathbf{v}_{ca}\equiv \mathbf{0}$; 
            \item[\textup{(vi)}] $\mathbf{v}_{ac}\mathbf{v}_{cd}\mathbf{v}_{db}\equiv \mathbf{0}$. 
        \end{enumerate}  
    \end{multicols}
\end{prop}

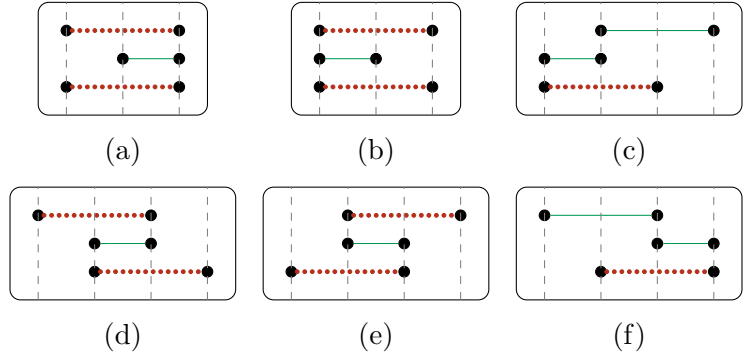
\begin{figure}[H]
    \centering
$$
\begin{array}{ccc}
\begin{tikzpicture}[scale=0.75]
        \draw[rounded corners] (-0.5, 0) rectangle (2.5, 2) {};
    
        \node (6) at (0,1.5) [circle, draw = black,fill=black, inner sep = 0.5mm] {};
        \node (7) at (2,1.5) [circle, draw = black,fill=black, inner sep = 0.5mm] {};
        \draw[decorate sep={0.5mm}{1mm},fill, BrickRed] (6)--(7);

        \node (12) at (1,1) [circle, draw = black,fill=black, inner sep = 0.5mm] {};
        \node (13) at (2,1) [circle, draw = black,fill=black, inner sep = 0.5mm] {};
        \draw[ForestGreen](12)--(13);
        
        \node (14) at (0,0.5) [circle, draw = black,fill=black, inner sep = 0.5mm] {};
        \node (15) at (2,0.5) [circle, draw = black,fill=black, inner sep = 0.5mm] {};
        \draw[decorate sep={0.5mm}{1mm},fill, BrickRed] (14)--(15);

        \draw[dashed,gray] (0,0)--(0,2);
        \draw[dashed,gray] (1,0)--(1,2);
        \draw[dashed,gray] (2,0)--(2,2);
        \node at (1, -0.65) {(a)};
    \end{tikzpicture} & \begin{tikzpicture}[scale=0.75]
        \draw[rounded corners] (-0.5, 0) rectangle (2.5, 2) {};
    
        \node (6) at (0,1.5) [circle, draw = black,fill=black, inner sep = 0.5mm] {};
        \node (7) at (2,1.5) [circle, draw = black,fill=black, inner sep = 0.5mm] {};
        \draw[decorate sep={0.5mm}{1mm},fill, BrickRed] (6)--(7);

        \node (12) at (0,1) [circle, draw = black,fill=black, inner sep = 0.5mm] {};
        \node (13) at (1,1) [circle, draw = black,fill=black, inner sep = 0.5mm] {};
        \draw[ForestGreen](12)--(13);
        
        \node (14) at (0,0.5) [circle, draw = black,fill=black, inner sep = 0.5mm] {};
        \node (15) at (2,0.5) [circle, draw = black,fill=black, inner sep = 0.5mm] {};
        \draw[decorate sep={0.5mm}{1mm},fill, BrickRed] (14)--(15);

        \draw[dashed,gray] (0,0)--(0,2);
        \draw[dashed,gray] (1,0)--(1,2);
        \draw[dashed,gray] (2,0)--(2,2);
        \node at (1, -0.65) {(b)};
    \end{tikzpicture} & \begin{tikzpicture}[scale=0.75]
        \draw[rounded corners] (-0.5, 0) rectangle (3.5, 2) {};
        
        \node (1) at (0,0.5) [circle, draw = black,fill=black, inner sep = 0.5mm] {};
        \node (2) at (2,0.5) [circle, draw = black,fill=black, inner sep = 0.5mm] {};
        \draw[decorate sep={0.5mm}{1mm},fill, BrickRed] (1)--(2);
        
        \node (3) at (0,1) [circle, draw = black,fill=black, inner sep = 0.5mm] {};
        \node (4) at (1,1) [circle, draw = black,fill=black, inner sep = 0.5mm] {};
        \draw[ForestGreen](3)--(4);
        
        \node (5) at (1,1.5) [circle, draw = black,fill=black, inner sep = 0.5mm] {};
        \node (6) at (3,1.5) [circle, draw = black,fill=black, inner sep = 0.5mm] {};
        \draw[ForestGreen](5)--(6);

        \draw[dashed,gray] (0,0)--(0,2);
        \draw[dashed,gray] (1,0)--(1,2);
        \draw[dashed,gray] (2,0)--(2,2);
        \draw[dashed,gray] (3,0)--(3,2);
        \node at (1.5, -0.65) {(c)};
    \end{tikzpicture} \\
    \begin{tikzpicture}[scale=0.75]
        \draw[rounded corners] (-0.5, 0) rectangle (3.5, 2) {};
        
        \node (1) at (1,0.5) [circle, draw = black,fill=black, inner sep = 0.5mm] {};
        \node (2) at (3,0.5) [circle, draw = black,fill=black, inner sep = 0.5mm] {};
        \draw[decorate sep={0.5mm}{1mm},fill, BrickRed] (1)--(2);
        
        \node (3) at (1,1) [circle, draw = black,fill=black, inner sep = 0.5mm] {};
        \node (4) at (2,1) [circle, draw = black,fill=black, inner sep = 0.5mm] {};
        \draw[ForestGreen](3)--(4);
        
        \node (5) at (0,1.5) [circle, draw = black,fill=black, inner sep = 0.5mm] {};
        \node (6) at (2,1.5) [circle, draw = black,fill=black, inner sep = 0.5mm] {};
        \draw[decorate sep={0.5mm}{1mm},fill, BrickRed] (5)--(6);

        \draw[dashed,gray] (0,0)--(0,2);
        \draw[dashed,gray] (1,0)--(1,2);
        \draw[dashed,gray] (2,0)--(2,2);
        \draw[dashed,gray] (3,0)--(3,2);
        \node at (1.5, -0.65) {(d)};
    \end{tikzpicture} & \begin{tikzpicture}[scale=0.75]
        \draw[rounded corners] (-0.5, 0) rectangle (3.5, 2) {};
        
        \node (1) at (1,1.5) [circle, draw = black,fill=black, inner sep = 0.5mm] {};
        \node (2) at (3,1.5) [circle, draw = black,fill=black, inner sep = 0.5mm] {};
        \draw[decorate sep={0.5mm}{1mm},fill, BrickRed] (1)--(2);
        
        \node (3) at (1,1) [circle, draw = black,fill=black, inner sep = 0.5mm] {};
        \node (4) at (2,1) [circle, draw = black,fill=black, inner sep = 0.5mm] {};
        \draw[ForestGreen](3)--(4);
        
        \node (5) at (0,0.5) [circle, draw = black,fill=black, inner sep = 0.5mm] {};
        \node (6) at (2,0.5) [circle, draw = black,fill=black, inner sep = 0.5mm] {};
        \draw[decorate sep={0.5mm}{1mm},fill, BrickRed] (5)--(6);

        \draw[dashed,gray] (0,0)--(0,2);
        \draw[dashed,gray] (1,0)--(1,2);
        \draw[dashed,gray] (2,0)--(2,2);
        \draw[dashed,gray] (3,0)--(3,2);
        \node at (1.5, -0.65) {(e)};
    \end{tikzpicture}&\begin{tikzpicture}[scale=0.75]
        \draw[rounded corners] (-0.5, 0) rectangle (3.5, 2) {};
        
        \node (1) at (0,1.5) [circle, draw = black,fill=black, inner sep = 0.5mm] {};
        \node (2) at (2,1.5) [circle, draw = black,fill=black, inner sep = 0.5mm] {};
        \draw[ForestGreen](1)--(2);
        
        \node (3) at (2,1) [circle, draw = black,fill=black, inner sep = 0.5mm] {};
        \node (4) at (3,1) [circle, draw = black,fill=black, inner sep = 0.5mm] {};
        \draw[ForestGreen](3)--(4);
        
        \node (5) at (1,0.5) [circle, draw = black,fill=black, inner sep = 0.5mm] {};
        \node (6) at (3,0.5) [circle, draw = black,fill=black, inner sep = 0.5mm] {};
        \draw[decorate sep={0.5mm}{1mm},fill, BrickRed] (5)--(6);

        \draw[dashed,gray] (0,0)--(0,2);
        \draw[dashed,gray] (1,0)--(1,2);
        \draw[dashed,gray] (2,0)--(2,2);
        \draw[dashed,gray] (3,0)--(3,2);
        \node at (1.5, -0.65) {(f)};
    \end{tikzpicture}
\end{array}
$$
    \caption{Minimal zero equivalences involving three operators}
    \label{fig:deg3zrels}
\end{figure}

It turns out that these are all the zero equivalences needed to generate the full list using the transformations from Section~\ref{sec:ops}. The next result summarizes all the minimal zero equivalences of degree three, together with their corresponding diagrammatic representation.

\begin{theorem}\label{thm:three0e}
    \begin{enumerate}
        \item[\textup{(i)}] For $a<b<c$, we have $\mathbf{v}_{bc}\mathbf{v}_{ab}\mathbf{v}_{bc}\equiv \mathbf{v}_{ab}\mathbf{v}_{bc}\mathbf{v}_{ab}\equiv \mathbf{0}$. $$\scalebox{0.8}{\begin{tikzpicture}[scale=0.75]
        \draw[rounded corners] (-0.5, 0) rectangle (2.5, 2) {};
    
        \node (6) at (1,1.5) [circle, draw = black,fill=black, inner sep = 0.5mm] {};
        \node (7) at (2,1.5) [circle, draw = black,fill=black, inner sep = 0.5mm] {};
        \draw[ForestGreen](6)--(7);

        \node (12) at (0,1) [circle, draw = black,fill=black, inner sep = 0.5mm] {};
        \node (13) at (1,1) [circle, draw = black,fill=black, inner sep = 0.5mm] {};
        \draw[ForestGreen](12)--(13);
        
        \node (14) at (1,0.5) [circle, draw = black,fill=black, inner sep = 0.5mm] {};
        \node (15) at (2,0.5) [circle, draw = black,fill=black, inner sep = 0.5mm] {};
        \draw[ForestGreen](14)--(15);

        \draw[dashed,gray] (0,0)--(0,2);
        \draw[dashed,gray] (1,0)--(1,2);
        \draw[dashed,gray] (2,0)--(2,2);
    \end{tikzpicture}}\hspace{0.25cm}
    \begin{tikzpicture}
        \node at (0,0){};
        \node at (0,0.5){\large{$\equiv$}};
    \end{tikzpicture}\hspace{0.25cm}
    \scalebox{0.8}{\begin{tikzpicture}[scale=0.75]
        \draw[rounded corners] (-0.5, 0) rectangle (2.5, 2) {};
    
        \node (6) at (0,1.5) [circle, draw = black,fill=black, inner sep = 0.5mm] {};
        \node (7) at (1,1.5) [circle, draw = black,fill=black, inner sep = 0.5mm] {};
        \draw[ForestGreen](6)--(7);

        \node (12) at (1,1) [circle, draw = black,fill=black, inner sep = 0.5mm] {};
        \node (13) at (2,1) [circle, draw = black,fill=black, inner sep = 0.5mm] {};
        \draw[ForestGreen](12)--(13);
        
        \node (14) at (0,0.5) [circle, draw = black,fill=black, inner sep = 0.5mm] {};
        \node (15) at (1,0.5) [circle, draw = black,fill=black, inner sep = 0.5mm] {};
        \draw[ForestGreen](14)--(15);

        \draw[dashed,gray] (0,0)--(0,2);
        \draw[dashed,gray] (1,0)--(1,2);
        \draw[dashed,gray] (2,0)--(2,2);
    \end{tikzpicture}}\hspace{0.25cm} 
    \begin{tikzpicture}
        \node at (0,0){};
        \node at (0,0.5){\large{$\equiv$}};
    \end{tikzpicture}\hspace{0.25cm}
    \begin{tikzpicture}
        \node at (0,0){};
        \node at (0,0.5){\large{$\mathbf{0}$}};
    \end{tikzpicture}$$

    \item[\textup{(ii)}] For $a<b<c$, we have $\mathbf{v}_{ca}\mathbf{v}_{bc}\mathbf{v}_{ca}\equiv \mathbf{v}_{ab}\mathbf{v}_{ca}\mathbf{v}_{ab}\equiv\mathbf{v}_{bc}\mathbf{v}_{ca}\mathbf{v}_{bc}\equiv \mathbf{v}_{ca}\mathbf{v}_{ab}\mathbf{v}_{ca}\equiv \mathbf{0}$. 
    $$\scalebox{0.8}{\begin{tikzpicture}[scale=0.75]
        \draw[rounded corners] (-0.5, 0) rectangle (2.5, 2) {};
    
        \node (6) at (0,1.5) [circle, draw = black,fill=black, inner sep = 0.5mm] {};
        \node (7) at (2,1.5) [circle, draw = black,fill=black, inner sep = 0.5mm] {};
        \draw[decorate sep={0.5mm}{1mm},fill, BrickRed] (6)--(7);

        \node (12) at (1,1) [circle, draw = black,fill=black, inner sep = 0.5mm] {};
        \node (13) at (2,1) [circle, draw = black,fill=black, inner sep = 0.5mm] {};
        \draw[ForestGreen](12)--(13);
        
        \node (14) at (0,0.5) [circle, draw = black,fill=black, inner sep = 0.5mm] {};
        \node (15) at (2,0.5) [circle, draw = black,fill=black, inner sep = 0.5mm] {};
        \draw[decorate sep={0.5mm}{1mm},fill, BrickRed] (14)--(15);

        \draw[dashed,gray] (0,0)--(0,2);
        \draw[dashed,gray] (1,0)--(1,2);
        \draw[dashed,gray] (2,0)--(2,2);
    \end{tikzpicture}}\hspace{0.25cm}\begin{tikzpicture}
        \node at (0,0){};
        \node at (0,0.5){\large{$\equiv$}};
    \end{tikzpicture}\hspace{0.25cm}\scalebox{0.8}{\begin{tikzpicture}[scale=0.75]
        \draw[rounded corners] (-0.5, 0) rectangle (2.5, 2) {};
    
        \node (6) at (0,1.5) [circle, draw = black,fill=black, inner sep = 0.5mm] {};
        \node (7) at (1,1.5) [circle, draw = black,fill=black, inner sep = 0.5mm] {};
        \draw[ForestGreen](6)--(7);

        \node (12) at (0,1) [circle, draw = black,fill=black, inner sep = 0.5mm] {};
        \node (13) at (2,1) [circle, draw = black,fill=black, inner sep = 0.5mm] {};
        \draw[decorate sep={0.5mm}{1mm},fill, BrickRed] (12)--(13);
        
        \node (14) at (0,0.5) [circle, draw = black,fill=black, inner sep = 0.5mm] {};
        \node (15) at (1,0.5) [circle, draw = black,fill=black, inner sep = 0.5mm] {};
        \draw[ForestGreen](14)--(15);

        \draw[dashed,gray] (0,0)--(0,2);
        \draw[dashed,gray] (1,0)--(1,2);
        \draw[dashed,gray] (2,0)--(2,2);
    \end{tikzpicture}}\hspace{0.25cm}\begin{tikzpicture}
        \node at (0,0){};
        \node at (0,0.5){\large{$\equiv$}};
    \end{tikzpicture}\scalebox{0.8}{\begin{tikzpicture}[scale=0.75]
        \draw[rounded corners] (-0.5, 0) rectangle (2.5, 2) {};
    
        \node (6) at (1,1.5) [circle, draw = black,fill=black, inner sep = 0.5mm] {};
        \node (7) at (2,1.5) [circle, draw = black,fill=black, inner sep = 0.5mm] {};
        \draw[ForestGreen](6)--(7);

        \node (12) at (0,1) [circle, draw = black,fill=black, inner sep = 0.5mm] {};
        \node (13) at (2,1) [circle, draw = black,fill=black, inner sep = 0.5mm] {};
        \draw[decorate sep={0.5mm}{1mm},fill, BrickRed] (12)--(13);
        
        \node (14) at (1,0.5) [circle, draw = black,fill=black, inner sep = 0.5mm] {};
        \node (15) at (2,0.5) [circle, draw = black,fill=black, inner sep = 0.5mm] {};
        \draw[ForestGreen](14)--(15);

        \draw[dashed,gray] (0,0)--(0,2);
        \draw[dashed,gray] (1,0)--(1,2);
        \draw[dashed,gray] (2,0)--(2,2);
    \end{tikzpicture}}\hspace{0.25cm}\begin{tikzpicture}
        \node at (0,0){};
        \node at (0,0.5){\large{$\equiv$}};
    \end{tikzpicture}\hspace{0.25cm}\scalebox{0.8}{\begin{tikzpicture}[scale=0.75]
        \draw[rounded corners] (-0.5, 0) rectangle (2.5, 2) {};
    
        \node (6) at (0,1.5) [circle, draw = black,fill=black, inner sep = 0.5mm] {};
        \node (7) at (2,1.5) [circle, draw = black,fill=black, inner sep = 0.5mm] {};
        \draw[decorate sep={0.5mm}{1mm},fill, BrickRed] (6)--(7);

        \node (12) at (0,1) [circle, draw = black,fill=black, inner sep = 0.5mm] {};
        \node (13) at (1,1) [circle, draw = black,fill=black, inner sep = 0.5mm] {};
        \draw[ForestGreen](12)--(13);
        
        \node (14) at (0,0.5) [circle, draw = black,fill=black, inner sep = 0.5mm] {};
        \node (15) at (2,0.5) [circle, draw = black,fill=black, inner sep = 0.5mm] {};
        \draw[decorate sep={0.5mm}{1mm},fill, BrickRed] (14)--(15);

        \draw[dashed,gray] (0,0)--(0,2);
        \draw[dashed,gray] (1,0)--(1,2);
        \draw[dashed,gray] (2,0)--(2,2);
    \end{tikzpicture}}\hspace{0.25cm}\begin{tikzpicture}
        \node at (0,0){};
        \node at (0,0.5){\large{$\equiv$}};
    \end{tikzpicture}\hspace{0.25cm}\begin{tikzpicture}
        \node at (0,0){};
        \node at (0,0.5){\large{$\mathbf{0}$}};
    \end{tikzpicture}$$
    \item[\textup{(iii)}] For $a<b<c<d$, we have $$\mathbf{v}_{ac}\mathbf{v}_{da}\mathbf{v}_{bd}\equiv \mathbf{v}_{bd}\mathbf{v}_{ab}\mathbf{v}_{ca}\equiv \mathbf{v}_{ca}\mathbf{v}_{bc}\mathbf{v}_{bd}\equiv \mathbf{v}_{db}\mathbf{v}_{cd}\mathbf{v}_{ac}\equiv \mathbf{v}_{bd}\mathbf{v}_{da}\mathbf{v}_{ac}\equiv \mathbf{v}_{ac}\mathbf{v}_{cd}\mathbf{v}_{db}\equiv \mathbf{v}_{db}\mathbf{v}_{bc}\mathbf{v}_{ca}\equiv \mathbf{v}_{ca}\mathbf{v}_{ab}\mathbf{v}_{bd}\equiv \mathbf{0}.$$ $$\scalebox{0.8}
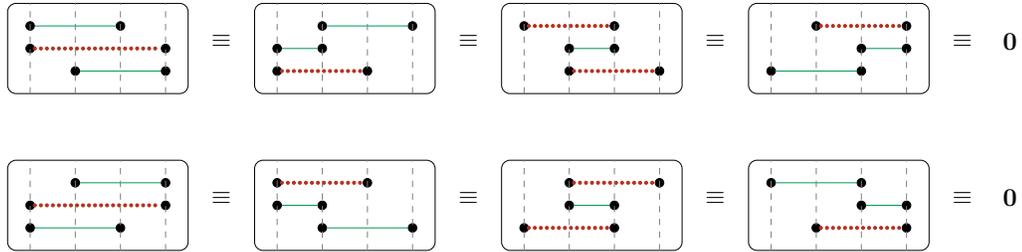
{\begin{tikzpicture}[scale=0.75]
        \draw[rounded corners] (-0.5, 0) rectangle (3.5, 2) {};
        
        \node (1) at (1,0.5) [circle, draw = black,fill=black, inner sep = 0.5mm] {};
        \node (2) at (3,0.5) [circle, draw = black,fill=black, inner sep = 0.5mm] {};
        \draw[ForestGreen](1)--(2);
        
        \node (3) at (0,1) [circle, draw = black,fill=black, inner sep = 0.5mm] {};
        \node (4) at (3,1) [circle, draw = black,fill=black, inner sep = 0.5mm] {};
        \draw[decorate sep={0.5mm}{1mm},fill, BrickRed] (3)--(4);
        
        \node (5) at (0,1.5) [circle, draw = black,fill=black, inner sep = 0.5mm] {};
        \node (6) at (2,1.5) [circle, draw = black,fill=black, inner sep = 0.5mm] {};
        \draw[ForestGreen](5)--(6);

        \draw[dashed,gray] (0,0)--(0,2);
        \draw[dashed,gray] (1,0)--(1,2);
        \draw[dashed,gray] (2,0)--(2,2);
        \draw[dashed,gray] (3,0)--(3,2);
    \end{tikzpicture}\hspace{0.25cm}\begin{tikzpicture}
        \node at (0,0){};
        \node at (0,0.75){\large{$\equiv$}};
    \end{tikzpicture}\hspace{0.25cm}\begin{tikzpicture}[scale=0.75]
        \draw[rounded corners] (-0.5, 0) rectangle (3.5, 2) {};
        
        \node (1) at (0,0.5) [circle, draw = black,fill=black, inner sep = 0.5mm] {};
        \node (2) at (2,0.5) [circle, draw = black,fill=black, inner sep = 0.5mm] {};
        \draw[decorate sep={0.5mm}{1mm},fill, BrickRed] (1)--(2);
        
        \node (3) at (0,1) [circle, draw = black,fill=black, inner sep = 0.5mm] {};
        \node (4) at (1,1) [circle, draw = black,fill=black, inner sep = 0.5mm] {};
        \draw[ForestGreen](3)--(4);
        
        \node (5) at (1,1.5) [circle, draw = black,fill=black, inner sep = 0.5mm] {};
        \node (6) at (3,1.5) [circle, draw = black,fill=black, inner sep = 0.5mm] {};
        \draw[ForestGreen](5)--(6);

        \draw[dashed,gray] (0,0)--(0,2);
        \draw[dashed,gray] (1,0)--(1,2);
        \draw[dashed,gray] (2,0)--(2,2);
        \draw[dashed,gray] (3,0)--(3,2);
    \end{tikzpicture}\hspace{0.25cm}\begin{tikzpicture}
        \node at (0,0){};
        \node at (0,0.75){\large{$\equiv$}};
    \end{tikzpicture}\hspace{0.25cm}\begin{tikzpicture}[scale=0.75]
        \draw[rounded corners] (-0.5, 0) rectangle (3.5, 2) {};
        
        \node (1) at (1,0.5) [circle, draw = black,fill=black, inner sep = 0.5mm] {};
        \node (2) at (3,0.5) [circle, draw = black,fill=black, inner sep = 0.5mm] {};
        \draw[decorate sep={0.5mm}{1mm},fill, BrickRed] (1)--(2);
        
        \node (3) at (1,1) [circle, draw = black,fill=black, inner sep = 0.5mm] {};
        \node (4) at (2,1) [circle, draw = black,fill=black, inner sep = 0.5mm] {};
        \draw[ForestGreen](3)--(4);
        
        \node (5) at (0,1.5) [circle, draw = black,fill=black, inner sep = 0.5mm] {};
        \node (6) at (2,1.5) [circle, draw = black,fill=black, inner sep = 0.5mm] {};
        \draw[decorate sep={0.5mm}{1mm},fill, BrickRed] (5)--(6);

        \draw[dashed,gray] (0,0)--(0,2);
        \draw[dashed,gray] (1,0)--(1,2);
        \draw[dashed,gray] (2,0)--(2,2);
        \draw[dashed,gray] (3,0)--(3,2);
    \end{tikzpicture}\hspace{0.25cm}\begin{tikzpicture}
        \node at (0,0){};
        \node at (0,0.75){\large{$\equiv$}};
    \end{tikzpicture}\hspace{0.25cm}\begin{tikzpicture}[scale=0.75]
        \draw[rounded corners] (-0.5, 0) rectangle (3.5, 2) {};
        
        \node (1) at (0,0.5) [circle, draw = black,fill=black, inner sep = 0.5mm] {};
        \node (2) at (2,0.5) [circle, draw = black,fill=black, inner sep = 0.5mm] {};
        \draw[ForestGreen](1)--(2);
        
        \node (3) at (2,1) [circle, draw = black,fill=black, inner sep = 0.5mm] {};
        \node (4) at (3,1) [circle, draw = black,fill=black, inner sep = 0.5mm] {};
        \draw[ForestGreen](3)--(4);
        
        \node (5) at (1,1.5) [circle, draw = black,fill=black, inner sep = 0.5mm] {};
        \node (6) at (3,1.5) [circle, draw = black,fill=black, inner sep = 0.5mm] {};
        \draw[decorate sep={0.5mm}{1mm},fill, BrickRed] (5)--(6);

        \draw[dashed,gray] (0,0)--(0,2);
        \draw[dashed,gray] (1,0)--(1,2);
        \draw[dashed,gray] (2,0)--(2,2);
        \draw[dashed,gray] (3,0)--(3,2);
    \end{tikzpicture}\hspace{0.25cm}\begin{tikzpicture}
        \node at (0,0){};
        \node at (0,0.75){\large{$\equiv$}};
    \end{tikzpicture}\hspace{0.25cm}\begin{tikzpicture}
        \node at (0,0){};
        \node at (0,0.75){\large{$\mathbf{0}$}};
    \end{tikzpicture}}$$

    $$\scalebox{0.8}{\begin{tikzpicture}[scale=0.75]
        \draw[rounded corners] (-0.5, 0) rectangle (3.5, 2) {};
        
        \node (1) at (1,1.5) [circle, draw = black,fill=black, inner sep = 0.5mm] {};
        \node (2) at (3,1.5) [circle, draw = black,fill=black, inner sep = 0.5mm] {};
        \draw[ForestGreen](1)--(2);
        
        \node (3) at (0,1) [circle, draw = black,fill=black, inner sep = 0.5mm] {};
        \node (4) at (3,1) [circle, draw = black,fill=black, inner sep = 0.5mm] {};
        \draw[decorate sep={0.5mm}{1mm},fill, BrickRed] (3)--(4);
        
        \node (5) at (0,0.5) [circle, draw = black,fill=black, inner sep = 0.5mm] {};
        \node (6) at (2,0.5) [circle, draw = black,fill=black, inner sep = 0.5mm] {};
        \draw[ForestGreen](5)--(6);

        \draw[dashed,gray] (0,0)--(0,2);
        \draw[dashed,gray] (1,0)--(1,2);
        \draw[dashed,gray] (2,0)--(2,2);
        \draw[dashed,gray] (3,0)--(3,2);
    \end{tikzpicture}\hspace{0.25cm}\begin{tikzpicture}
        \node at (0,0){};
        \node at (0,0.75){\large{$\equiv$}};
    \end{tikzpicture}\hspace{0.25cm}\begin{tikzpicture}[scale=0.75]
        \draw[rounded corners] (-0.5, 0) rectangle (3.5, 2) {};
        
        \node (1) at (0,1.5) [circle, draw = black,fill=black, inner sep = 0.5mm] {};
        \node (2) at (2,1.5) [circle, draw = black,fill=black, inner sep = 0.5mm] {};
        \draw[decorate sep={0.5mm}{1mm},fill, BrickRed] (1)--(2);
        
        \node (3) at (0,1) [circle, draw = black,fill=black, inner sep = 0.5mm] {};
        \node (4) at (1,1) [circle, draw = black,fill=black, inner sep = 0.5mm] {};
        \draw[ForestGreen](3)--(4);
        
        \node (5) at (1,0.5) [circle, draw = black,fill=black, inner sep = 0.5mm] {};
        \node (6) at (3,0.5) [circle, draw = black,fill=black, inner sep = 0.5mm] {};
        \draw[ForestGreen](5)--(6);

        \draw[dashed,gray] (0,0)--(0,2);
        \draw[dashed,gray] (1,0)--(1,2);
        \draw[dashed,gray] (2,0)--(2,2);
        \draw[dashed,gray] (3,0)--(3,2);
    \end{tikzpicture}\hspace{0.25cm}\begin{tikzpicture}
        \node at (0,0){};
        \node at (0,0.75){\large{$\equiv$}};
    \end{tikzpicture}\hspace{0.25cm}\begin{tikzpicture}[scale=0.75]
        \draw[rounded corners] (-0.5, 0) rectangle (3.5, 2) {};
        
        \node (1) at (1,1.5) [circle, draw = black,fill=black, inner sep = 0.5mm] {};
        \node (2) at (3,1.5) [circle, draw = black,fill=black, inner sep = 0.5mm] {};
        \draw[decorate sep={0.5mm}{1mm},fill, BrickRed] (1)--(2);
        
        \node (3) at (1,1) [circle, draw = black,fill=black, inner sep = 0.5mm] {};
        \node (4) at (2,1) [circle, draw = black,fill=black, inner sep = 0.5mm] {};
        \draw[ForestGreen](3)--(4);
        
        \node (5) at (0,0.5) [circle, draw = black,fill=black, inner sep = 0.5mm] {};
        \node (6) at (2,0.5) [circle, draw = black,fill=black, inner sep = 0.5mm] {};
        \draw[decorate sep={0.5mm}{1mm},fill, BrickRed] (5)--(6);

        \draw[dashed,gray] (0,0)--(0,2);
        \draw[dashed,gray] (1,0)--(1,2);
        \draw[dashed,gray] (2,0)--(2,2);
        \draw[dashed,gray] (3,0)--(3,2);
    \end{tikzpicture}\hspace{0.25cm}\begin{tikzpicture}
        \node at (0,0){};
        \node at (0,0.75){\large{$\equiv$}};
    \end{tikzpicture}\hspace{0.25cm}\begin{tikzpicture}[scale=0.75]
        \draw[rounded corners] (-0.5, 0) rectangle (3.5, 2) {};
        
        \node (1) at (0,1.5) [circle, draw = black,fill=black, inner sep = 0.5mm] {};
        \node (2) at (2,1.5) [circle, draw = black,fill=black, inner sep = 0.5mm] {};
        \draw[ForestGreen](1)--(2);
        
        \node (3) at (2,1) [circle, draw = black,fill=black, inner sep = 0.5mm] {};
        \node (4) at (3,1) [circle, draw = black,fill=black, inner sep = 0.5mm] {};
        \draw[ForestGreen](3)--(4);
        
        \node (5) at (1,0.5) [circle, draw = black,fill=black, inner sep = 0.5mm] {};
        \node (6) at (3,0.5) [circle, draw = black,fill=black, inner sep = 0.5mm] {};
        \draw[decorate sep={0.5mm}{1mm},fill, BrickRed] (5)--(6);

        \draw[dashed,gray] (0,0)--(0,2);
        \draw[dashed,gray] (1,0)--(1,2);
        \draw[dashed,gray] (2,0)--(2,2);
        \draw[dashed,gray] (3,0)--(3,2);
    \end{tikzpicture}\hspace{0.25cm}\begin{tikzpicture}
        \node at (0,0){};
        \node at (0,0.75){\large{$\equiv$}};
    \end{tikzpicture}\hspace{0.25cm}\begin{tikzpicture}
        \node at (0,0){};
        \node at (0,0.75){\large{$\mathbf{0}$}};
    \end{tikzpicture}}$$
    \end{enumerate}
\end{theorem}

Moving to non-zero equivalences of degree three, we continue with our systematic approach; that is, we gather up all the non-zero compositions of degree three for which the first is quantum, and apply all the equivalence-preserving transformations from Section~\ref{sec:ops}. Next, we group the resulting collection of compositions into sets consisting of those with the same support and number of quantum operators. Then, within each set, we determine which compositions are equivalent. As a final step, we remove any equivalences found that result from applications of degree two equivalences.

As the number of sets of compositions that need to be considered is quite large, we settle for providing only a few examples illustrating our approach. To start, we have that $$C_1=\{\mathbf{v}_{12}\mathbf{v}_{21}\mathbf{v}_{12}\}$$ is the collection of nonzero compositions of three operators with one quantum and support $[2]$, $$C_2=\{\mathbf{v}_{21}\mathbf{v}_{12}\mathbf{v}_{21}\}$$ is the collection of nonzero compositions of three operators with two quantum and support $[2]$, and 
\begin{multline*}C_3=\{\mathbf{v}_{23}\mathbf{v}_{12}\mathbf{v}_{31},\mathbf{v}_{31}\mathbf{v}_{12}\mathbf{v}_{23},\mathbf{v}_{12}\mathbf{v}_{23}\mathbf{v}_{31},\mathbf{v}_{31}\mathbf{v}_{23}\mathbf{v}_{12},\mathbf{v}_{23}\mathbf{v}_{12}\mathbf{v}_{21},\mathbf{v}_{21}\mathbf{v}_{12}\mathbf{v}_{23},\mathbf{v}_{12}\mathbf{v}_{23}\mathbf{v}_{32},\\
\mathbf{v}_{32}\mathbf{v}_{23}\mathbf{v}_{12}, \mathbf{v}_{13}\mathbf{v}_{31}\mathbf{v}_{12}, \mathbf{v}_{13}\mathbf{v}_{31}\mathbf{v}_{23},\mathbf{v}_{13}\mathbf{v}_{31}\mathbf{v}_{13}, \mathbf{v}_{23}\mathbf{v}_{31}\mathbf{v}_{13},\mathbf{v}_{12}\mathbf{v}_{31}\mathbf{v}_{23}, \mathbf{v}_{23}\mathbf{v}_{31}\mathbf{v}_{12}\}
\end{multline*}
is the set of non-zero compositions of three operators with support $[3]$ and one quantum operator. Comparing compositions in $C_3$, we obtain the equivalences $$\mathbf{v}_{cb}\mathbf{v}_{bc}\mathbf{v}_{ab}\equiv \mathbf{v}_{ab}\mathbf{v}_{ca}\mathbf{v}_{ac},\quad\quad \mathbf{v}_{ac}\mathbf{v}_{ca}\mathbf{v}_{bc}\equiv \mathbf{v}_{bc}\mathbf{v}_{ab}\mathbf{v}_{ba},\quad\quad \mathbf{v}_{ba}\mathbf{v}_{ab}\mathbf{v}_{bc}\equiv \mathbf{v}_{bc}\mathbf{v}_{ca}\mathbf{v}_{ac},$$ and $$\mathbf{v}_{ac}\mathbf{v}_{ca}\mathbf{v}_{ab}\equiv \mathbf{v}_{ab}\mathbf{v}_{bc}\mathbf{v}_{cb},$$ none of which are consequences of degree-two nonzero equivalences, i.e., all four equivalences are minimal. On the other hand, comparing the remaining nonzero compositions of three operators with one quantum and support $[6]$, one finds the equivalence $\mathbf{v}_{12} \mathbf{v}_{43}\mathbf{v}_{56} \equiv \mathbf{v}_{56}\mathbf{v}_{12}\mathbf{v}_{43}$. As the equivalence $\mathbf{v}_{12}\mathbf{v}_{43}\mathbf{v}_{56}\equiv \mathbf{v}_{56} \mathbf{v}_{12}\mathbf{v}_{43}$ is a consequence of nonzero degree-two equivalences, we would not include it in our final list of minimal degree-three nonzero equivalences.

The following result summarizes the output of the process outlined above.
\begin{theorem}\label{thm:d3nz}
    Let $a<b<c<d$. The following is a complete list of the non-zero equivalences of degree three: 
    \begin{multicols}{2}
    \begin{enumerate}
        \item[\textup{(i)}] $\mathbf{v}_{cb}\mathbf{v}_{bc}\mathbf{v}_{ab}\equiv \mathbf{v}_{ab}\mathbf{v}_{ca}\mathbf{v}_{ac}$
        \item[\textup{(ii)}] $\mathbf{v}_{ac}\mathbf{v}_{ca}\mathbf{v}_{bc}\equiv \mathbf{v}_{bc}\mathbf{v}_{ab}\mathbf{v}_{ba}$
        \item[\textup{(iii)}] $\mathbf{v}_{ba}\mathbf{v}_{ab}\mathbf{v}_{ca}\equiv \mathbf{v}_{ca}\mathbf{v}_{bc}\mathbf{v}_{cb}$
        \item[\textup{(iv)}] $\mathbf{v}_{ba}\mathbf{v}_{ab}\mathbf{v}_{bc}\equiv \mathbf{v}_{bc}\mathbf{v}_{ca}\mathbf{v}_{ac}$
        \item[\textup{(v)}] $\mathbf{v}_{cb}\mathbf{v}_{bc}\mathbf{v}_{ca}\equiv \mathbf{v}_{ca}\mathbf{v}_{ab}\mathbf{v}_{ba}$
        \item[\textup{(vi)}] $\mathbf{v}_{ac}\mathbf{v}_{ca}\mathbf{v}_{ab}\equiv \mathbf{v}_{ab}\mathbf{v}_{bc}\mathbf{v}_{cb}$
        \item[\textup{(vii)}] $\mathbf{v}_{db}\mathbf{v}_{bc}\mathbf{v}_{ab}\equiv \mathbf{v}_{ab}\mathbf{v}_{da}\mathbf{v}_{ac}$
        \item[\textup{(viii)}] $\mathbf{v}_{ac}\mathbf{v}_{cd}\mathbf{v}_{bc}\equiv \mathbf{v}_{bc}\mathbf{v}_{ab}\mathbf{v}_{bd}$
        \item[\textup{(ix)}] $\mathbf{v}_{bd}\mathbf{v}_{da}\mathbf{v}_{cd}\equiv \mathbf{v}_{cd}\mathbf{v}_{bc}\mathbf{v}_{ca}$
        \item[\textup{(x)}] $\mathbf{v}_{ca}\mathbf{v}_{ab}\mathbf{v}^q_{da}\equiv \mathbf{v}_{da}\mathbf{v}_{cd}\mathbf{v}_{db}$
        \item[\textup{(xi)}] $\mathbf{v}_{ca}\mathbf{v}_{bc}\mathbf{v}_{cd}\equiv \mathbf{v}_{cd}\mathbf{v}_{da}\mathbf{v}_{bd}$
        \item[\textup{(xii)}] $\mathbf{v}_{db}\mathbf{v}_{cd}\mathbf{v}_{da}\equiv \mathbf{v}_{da}\mathbf{v}_{ab}\mathbf{v}_{ca}$
        \item[\textup{(xiii)}] $\mathbf{v}_{ac}\mathbf{v}_{da}\mathbf{v}_{ab}\equiv \mathbf{v}_{ab}\mathbf{v}_{bc}\mathbf{v}_{db}$
        \item[\textup{(xiv)}] $\mathbf{v}_{bd}\mathbf{v}_{ab}\mathbf{v}_{bc}\equiv \mathbf{v}_{bc}\mathbf{v}_{cd}\mathbf{v}_{ac}$
    \end{enumerate}
    \end{multicols}
\end{theorem}

\subsection{Degree 4 Equivalences}
In this section, we discuss minimal equivalence relations of degree four. The general approach is the same as in the cases of degree two and three equivalences. However, the number of compositions that need to be considered is becoming quite large. In addition, identifying those equivalences that are not minimal is becoming more challenging. Therefore, we start with two technical results that help us identify non-minimal equivalences as well as restrict our search.

\begin{lemma}\label{lem:supp2n-2help}
    Let $\mathbf{v}=\mathbf{v}_{a_4b_4}\mathbf{v}_{a_3b_3}\mathbf{v}_{a_2b_2}\mathbf{v}_{a_1b_1}$ with 
    \begin{itemize}
        \item[$(i)$] $|\{a_1,b_1\}\cap\{a_2,b_2\}|=1$,
        \item[$(ii)$] $|\{a_3,b_3\}\cap\{a_4,b_4\}|=1$, and
        \item[$(iii)$] $|\{a_1,b_1,a_2,b_2\}\cap\{a_3,b_3,a_4,b_4\}|=0$.
    \end{itemize}
    If $\mathbf{v}\equiv \mathbf{0}$, then it is a consequence of a degree two or degree three zero equivalence; that is, $\mathbf{v}\equiv\mathbf{v}_{a'_4b'_4}\mathbf{v}_{a'_3b'_3}\mathbf{v}_{a'_2b'_2}\mathbf{v}_{a'_1b'_1}$ where either $\mathbf{v}_{a'_{i+1}b'_{i+1}}\mathbf{v}_{a'_{i}b'_{i}}\equiv \mathbf{0}$ for $i\in [3]$ or $\mathbf{v}_{a'_{i+2}b'_{i+2}}\mathbf{v}_{a'_{i+1}b'_{i+1}}\mathbf{v}_{a'_{i}b'_{i}}\equiv \mathbf{0}$ for $i\in [2]$.
\end{lemma}

\begin{remark}
    Note that the conditions on the values of the indices of $\mathbf{v}$ in Lemma~\ref{lem:supp2n-2help} essentially say that $\mathbf{v}$ is the composition of two disjoint compositions of degree two, each of them formed by two different operators whose supports share a value. 
\end{remark}

\begin{proof}
    Assume that
    \begin{align*}
        (\star)~~&\textup{if $\mathbf{v}\equiv\mathbf{v}_{a'_4b'_4}\mathbf{v}_{a'_3b'_3}\mathbf{v}_{a'_2b'_2}\mathbf{v}_{a'_1b'_1}$, then $\mathbf{v}_{a'_{i+1}b'_{i+1}}\mathbf{v}_{a'_{i}b'_{i}}\not\equiv \mathbf{0}$ for $i\in [3]$ and}\\
        &\textup{$\mathbf{v}_{a'_{i+2}b'_{i+2}}\mathbf{v}_{a'_{i+1}b'_{i+1}}\mathbf{v}_{a'_{i}b'_{i}}\not\equiv \mathbf{0}$ for $i\in [2]$.}
    \end{align*}
    We show that $\mathbf{v}\not\equiv \mathbf{0}$. Consider the representation $\mathbf{v}=\mathbf{v}_{a_4b_4}\mathbf{v}_{a_3b_3}\mathbf{v}_{a_2b_2}\mathbf{v}_{a_1b_1}$ with 
    \begin{itemize}
        \item[$(i)$] $|\{a_1,b_1\}\cap\{a_2,b_2\}|=1$,
        \item[$(ii)$] $|\{a_3,b_3\}\cap\{a_4,b_4\}|=1$, and
        \item[$(iii)$] $|\{a_1,b_1,a_2,b_2\}\cap\{a_3,b_3,a_4,b_4\}|=0$.
    \end{itemize}
    Note that since $|\{a_1,b_1\}\cap\{a_2,b_2\}|=1$, $|\{a_3,b_3\}\cap\{a_4,b_4\}|=1$, $\mathbf{v}_{a_2b_2}\mathbf{v}_{a_1b_1}\not\equiv 0$, and $\mathbf{v}_{a_4b_4}\mathbf{v}_{a_3b_3}\not\equiv 0$ by assumption, considering Theorem~\ref{thm:deg2zero}, it follows that $a_i>b_i$ for at most one $i\in[2]$ and at most one $i\in\{3,4\}$. Moreover, note that since $|\{a_1,b_1,a_2,b_2\}\cap\{a_3,b_3,a_4,b_4\}|=0$, combining Theorem~\ref{thm:deg2zero}~(ii), (iv), (vi) and Proposition~\ref{prop:deg2rel}, we have that 
$$(\star\star)~~\mathbf{v}_{a_4b_4}\mathbf{v}_{a_3b_3}\mathbf{v}_{a_2b_2}\mathbf{v}_{a_1b_1}\equiv \mathbf{v}_{a_4b_4}\mathbf{v}_{a_2b_2}\mathbf{v}_{a_3b_3}\mathbf{v}_{a_1b_1}
\equiv\mathbf{v}_{a_2b_2}\mathbf{v}_{a_4b_4}\mathbf{v}_{a_1b_1}\mathbf{v}_{a_3b_3}
\equiv\mathbf{v}_{a_2b_2}\mathbf{v}_{a_1b_1}\mathbf{v}_{a_4b_4}\mathbf{v}_{a_3b_3}.$$
    There are three cases to consider.
    \bigskip

    \noindent
    \underline{Case 1:} $a_1>b_1$ or $a_2>b_2$, and $a_3>b_3$ or $a_4>b_4$. We assume that $a_1>b_1$ and $a_3>b_3$, as the other cases follow similarly. Since $|\{a_1,b_1\}\cap\{a_{2},b_{2}\}|=1$ and $|\{a_3,b_3\}\cap\{a_4,b_4\}|=1$, considering Theorem~\ref{thm:deg2zero}, we must have either $b_1=a_2<b_2<a_1$ or $b_1<a_2<b_2=a_1$ and either $b_3=a_4<b_4<a_3$ or $b_3<a_4<b_4=a_3$. As the arguments are all similar, we assume that $b_1=a_2<b_2<a_1$ and $b_3=a_4<b_4<a_3$. Now, combining $(\star)$ and $(\star\star)$, we find that $\mathbf{v}_{a_ib_i}\mathbf{v}_{a_jb_j}\not\equiv \mathbf{0}$ and $\mathbf{v}_{a_ib_i}\mathbf{v}_{a_1b_1}\not\equiv \mathbf{0}$ for $i\in \{3,4\}$ and $j\in \{1,2\}$. Thus, since $|\{a_1,b_1,a_2,b_2\}\cap\{a_3,b_3,a_4,b_4\}|=0$, by
    Theorem~\ref{thm:deg2zero}, it must be the case that either 
    \begin{itemize}
        \item $b_2<a_4<a_3<a_1$, and the result follows since for $$u_1=(a_1,a_3,b_4,a_4,b_2,a_2,\hdots),$$ we have $\mathbf{v}\bullet_2u_1\neq 0$; or 
        \item $b_4<a_2<a_1<a_3$, and the result follows since for $$u_2=(a_3,a_1,b_2,a_2,b_4,a_4,\hdots),$$ we have $\mathbf{v}\bullet_2u_2\neq 0$.
    \end{itemize}
    \bigskip

    \noindent
    \underline{Case 2:} $a_i<b_i$ for $i\in \{3,4\}$ and either $a_1>b_1$ or $a_2>b_2$. We assume that $a_1>b_1$, as the other case follows similarly. Since $|\{a_1,b_1\}\cap\{a_2,b_2\}|=1$, considering Theorem~\ref{thm:deg2zero}, we must have either $b_1=a_2<b_2<a_1$ or $b_1<a_2<b_2=a_1$. Moreover, since $|\{a_3,b_3\}\cap\{a_4,b_4\}|=1$ with $a_3<b_3$ and $a_4<b_4$, once again considering Theorem~\ref{thm:deg2zero}, we must have either $a_3<b_3=a_4<b_4$ or $a_4<b_4=a_3<b_3$. As the arguments are all similar, we assume that $b_1=a_2<b_2<a_1$ and $a_3<b_3=a_4<b_4$. Now, since $\mathbf{v}_{a_ib_i}\mathbf{v}_{a_jb_j}\not\equiv \mathbf{0}$ and $\mathbf{v}_{a_ib_i}\mathbf{v}_{a_1b_1}\not\equiv \mathbf{0}$ for $i\in\{3,4\}$ and $j\in\{1,2\}$ as noted in Case 1 and $|\{a_1,b_1,a_2,b_2\}\cap\{a_3,b_3,a_4,b_4\}|=0$, applying Theorem~\ref{thm:deg2zero} it must be the case that either $b_4<a_2$, $a_2<a_3<b_4<b_2$, $b_2<a_3<b_4<a_1$, or $a_1<a_3$.
    In all cases, if $u=(a_3,a_1,b_2,a_2,a_4,b_4,\hdots)$, then $\mathbf{v}\bullet_2u\neq 0$.
    \bigskip

    \noindent
    \underline{Case 3:} $a_i<b_i$ for all $i\in [4]$. Since $|\{a_1,b_1\}\cap\{a_2,b_2\}|=1$ and $|\{a_3,b_3\}\cap\{a_4,b_4\}|=1$, considering Theorem~\ref{thm:deg2zero}, we must have 
    \begin{itemize}
        \item either $a_1<b_1=a_2<b_2$ or $a_2<b_2=a_1<b_1$ and
        \item either $a_3<b_3=a_4<b_4$ or $a_4<b_4=a_3<b_3$.
    \end{itemize}
    Since the arguments are all similar, we assume that $a_1<b_1=a_2<b_2$ and $a_3<b_3=a_4<b_4$. Now, since $\mathbf{v}_{a_ib_i}\mathbf{v}_{a_jb_j}\not\equiv \mathbf{0}$ and $\mathbf{v}_{a_ib_i}\mathbf{v}_{a_1b_1}\not\equiv \mathbf{0}$ for $i\in\{3,4\}$ and $j\in\{1,2\}$ as noted in Case 1 and $|\{a_1,b_1,a_2,b_2\}\cap\{a_3,b_3,a_4,b_4\}|=0$, applying Theorem~\ref{thm:deg2zero} we have that either
    \begin{itemize}
        \item $b_4<a_1$, $b_2<a_3$, $a_1<a_3<b_4<a_2$, or $a_2<a_3<b_4<a_2$, and the result follows since for $$u_1=(a_3,a_1,b_2,a_2,b_4,a_4,\hdots),$$ we have $\mathbf{v}\bullet_2u_1\neq 0$; or
        \item $a_3<a_1<b_2<a_4$, or $a_4<a_1<b_2<b_4$, and the result follows since for $$u_2=(a_1,a_3,b_4,a_4,b_2,a_2,\hdots),$$ we have $\mathbf{v}\bullet_2u_2\neq 0$. \qedhere
    \end{itemize}
\end{proof}

The next result is of a similar form to Lemma~\ref{lem:help1}. As the proof is similar to that of Lemma~\ref{lem:supp2n-2help}, we omit it.

\begin{lemma}\label{lem:suppbounddeg4z}
    Let $n\geq 4$ and consider a composition $\mathbf{v}=\mathbf{v}_{a_nb_n}\cdots\mathbf{v}_{a_1b_1}$ with $|\supp(\mathbf{v})|=2n-2$. If $\mathbf{v}\equiv \mathbf{0}$, then it is a consequence of a degree two or three zero equivalence; that is, $\mathbf{v}\equiv\mathbf{v}_{a'_nb'_n}\cdots\mathbf{v}_{a'_1b'_1}$ where either $\mathbf{v}_{a'_{i+1}b'_{i+1}}\mathbf{v}_{a'_{i}b'_{i}}\equiv \mathbf{0}$ for $i\in [n-1]$ or $\mathbf{v}_{a'_{i+2}b'_{i+2}}\mathbf{v}_{a'_{i+1}b'_{i+1}}\mathbf{v}_{a'_{i}b'_{i}}\equiv \mathbf{0}$ for $i\in [n-2]$.
\end{lemma}

Finally, we summarize the output of a full study of the degree four equivalences as done for degree three equivalences in Section~\ref{sec:order3}. The fact that the equivalences are in fact minimal is established in Theorem~\ref{thm:arblength}, where generalizations of arbitrary degree are considered.

\begin{theorem}\label{thm:4zero}
The minimal degree four equivalences consist of the compositions illustrated in Figure~\ref{fig:d4} along with any related to them by the application of the equivalence-preserving transformations of Section~\ref{sec:ops}.

\begin{figure}[H]
    \centering
    \begin{equation*}
\begin{array}{ccc}
\text{Zero equivalence} & \qquad & \text{Non-zero equivalence} \\
\begin{tikzpicture}[scale=0.75]
        \draw[rounded corners] (-0.5, 0) rectangle (4.5, 2.5) {};
        
        \node (1) at (0,0.5) [circle, draw = black,fill=black, inner sep = 0.5mm] {};
        \node (2) at (3,0.5) [circle, draw = black,fill=black, inner sep = 0.5mm] {};
        \draw[decorate sep={0.5mm}{1mm},fill, BrickRed] (1)--(2);
        
        \node (3) at (0,1) [circle, draw = black,fill=black, inner sep = 0.5mm] {};
        \node (4) at (1,1) [circle, draw = black,fill=black, inner sep = 0.5mm] {};
        \draw[ForestGreen](3)--(4);
        
        \node (5) at (1,1.5) [circle, draw = black,fill=black, inner sep = 0.5mm] {};
        \node (6) at (2,1.5) [circle, draw = black,fill=black, inner sep = 0.5mm] {};
        \draw[ForestGreen](5)--(6);

        \node (7) at (2,2) [circle, draw = black,fill=black, inner sep = 0.5mm] {};
        \node (8) at (4,2) [circle, draw = black,fill=black, inner sep = 0.5mm] {};
        \draw[ForestGreen](7)--(8);
        
        \draw[dashed,gray] (0,0)--(0,2.5);
        \draw[dashed,gray] (1,0)--(1,2.5);
        \draw[dashed,gray] (2,0)--(2,2.5);
        \draw[dashed,gray] (3,0)--(3,2.5);
        \draw[dashed,gray] (4,0)--(4,2.5);
    \end{tikzpicture}\hspace{0.25cm}\hspace{0.25cm}\begin{tikzpicture}
        \node at (0,0) {};
        \node at (0,1){\large{$\equiv \hspace{0.25cm} \mathbf{0}$}};
    \end{tikzpicture}
     &\hspace{0.5cm} &
\begin{tikzpicture}[scale=0.75]
        \draw[rounded corners] (-0.5, 0) rectangle (3.5, 2.5) {};
        
        \node (1) at (0,0.5) [circle, draw = black,fill=black, inner sep = 0.5mm] {};
        \node (2) at (3,0.5) [circle, draw = black,fill=black, inner sep = 0.5mm] {};
        \draw[decorate sep={0.5mm}{1mm},fill, BrickRed] (1)--(2);
        
        \node (3) at (0,1) [circle, draw = black,fill=black, inner sep = 0.5mm] {};
        \node (4) at (1,1) [circle, draw = black,fill=black, inner sep = 0.5mm] {};
        \draw[ForestGreen](3)--(4);
        
        \node (5) at (1,1.5) [circle, draw = black,fill=black, inner sep = 0.5mm] {};
        \node (6) at (2,1.5) [circle, draw = black,fill=black, inner sep = 0.5mm] {};
        \draw[ForestGreen](5)--(6);

        \node (7) at (0,2) [circle, draw = black,fill=black, inner sep = 0.5mm] {};
        \node (8) at (2,2) [circle, draw = black,fill=black, inner sep = 0.5mm] {};
        \draw[decorate sep={0.5mm}{1mm},fill, BrickRed] (7)--(8);
        
        \draw[dashed,gray] (0,0)--(0,2.5);
        \draw[dashed,gray] (1,0)--(1,2.5);
        \draw[dashed,gray] (2,0)--(2,2.5);
        \draw[dashed,gray] (3,0)--(3,2.5);
    \end{tikzpicture}\hspace{0.25cm}\begin{tikzpicture}
        \node at (0,0){};
        \node at (0,0.75){\large{$\equiv$}};
    \end{tikzpicture}\hspace{0.25cm}\begin{tikzpicture}[scale=0.8]
        \draw[rounded corners] (-0.5, 0) rectangle (3.5, 2.5) {};
        
        \node (1) at (1,0.5) [circle, draw = black,fill=black, inner sep = 0.5mm] {};
        \node (2) at (3,0.5) [circle, draw = black,fill=black, inner sep = 0.5mm] {};
        \draw[decorate sep={0.5mm}{1mm},fill, BrickRed] (1)--(2);
        
        \node (3) at (1,1) [circle, draw = black,fill=black, inner sep = 0.5mm] {};
        \node (4) at (2,1) [circle, draw = black,fill=black, inner sep = 0.5mm] {};
        \draw[ForestGreen](3)--(4);
        
        \node (5) at (2,1.5) [circle, draw = black,fill=black, inner sep = 0.5mm] {};
        \node (6) at (3,1.5) [circle, draw = black,fill=black, inner sep = 0.5mm] {};
        \draw[ForestGreen](5)--(6);

        \node (7) at (0,2) [circle, draw = black,fill=black, inner sep = 0.5mm] {};
        \node (8) at (3,2) [circle, draw = black,fill=black, inner sep = 0.5mm] {};
        \draw[decorate sep={0.5mm}{1mm},fill, BrickRed] (7)--(8);
        
        \draw[dashed,gray] (0,0)--(0,2.5);
        \draw[dashed,gray] (1,0)--(1,2.5);
        \draw[dashed,gray] (2,0)--(2,2.5);
        \draw[dashed,gray] (3,0)--(3,2.5);
    \end{tikzpicture}
\end{array}
\end{equation*}
    \caption{Minimal equivalences of degree four}
    \label{fig:d4}
\end{figure}
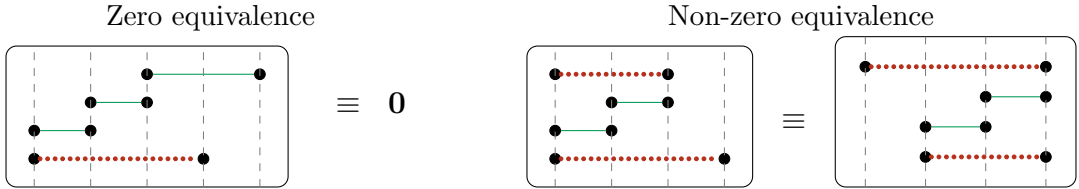
\end{theorem}

While we could continue as above and move to equivalences of degrees five, six, etc, evidently, we need more than just such experimental evidence to arrive at a complete list of equivalences. In particular, a better understanding of the quantum $k$-Bruhat order seems crucial (see the end of Section~\ref{sec:arblength} for a discussion). Without this, we can merely conjecture a complete list of equivalences. In this regard, we believe that the equivalences found above, along with two generalizations of arbitrary degree discussed in the following section, do, indeed, form a complete list of equivalences.

\section{Equivalences of arbitrary degree}\label{sec:arblength}

In this section, we establish two families of minimal relations of arbitrary degree, one a zero relation and the other non-zero, generalizing the minimal degree-four equivalences found in the previous section. Specifically, we dedicate the remainder of this section, as well as the paper, to the proof of the following result.

\begin{theorem}\label{thm:arblength}
Let $n\geq 4$ and consider $a_1<a_2<\cdots<a_n$. Then, we have the following minimal equivalences:
\begin{enumerate}
    \item[$(i)$] $\mathbf{v}_{a_{n-2}a_n}\mathbf{v}_{a_{n-3}a_{n-2}}\cdots\mathbf{v}_{a_1a_2}\mathbf{v}_{a_{n-1}a_1}\equiv \mathbf{0}$; and
    \item[$(ii)$] $\mathbf{v}_{a_{n-1}a_1}\mathbf{v}_{a_{n-2}a_{n-1}}\mathbf{v}_{a_{n-3}a_{n-2}}\cdots \mathbf{v}_{a_{1}a_{2}}\mathbf{v}_{a_n,a_1}\equiv \mathbf{v}_{a_na_1}\mathbf{v}_{a_{n-1}a_{n}}\mathbf{v}_{a_{n-2}a_{n-1}}\cdots \mathbf{v}_{a_{2}a_{3}}\mathbf{v}_{a_na_2}$.
\end{enumerate}
\end{theorem}

\begin{figure}[H]
    \centering
    $$\begin{array}{ccc}
    \text{Zero equivalence} & \qquad & \text{Non-zero equivalence} \\
    \begin{tikzpicture}[scale=0.8]
        \draw[rounded corners] (-0.5, 0) rectangle (5.5, 3) {};
        
        \node (1) at (0,0.5) [circle, draw = black,fill=black, inner sep = 0.5mm] {};
        \node (2) at (4,0.5) [circle, draw = black,fill=black, inner sep = 0.5mm] {};
        \draw[decorate sep={0.5mm}{1mm},fill, BrickRed] (1)--(2);
        
        \node (3) at (0,1) [circle, draw = black,fill=black, inner sep = 0.5mm] {};
        \node (4) at (1,1) [circle, draw = black,fill=black, inner sep = 0.5mm] {};
        \draw[ForestGreen](3)--(4);

        \draw[decorate sep={0.3mm}{1mm},fill,black] (1.25,1.25)--(1.75,1.75);
        
        \node (5) at (2,2) [circle, draw = black,fill=black, inner sep = 0.5mm] {};
        \node (6) at (3,2) [circle, draw = black,fill=black, inner sep = 0.5mm] {};
        \draw[ForestGreen](5)--(6);

        \node (7) at (3,2.5) [circle, draw = black,fill=black, inner sep = 0.5mm] {};
        \node (8) at (5,2.5) [circle, draw = black,fill=black, inner sep = 0.5mm] {};
        \draw[ForestGreen](7)--(8);
        
        \draw[dashed,gray] (0,0)--(0,3);
        \draw[dashed,gray] (1,0)--(1,3);
        \draw[dashed,gray] (2,0)--(2,3);
        \draw[dashed,gray] (3,0)--(3,3);
        \draw[dashed,gray] (4,0)--(4,3);
        \draw[dashed,gray] (5,0)--(5,3);
    \end{tikzpicture}\hspace{0.25cm}\begin{tikzpicture}
        \node at (0,0) {};
        \node at (0,1){\large{$\equiv \hspace{0.25cm} \mathbf{0}$}};
    \end{tikzpicture}
     &\hspace{0.5cm} &
\begin{tikzpicture}[scale=0.8]
        \draw[rounded corners] (-0.5, 0) rectangle (4.5, 3) {};
        
        \node (1) at (0,0.5) [circle, draw = black,fill=black, inner sep = 0.5mm] {};
        \node (2) at (4,0.5) [circle, draw = black,fill=black, inner sep = 0.5mm] {};
        \draw[decorate sep={0.5mm}{1mm},fill, BrickRed] (1)--(2);
        
        \node (3) at (0,1) [circle, draw = black,fill=black, inner sep = 0.5mm] {};
        \node (4) at (1,1) [circle, draw = black,fill=black, inner sep = 0.5mm] {};
        \draw[ForestGreen](3)--(4);
        
        \draw[decorate sep={0.3mm}{1mm},fill,black] (1.25,1.25)--(1.75,1.75);

        \node (7) at (2,2) [circle, draw = black,fill=black, inner sep = 0.5mm] {};
        \node (8) at (3,2) [circle, draw = black,fill=black, inner sep = 0.5mm] {};
        \draw[ForestGreen](7)--(8);

        \node (9) at (0,2.5) [circle, draw = black,fill=black, inner sep = 0.5mm] {};
        \node (10) at (3,2.5) [circle, draw = black,fill=black, inner sep = 0.5mm] {};
        \draw[decorate sep={0.5mm}{1mm},fill, BrickRed] (9)--(10);
        
        \draw[dashed,gray] (0,0)--(0,3);
        \draw[dashed,gray] (1,0)--(1,3);
        \draw[dashed,gray] (2,0)--(2,3);
        \draw[dashed,gray] (3,0)--(3,3);
        \draw[dashed,gray] (4,0)--(4,3);
    \end{tikzpicture}\hspace{0.25cm}\begin{tikzpicture}
        \node at (0,0){};
        \node at (0,1){\large{$\equiv$}};
    \end{tikzpicture}\hspace{0.25cm}\begin{tikzpicture}[scale=0.8]
        \draw[rounded corners] (-0.5, 0) rectangle (4.5, 3) {};
        
        \node (1) at (1,0.5) [circle, draw = black,fill=black, inner sep = 0.5mm] {};
        \node (2) at (4,0.5) [circle, draw = black,fill=black, inner sep = 0.5mm] {};
        \draw[decorate sep={0.5mm}{1mm},fill, BrickRed] (1)--(2);
        
        \node (3) at (1,1) [circle, draw = black,fill=black, inner sep = 0.5mm] {};
        \node (4) at (2,1) [circle, draw = black,fill=black, inner sep = 0.5mm] {};
        \draw[ForestGreen](3)--(4);

        \draw[decorate sep={0.3mm}{1mm},fill,black] (2.25,1.25)--(2.75,1.75);

        \node (7) at (3,2) [circle, draw = black,fill=black, inner sep = 0.5mm] {};
        \node (8) at (4,2) [circle, draw = black,fill=black, inner sep = 0.5mm] {};
        \draw[ForestGreen](7)--(8);

        \node (9) at (0,2.5) [circle, draw = black,fill=black, inner sep = 0.5mm] {};
        \node (10) at (4,2.5) [circle, draw = black,fill=black, inner sep = 0.5mm] {};
        \draw[decorate sep={0.5mm}{1mm},fill, BrickRed] (9)--(10);
        
        \draw[dashed,gray] (0,0)--(0,3);
        \draw[dashed,gray] (1,0)--(1,3);
        \draw[dashed,gray] (2,0)--(2,3);
        \draw[dashed,gray] (3,0)--(3,3);
        \draw[dashed,gray] (4,0)--(4,3);
    \end{tikzpicture}
    \end{array}$$
    \caption{Equivalences of arbitrary degree}
    \label{fig:arbdegequiv}
\end{figure}
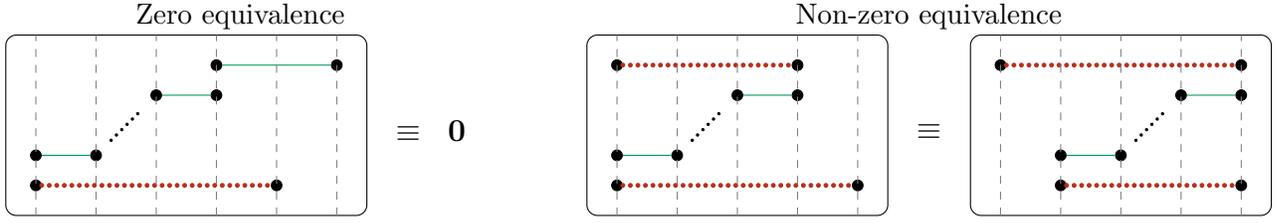

The zero equivalences of Theorem~\ref{thm:arblength} appear in~\cite{qkB1}, where the authors show that they are, indeed, equivalent to zero. Thus, it remains to show that they are also minimal.

\begin{prop}\label{prop:arblength-minimal}
For $n\geq 4$ and $a_1<a_2<\cdots<a_n$, the following zero equivalence is minimal $$\mathbf{v}_{a_{n-2}a_n}\mathbf{v}_{a_{n-3}a_{n-2}}\cdots\mathbf{v}_{a_1a_2}\mathbf{v}_{a_{n-1}a_1}\equiv \mathbf{0}.$$
\end{prop}

\begin{proof}
    Let $\mathbf{v}=\mathbf{v}_{a_{n-2}a_n}\mathbf{v}_{a_{n-3}a_{n-2}}\cdots\mathbf{v}_{a_1a_2}\mathbf{v}_{a_{n-1}a_1}$ with $a_1<a_2<\cdots<a_n$. There are two ways in which the relation $\mathbf{v}\equiv \mathbf{0}$ could be non-minimal:
    \begin{itemize}
        \item[(1)] there exists a contiguous subcomposition of $\mathbf{v}$ equivalent to $\mathbf{0}$ or
        \item[(2)] one can apply non-zero equivalences to $\mathbf{v}$ to form an expression satisfying (1). 
    \end{itemize}
    For (1), considering the results concerning compositions of classical operators in~\cite{BS2}, it follows that 
    $$\mathbf{v}_{a_{n-i}a_{n-i+1}}\mathbf{v}_{a_{n-i+1}a_{n-i+2}}\cdots \mathbf{v}_{a_{n-j}a_{n-j+1}}\not\equiv \mathbf{0}$$ 
    for $3\le i<j\le n-1$ and 
    $$\mathbf{v}_{a_{n-2}a_{n}}\mathbf{v}_{a_{n-3}a_{n-2}}\cdots \mathbf{v}_{a_{n-i}a_{n-i+1}}\not\equiv \mathbf{0}$$ \
    for $3\le i\le n-1$. 
    It remains to consider compositions of the form $\mathbf{v}_{a_{i}a_{i+1}}\cdots \mathbf{v}_{a_1a_2}\mathbf{v}_{a_{n-1}a_1}$ for $1\le i\le n-3$. Considering Theorem~\ref{thm:flat}, it suffices to show that $\mathbf{v}_i=\mathbf{v}_{i,i+1}\cdots \mathbf{v}_{12}\mathbf{v}_{n-1,1}\not\equiv \mathbf{0}$ for $1\le i\le n-3$. To this end, note that 
    $$\mathbf{v}_i\bullet_2(n,n-1|1,2,3,\hdots,i,i+1,i+2,\hdots,n-2)=q_2(n,i+1|n-1,1,2,\hdots,i-1,i,i+2,\hdots,n-2)\neq 0.$$ 
    Thus, if $\mathbf{v}\equiv \mathbf{0}$ is non-minimal, then it is a consequence of (2).

    If one can apply nontrivial non-zero equivalences to $v$, then they must apply to a subcomposition of one of the following forms: $$\mathbf{v}_{a_{n-i}a_{n-i+1}}\mathbf{v}_{a_{n-i+1}a_{n-i+2}}\cdots \mathbf{v}_{a_{n-j}a_{n-j+1}}$$ for $3\le i<j\le n-1$, 
    $$\mathbf{v}_{a_{n-2}a_{n}}\mathbf{v}_{a_{n-3}a_{n-2}}\cdots \mathbf{v}_{a_{n-i}a_{n-i+1}}$$ 
    for $3\le i\le n-1$, or $$\mathbf{v}_{a_{i}a_{i+1}}\cdots \mathbf{v}_{a_1a_2}\mathbf{v}_{a_{n-1}a_1}$$ for $1\le i\le n-3$. Equivalently, applying Theorem~\ref{thm:flat}, there must exist nontrivial non-zero equivalences involving either an composition of operators of the form $$\mathbf{v}_{i,i+1}\mathbf{v}_{i-1,i}\cdots \mathbf{v}_{j,j+1}$$ for $1\le j<i\le n-1$ or of the form $$\mathbf{v}_{i,i+1}\cdots \mathbf{v}_{12}\mathbf{v}_{n-1,1}$$ for $1\le i\le n-3$. Note that cyclically shifting the latter composition of operators, we obtain $$\mathbf{v}_{i+1,i+2}\cdots \mathbf{v}_{23}\mathbf{v}_{12}.$$ Once again, considering the results concerning compositions of classical operators in~\cite{BS2}, it follows that there exists no nontrivial non-zero equivalences involving either of the two possible subcompositions of $\mathbf{v}$ described above. Therefore, the equivalence $\mathbf{v}\equiv \mathbf{0}$ must be minimal.
\end{proof}

Applying the cyclic shift transformation to the composition of Proposition~\ref{prop:arblength-minimal}, we obtain the following list of zero equivalences of arbitrary degree. See Figure~\ref{fig:cycshiftarblengz} for an illustration of these compositions for $n=4$.

\begin{corollary}\label{cor:pathovercs}
    Let $n\geq 4$ and consider $a_1<a_2<\cdots<a_n$. Then the following equivalences are minimal
    \begin{enumerate}
        \item[$(i)$] $\mathbf{v}_{a_{n-1}a_1}\mathbf{v}_{a_{n-2}a_{n-1}}\cdots\mathbf{v}_{a_2a_3}\mathbf{v}_{a_na_2}\equiv \mathbf{0}$,
        \item[$(ii)$] $\mathbf{v}_{a_na_2}\mathbf{v}_{a_{n-1}a_n}\cdots\mathbf{v}_{a_3a_4}\mathbf{v}_{a_1a_3}\equiv \mathbf{0}$, and
        \item[$(iii)$] $\mathbf{v}_{a_{i-1}a_{i+1}}\mathbf{v}_{a_{i-2}a_{i-1}}\cdots\mathbf{v}_{a_1a_2}\mathbf{v}_{a_na_1}\mathbf{v}_{a_{n-1}a_n}\cdots\mathbf{v}_{a_{i+2}a_{i+3}}\mathbf{v}_{a_ia_{i+2}}\equiv \mathbf{0}$ for $1<i<n-1$.
    \end{enumerate}
\end{corollary}

    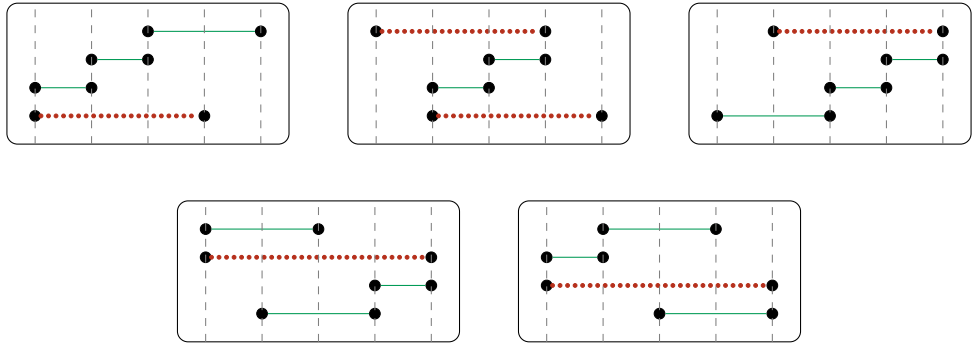
\begin{figure}[H]
        $$\begin{tikzpicture}[scale=0.75]
        \draw[rounded corners] (-0.5, 0) rectangle (4.5, 2.5) {};
        
        \node (1) at (0,0.5) [circle, draw = black,fill=black, inner sep = 0.5mm] {};
        \node (2) at (3,0.5) [circle, draw = black,fill=black, inner sep = 0.5mm] {};
        \draw[decorate sep={0.5mm}{1mm},fill, BrickRed] (1)--(2);
        
        \node (3) at (0,1) [circle, draw = black,fill=black, inner sep = 0.5mm] {};
        \node (4) at (1,1) [circle, draw = black,fill=black, inner sep = 0.5mm] {};
        \draw[ForestGreen](3)--(4);
        
        \node (5) at (1,1.5) [circle, draw = black,fill=black, inner sep = 0.5mm] {};
        \node (6) at (2,1.5) [circle, draw = black,fill=black, inner sep = 0.5mm] {};
        \draw[ForestGreen](5)--(6);

        \node (7) at (2,2) [circle, draw = black,fill=black, inner sep = 0.5mm] {};
        \node (8) at (4,2) [circle, draw = black,fill=black, inner sep = 0.5mm] {};
        \draw[ForestGreen](7)--(8);
        
        \draw[dashed,gray] (0,0)--(0,2.5);
        \draw[dashed,gray] (1,0)--(1,2.5);
        \draw[dashed,gray] (2,0)--(2,2.5);
        \draw[dashed,gray] (3,0)--(3,2.5);
        \draw[dashed,gray] (4,0)--(4,2.5);
    \end{tikzpicture}\quad\quad \begin{tikzpicture}[scale=0.75]
        \draw[rounded corners] (-0.5, 0) rectangle (4.5, 2.5) {};
        
        \node (1) at (1,0.5) [circle, draw = black,fill=black, inner sep = 0.5mm] {};
        \node (2) at (4,0.5) [circle, draw = black,fill=black, inner sep = 0.5mm] {};
        \draw[decorate sep={0.5mm}{1mm},fill, BrickRed] (1)--(2);
        
        \node (3) at (1,1) [circle, draw = black,fill=black, inner sep = 0.5mm] {};
        \node (4) at (2,1) [circle, draw = black,fill=black, inner sep = 0.5mm] {};
        \draw[ForestGreen](3)--(4);
        
        \node (5) at (2,1.5) [circle, draw = black,fill=black, inner sep = 0.5mm] {};
        \node (6) at (3,1.5) [circle, draw = black,fill=black, inner sep = 0.5mm] {};
        \draw[ForestGreen](5)--(6);

        \node (7) at (0,2) [circle, draw = black,fill=black, inner sep = 0.5mm] {};
        \node (8) at (3,2) [circle, draw = black,fill=black, inner sep = 0.5mm] {};
        \draw[decorate sep={0.5mm}{1mm},fill, BrickRed] (7)--(8);
        
        \draw[dashed,gray] (0,0)--(0,2.5);
        \draw[dashed,gray] (1,0)--(1,2.5);
        \draw[dashed,gray] (2,0)--(2,2.5);
        \draw[dashed,gray] (3,0)--(3,2.5);
        \draw[dashed,gray] (4,0)--(4,2.5);
    \end{tikzpicture}\quad\quad \begin{tikzpicture}[scale=0.75]
        \draw[rounded corners] (-0.5, 0) rectangle (4.5, 2.5) {};
        
        \node (1) at (0,0.5) [circle, draw = black,fill=black, inner sep = 0.5mm] {};
        \node (2) at (2,0.5) [circle, draw = black,fill=black, inner sep = 0.5mm] {};
        \draw[ForestGreen](1)--(2);
        
        \node (3) at (2,1) [circle, draw = black,fill=black, inner sep = 0.5mm] {};
        \node (4) at (3,1) [circle, draw = black,fill=black, inner sep = 0.5mm] {};
        \draw[ForestGreen](3)--(4);
        
        \node (5) at (3,1.5) [circle, draw = black,fill=black, inner sep = 0.5mm] {};
        \node (6) at (4,1.5) [circle, draw = black,fill=black, inner sep = 0.5mm] {};
        \draw[ForestGreen](5)--(6);

        \node (7) at (1,2) [circle, draw = black,fill=black, inner sep = 0.5mm] {};
        \node (8) at (4,2) [circle, draw = black,fill=black, inner sep = 0.5mm] {};
        \draw[decorate sep={0.5mm}{1mm},fill, BrickRed] (7)--(8);
        
        \draw[dashed,gray] (0,0)--(0,2.5);
        \draw[dashed,gray] (1,0)--(1,2.5);
        \draw[dashed,gray] (2,0)--(2,2.5);
        \draw[dashed,gray] (3,0)--(3,2.5);
        \draw[dashed,gray] (4,0)--(4,2.5);
    \end{tikzpicture}$$

    $$\begin{tikzpicture}[scale=0.75]
        \draw[rounded corners] (-0.5, 0) rectangle (4.5, 2.5) {};
        
        \node (1) at (1,0.5) [circle, draw = black,fill=black, inner sep = 0.5mm] {};
        \node (2) at (3,0.5) [circle, draw = black,fill=black, inner sep = 0.5mm] {};
        \draw[ForestGreen](1)--(2);
        
        \node (3) at (3,1) [circle, draw = black,fill=black, inner sep = 0.5mm] {};
        \node (4) at (4,1) [circle, draw = black,fill=black, inner sep = 0.5mm] {};
        \draw[ForestGreen](3)--(4);
        
        \node (5) at (0,1.5) [circle, draw = black,fill=black, inner sep = 0.5mm] {};
        \node (6) at (4,1.5) [circle, draw = black,fill=black, inner sep = 0.5mm] {};
        \draw[decorate sep={0.5mm}{1mm},fill, BrickRed] (5)--(6);

        \node (7) at (0,2) [circle, draw = black,fill=black, inner sep = 0.5mm] {};
        \node (8) at (2,2) [circle, draw = black,fill=black, inner sep = 0.5mm] {};
        \draw[ForestGreen](7)--(8);
        
        \draw[dashed,gray] (0,0)--(0,2.5);
        \draw[dashed,gray] (1,0)--(1,2.5);
        \draw[dashed,gray] (2,0)--(2,2.5);
        \draw[dashed,gray] (3,0)--(3,2.5);
        \draw[dashed,gray] (4,0)--(4,2.5);
    \end{tikzpicture}\quad\quad \begin{tikzpicture}[scale=0.75]
        \draw[rounded corners] (-0.5, 0) rectangle (4.5, 2.5) {};
        
        \node (1) at (2,0.5) [circle, draw = black,fill=black, inner sep = 0.5mm] {};
        \node (2) at (4,0.5) [circle, draw = black,fill=black, inner sep = 0.5mm] {};
        \draw[ForestGreen](1)--(2);
        
        \node (3) at (0,1) [circle, draw = black,fill=black, inner sep = 0.5mm] {};
        \node (4) at (4,1) [circle, draw = black,fill=black, inner sep = 0.5mm] {};
        \draw[decorate sep={0.5mm}{1mm},fill, BrickRed] (3)--(4);
        
        \node (5) at (0,1.5) [circle, draw = black,fill=black, inner sep = 0.5mm] {};
        \node (6) at (1,1.5) [circle, draw = black,fill=black, inner sep = 0.5mm] {};
        \draw[ForestGreen](5)--(6);

        \node (7) at (1,2) [circle, draw = black,fill=black, inner sep = 0.5mm] {};
        \node (8) at (3,2) [circle, draw = black,fill=black, inner sep = 0.5mm] {};
        \draw[ForestGreen](7)--(8);
        
        \draw[dashed,gray] (0,0)--(0,2.5);
        \draw[dashed,gray] (1,0)--(1,2.5);
        \draw[dashed,gray] (2,0)--(2,2.5);
        \draw[dashed,gray] (3,0)--(3,2.5);
        \draw[dashed,gray] (4,0)--(4,2.5);
    \end{tikzpicture}$$
    \caption{Equivalences of Corollary~\ref{cor:pathovercs} for $n=4$}
    \label{fig:cycshiftarblengz}
    \end{figure}

Moving to the non-zero equivalence of Theorem~\ref{thm:arblength}, we split the proof into two parts. In Proposition~\ref{prop:newrel} we show that the compositions are equivalent, and in Proposition~\ref{prop:minnzrel} we show that the equivalence is minimal.

\begin{prop}\label{prop:newrel}
Let $n\geq 4$ and consider $a_1<a_2<\cdots<a_n$. Then 
$$\mathbf{v}_{a_{n-1}a_1}\mathbf{v}_{a_{n-2}a_{n-1}}\mathbf{v}_{a_{n-3}a_{n-2}}\cdots \mathbf{v}_{a_{1}a_{2}}\mathbf{v}_{a_n,a_1}\equiv \mathbf{v}_{a_na_1}\mathbf{v}_{a_{n-1}a_{n}}\mathbf{v}_{a_{n-2}a_{n-1}}\cdots \mathbf{v}_{a_{2}a_{3}}\mathbf{v}_{a_na_2}.$$
\end{prop}
\begin{proof}
Considering Theorem~\ref{thm:flat}, it suffices to show that $$\mathbf{v}_{n-1,1}\mathbf{v}_{n-2,n-1}\mathbf{v}_{n-3,n-2}\cdots \mathbf{v}_{12}\mathbf{v}_{n,1}\equiv \mathbf{v}_{n,1}\mathbf{v}_{n-1,n}\mathbf{v}_{n-2,n-1}\cdots \mathbf{v}_{23}\mathbf{v}_{n,2}.$$ Let's start with 
$$\mathbf{v^*_1}=\mathbf{v}_{n-1,1}\mathbf{v}_{n-2,n-1}\mathbf{v}_{n-3,n-2}\cdots \mathbf{v}_{12}\mathbf{v}_{n,1}.$$
For $w\in S_N$ with $N\ge n$, assume that $$\mathbf{v^*_1}\bullet_kw\neq 0$$ and $w^{-1}(i)=j_i$ for $1\le i\le n$.
\begin{enumerate}
    \item[$(1)$] Since $\mathbf{v^*_1}\bullet_kw\neq 0$, it follows that $\mathbf{v}_{n,1}\bullet_kw\neq 0$. Applying Proposition~\ref{prop:covercond}, we have that $\mathbf{v}_{n,1}\bullet_kw\neq 0$ if and only if 
    \begin{itemize}
        \item $j_n\le k<j_1$ and
        \item $w^{-1}(s)<j_n$ or $w^{-1}(s)>j_1$ for $s\in [N] \backslash [1,n]$.
    \end{itemize}
    Note that in $\mathbf{v}_{n,1}\bullet_kw$, we have $1$ in position $j_n\le k$, $n$ in position $j_1>k$, and $i$ in position $j_i$ for $2\le i\le n-1$. In addition, we gain a coefficient of $\mathbf{q}_{j_n,j_1}$.

    \item[$(2)$] As in $(1)$, we must have that $\mathbf{v}_{12}\bullet_k(\mathbf{v}_{n,1}\bullet_kw)\neq 0$. Considering the form of $\mathbf{v}_{n,1}\bullet_kw$ noted above and applying Proposition~\ref{prop:covercond}, we find that $\mathbf{v}_{12}\bullet_k(\mathbf{v}_{a_na_1}\bullet_kw)\neq 0$ if and only if $k<j_2$.
    Note that in $\mathbf{v}_{12}\mathbf{v}_{n,1}\bullet_kw$, we have $1$ in position $j_2$, $2$ in position $j_n$, $n$ in position $j_1$, and $i$ in position $j_i$ for $3\le i\le n-1$.

    \item[(3)] The same reasoning used in (2) above can be applied recursively to $$\mathbf{v}_{i,i+1}\bullet_k(\mathbf{v}_{i-1,i}\cdots \mathbf{v}_{12}\mathbf{v}_{n,1}\bullet_kw)$$ for $2\le i\le n-2$. In this way, we find that $$\mathbf{v}_{i,i+1}\bullet_k(\mathbf{v}_{i-1,i}\cdots \mathbf{v}_{12}\mathbf{v}_{n,1}\bullet_kw)\neq 0$$ if and only if $k<j_{i+1}$ for $2\le i\le n-2$. Note that in $\mathbf{v}_{n-2,n-1}\mathbf{v}_{n-3,n-2}\cdots \mathbf{v}_{12}\mathbf{v}_{n,1}\bullet_kw$, we have $i$ in position $j_{i+1}$ for $1\le i\le n-2$, $n-1$ in position $j_n$, and $n$ in position $j_1$.

    \item[$(4)$] Finally, since $\mathbf{v^*_1}\bullet_kw\neq 0$, we must have $$\mathbf{v}_{n-1,1}\bullet_k(\mathbf{v}_{n-2,n-1}\mathbf{v}_{n-3,n-2}\cdots \mathbf{v}_{12}\mathbf{v}_{n,1}\bullet_kw)\neq 0.$$ Keeping in mind that 1 is in position $j_2$ and $n$ position $j_1$ of $\mathbf{v}_{n-2,n-1}\mathbf{v}_{n-3,n-2}\cdots \mathbf{v}_{12}\mathbf{v}_{n,1}\bullet_kw$, applying Proposition~\ref{prop:covercond}, we find that $$\mathbf{v}_{n-1,1}\bullet_k(\mathbf{v}_{n-2,n-1}\mathbf{v}_{n-3,n-2}\cdots \mathbf{v}_{12}\mathbf{v}_{n,1}\bullet_kw)\neq 0$$ if and only if
    \begin{itemize}
        \item $j_1>j_2>k$ and
        \item $w^{-1}(s)<j_n$ or $w^{-1}(s)>j_2$ for $s\in [N]\backslash [1,n-1]$.
    \end{itemize}
    Note that in $\mathbf{v^*_1}\bullet_kw$, we have $1$ in position $j_n$, $i$ in position $j_{i+1}$ for $2\le i\le n-2$, $n-1$ in position $j_2$, and $n$ in position $j_1$. In addition, we gain a coefficient of $\mathbf{q}_{j_n,j_2}$.
\end{enumerate}
Thus, $\mathbf{v^*_1}\bullet_kw\neq 0$ if and only if
\begin{enumerate}
    \item[$(a_1)$] $j_n\le k<j_1,\hdots,j_{n-1}$,
    \item[$(b_1)$] $j_1>j_2$, and
    \item[$(c_1)$] $w^{-1}(s)<j_n$ or $w^{-1}(s)>j_1$ for $s\in [N]\backslash [1,n]$.
\end{enumerate}
Moreover, in this case, we have $\mathbf{v^*_1}\bullet_kw=\mathbf{q}^\alpha u$, where $u^{-1}(1)=j_n$, $u^{-1}(i)=j_{i+1}$ for $2\le i\le n-2$, $u^{-1}(n-1)=j_2$, $u^{-1}(n)=j_1$, $u^{-1}(s)=w^{-1}(s)$ for $s\in [N]\backslash [n]$, and $\mathbf{q}^\alpha=\mathbf{q}_{j_n,j_2}\mathbf{q}_{j_n,j_1}$.

Now, let $$\mathbf{v^*_2}=\mathbf{v}_{n,1}\mathbf{v}_{n-1,n}\mathbf{v}_{n-2,n-1}\cdots \mathbf{v}_{23}\mathbf{v}_{n,2}$$ and assume $\mathbf{v^*_2}\bullet_kw\neq 0$ for $w\in S_N$ as above.
\begin{enumerate}
    \item[$(1)$] Since $\mathbf{v^*_2}\bullet_kw\neq 0$, it follows that $\mathbf{v}_{n,2}\bullet_kw\neq 0$. Applying Proposition~\ref{prop:covercond}, we have that $\mathbf{v}_{n,2}\bullet_kw\neq 0$ if and only if 
    \begin{itemize}
        \item $j_n\le k<j_2$,
        \item $w^{-1}(s)<j_n$ or $w^{-1}(s)>j_2$ for $s\in [N] \backslash [2,n]$.
    \end{itemize}
    As a consequence of the second condition above, we have $j_1<j_n$ or $j_1>j_2$. Note that in $\mathbf{v}_{n,2}\bullet_kw$, we have $2$ in position $j_n\le k$, $i$ in position $j_i$ for $i\in [n]\backslash\{2,n\}$, and $n$ in position $j_2>k$. In addition, we gain a coefficient of $\mathbf{q}_{j_n,j_2}$.
    
    \item[$(2)$] As in $(2)$, we must have that $\mathbf{v}_{23}\bullet_k(\mathbf{v}_{n,2}\bullet_kw)\neq 0$. Considering the form of $\mathbf{v}_{n,2}\bullet_kw$ noted above and applying Proposition~\ref{prop:covercond}, $\mathbf{v}_{23}\bullet_k(\mathbf{v}_{n,2}\bullet_kw)\neq 0$ if and only if $k<j_3$. Note that in $\mathbf{v}_{23}\mathbf{v}_{n,2}\bullet_kw$, we have $1$ in position $j_1<j_n$ or $j_1>j_2$, $2$ in position $j_3> k$, $3$ in position $j_n\le k$, $i$ in position $j_i$ for $i\in [n]\backslash\{1,2,3,n\}$, and $n$ in position $j_2>k$.

    \item[$(3)$] The reasoning used in the case above can be applied recursively to $$\mathbf{v}_{i,i+1}\bullet_k(\mathbf{v}_{i-1,i}\cdots \mathbf{v}_{23}\mathbf{v}_{n,2}\bullet_kw)$$ for $3\le i\le n-2$. In this way, we find that $$\mathbf{v}_{i,i+1}\bullet_k(\mathbf{v}_{i-1,i}\cdots \mathbf{v}_{23}\mathbf{v}_{n,2}\bullet_kw)\neq 0$$ if and only if $k<j_{i+1}$ for $2\le i\le n-2$. Note that in $\mathbf{v}_{n-2,n-1}\cdots \mathbf{v}_{23}\mathbf{v}_{n,2}\bullet_kw$, we have $1$ in position $j_1<j_n$ or $j_1>j_2$, $i$ in position $j_{i+1}$ for $2\le i\le n-2$, $n-1$ in position $j_n$, and $n$ in position $j_2$.

    \item[$(4)$] Now, since $n$ is in position $j_2$ and $n-1$ in position $j_n$ of $\mathbf{v}_{n-2,n-1}\cdots \mathbf{v}_{23}\mathbf{v}_{n,2}\bullet_kw$ with $j_n<j_2$, we have that $$\mathbf{v}_{n-1,n}\bullet_k(\mathbf{v}_{n-2,n-1}\cdots \mathbf{v}_{23}\mathbf{v}_{n,2}\bullet_kw)\neq 0.$$ Moreover, in $\mathbf{v}_{n-1,n}\bullet_k(\mathbf{v}_{n-2,n-1}\cdots \mathbf{v}_{23}\mathbf{v}_{n,2}\bullet_kw)$, we have $1$ in position $j_1<j_n$ or $j_1>j_2$, $i$ in position $j_{i+1}$ for $2\le i\le n-2$, $n-1$ in position $j_2$, and $n$ in position $j_n$.

    \item[$(5)$] Finally, applying Proposition~\ref{prop:covercond}, we have that $$\mathbf{v}_{n,1}\bullet_k(\mathbf{v}_{n-1,n}\mathbf{v}_{n-2,n-1}\cdots \mathbf{v}_{23}\mathbf{v}_{n,2}\bullet_kw)\neq 0$$ if and only if
    \begin{itemize}
        \item $k<j_2<j_1$, and
        \item $w^{-1}(s)<j_n$ or $w^{-1}(s)>j_1$ for $s\in [N]\backslash [1,n]$.
    \end{itemize}
    Note that in $\mathbf{v^*_2}\bullet_kw$, we have $1$ in position $j_n$, $i$ in position $j_{i+1}$ for $1\le i\le n-2$, $n-1$ in position $j_2$, and $n$ in position $j_1$. In addition, we gain a coefficient of $\mathbf{q}_{j_n,j_1}$.
\end{enumerate}
Thus, $\mathbf{v^*_2}\bullet_kw\neq 0$ if and only if
\begin{enumerate}
    \item[$(a_2)$] $j_n\le k< j_1,\hdots,j_{n-1}$,
    \item[$(b_2)$] $j_1>j_2$, and
    \item[$(c_2)$] $w^{-1}(s)<j_n$ or $w^{-1}(s)>j_1$ for $s\in [N]\backslash [1,n]$.
\end{enumerate}
Moreover, in this case, we have $\mathbf{v^*_2}\bullet_kw=\mathbf{q}^\alpha u$, where $u^{-1}(1)=j_n$, $u^{-1}(i)=j_{i+1}$ for $2\le i\le n-2$, $u^{-1}(n-1)=j_2$, $u^{-1}(n)=j_1$, $u^{-1}(s)=w^{-1}(s)$ for $s\in [N]\backslash [n]$, and $\mathbf{q}^\alpha=\mathbf{q}_{j_n,j_2}\mathbf{q}_{j_n,j_1}$.

Comparing conditions for $\mathbf{v}_1^*\bullet_kw\neq 0$ and $\mathbf{v}_2^*\bullet_kw\neq 0$, we see that they are exactly the same; that is, $\mathbf{v_1^*}\bullet_kw\neq 0$ if and only if $\mathbf{v_2^*}\bullet_kw\neq 0$. Moreover, our work above shows that in this case $\mathbf{v_1^*}\bullet_kw=\mathbf{v_2^*}\bullet_kw$. To see that, $\mathbf{v^*_1}$ and $\mathbf{v_2^*}$ are both non-zero, taking $w^*=(\hdots a_n~|~a_2~a_1~a_3\hdots a_{n-1}\hdots)$ it is straightforward to verify that $\mathbf{v_1^*}\bullet_kw^*=\mathbf{v_2^*}\bullet_kw^*\neq 0$. The result follows.
\end{proof}

\begin{prop}\label{prop:minnzrel}
Let $n\geq 4$ and consider $a_1<a_2<\hdots<a_n$. The following  equivalence is minimal 
$$\mathbf{v}_{a_{n-1}a_1}\mathbf{v}_{a_{n-2}a_{n-1}}\mathbf{v}_{a_{n-3}a_{n-2}}\cdots \mathbf{v}_{a_{1}a_{2}}\mathbf{v}_{a_n,a_1}\equiv \mathbf{v}_{a_na_1}\mathbf{v}_{a_{n-1}a_{n}}\mathbf{v}_{a_{n-2}a_{n-1}}\cdots \mathbf{v}_{a_{2}a_{3}}\mathbf{v}_{a_na_2}.$$
\end{prop}
\begin{proof}
    Let $$\mathbf{v}_1=\mathbf{v}_{a_{n-1}a_1}\mathbf{v}_{a_{n-2}a_{n-1}}\mathbf{v}_{a_{n-3}a_{n-2}}\cdots \mathbf{v}_{a_{1}a_{2}}\mathbf{v}_{a_n,a_1}$$ and $$\mathbf{v}_2=\mathbf{v}_{a_na_1}\mathbf{v}_{a_{n-1}a_{n}}\mathbf{v}_{a_{n-2}a_{n-1}}\cdots \mathbf{v}_{a_{2}a_{3}}\mathbf{v}_{a_na_2}.$$ It suffices to show that no subcomposition of $\mathbf{v}_1$ or $\mathbf{v}_2$ consisting of contiguous operators satisfies a nontrivial, non-zero equivalence. In fact, considering Theorems~\ref{thm:flat},~\ref{thm:hshift}, and~\ref{thm:vshift}, it suffices to show that compositions of operators of the form  
    \begin{enumerate}
        \item[(1)] $\mathbf{v}_{i,i+1}\mathbf{v}_{i-1,i}\cdots\mathbf{v}_{j,j+1}$ for $1\le i<j$,
        \item[(2)] $\mathbf{v}_{i,i+1}\mathbf{v}_{i-1,i}\cdots\mathbf{v}_{12}\mathbf{v}_{n,1}$ for $1\le i<n-2$, and
        \item[(3)] $\mathbf{v}_{n-1,n}\mathbf{v}_{n-2,n-1}\cdots \mathbf{v}_{12}\mathbf{v}_{n,1}$
    \end{enumerate}
    satisfy no nontrivial, non-zero equivalences. Note that relations of the form (2) can be cyclically shifted to one of the form (1). Consequently, applying the work concerning compositions of classical operators in~\cite{BS2}, we find that there are no nontrivial, non-zero equivalences involving operators of the form (1) or (2). As for (3), letting $$\mathbf{\widehat{v}}=\mathbf{v}_{n-1,n}\mathbf{v}_{n-2,n-1}\cdots \mathbf{v}_{12}\mathbf{v}_{n,1},$$ note that $$\mathbf{\widehat{v}}\bullet_1(n|1,2,\hdots,n-1)=q_1(n|n-1,1,2,\hdots,n-2).$$ Assume that $\mathbf{v}^*\equiv \mathbf{\widehat{v}}$. Then there exists $\{a_1,b_1,\hdots,a_n,b_n\}=[n]$ such that $$\mathbf{v}^*=\mathbf{v}_{a_nb_n}\cdots\mathbf{v}_{a_1b_1}$$ and  $$\mathbf{v}^*\bullet_1(n|1,2,\hdots,n-1)=q_1(n|n-1,1,2,\hdots,n-2).$$ Evidently, considering the form of $(n|1,2,\hdots,n-1)$ and the action $\bullet_1$, we have that $\mathbf{v}_{a_1b_1}=\mathbf{v}_{n,1}$. Thus, 
    \begin{align*}
        \mathbf{v}_{a_nb_n}\cdots\mathbf{v}_{a_2b_2}\bullet_1(1|n,2,\hdots,n-1)&=\mathbf{v}_{n-1,n}\mathbf{v}_{n-2,n-1}\cdots \mathbf{v}_{12}\bullet_1(1|n,2,\hdots,n-1)\\
        &=(n|n-1,1,2,\hdots,n-2).
    \end{align*}
    We show that $\mathbf{v}_{a_ib_i}=\mathbf{v}_{i-1,i}$ for $2\le i\le n$ in increasing order of $i$. For $i=2$, considering the form of $(1|n,2,\hdots,n-1)$, we have that either $\mathbf{v}_{a_2b_2}=\mathbf{v}_{1,n}$ or $\mathbf{v}_{12}$. Applying Proposition~\ref{prop:BS} and Remark~\ref{rem:relbrels}, we find that $$\mathbf{v}_{1,n}\bullet_1(1|n,2,\hdots,n-1)=(n|1,2,\hdots,n-1)\nless^\mathbf{q}_1(n|n-1,1,2,\hdots,n-2).$$ Consequently, $\mathbf{v}_{a_2,b_2}=\mathbf{v}_{12}$. Now, assume that $\mathbf{v}_{a_ib_i}=\mathbf{v}_{i-1,i}$ for all $2\le i\le j<n-1$ where $2<j<n-1$. Then $$\mathbf{v}_{a_jb_j}\cdots\mathbf{v}_{a_2b_2}\bullet_1(1|n,2,\hdots,n-1)=(j|n,1,2\hdots,j-1,j+1,j+2,\hdots,n-1)$$ and
    \begin{align*}
        \mathbf{v}_{a_nb_n}\cdots\mathbf{v}_{a_{j+1}b_{j+1}}\bullet_1&(j|n,1,2\hdots,j-1,j+1,j+2,\hdots,n-1)\\
        &=\mathbf{v}_{n-1,n}\cdots \mathbf{v}_{j,j+1}\bullet_1(j|n,1,2\hdots,j-1,j+1,j+2,\hdots,n-1)\\
        &=(n|n-1,1,2,\hdots,n-2).
    \end{align*}
    Considering the form of $(j|n,1,2\hdots,j-1,j+1,j+2,n-1)$, we have that either $\mathbf{v}_{a_{j+1}b_{j+1}}=\mathbf{v}_{j,n}$ or $\mathbf{v}_{j,j+1}$. Once again applying Proposition~\ref{prop:BS} and Remark~\ref{rem:relbrels}, 
    \begin{align*}
        \mathbf{v}_{j,n}\bullet_1(j|n,1,2\hdots,j-1,j+1,j+2,\hdots,n-1)&=(n|j,1,2\hdots,j-1,j+1,j+2,\hdots,n-1)\\
        &\nless^\mathbf{q}_1(n|n-1,1,2,\hdots,n-2).
    \end{align*}
    Consequently, $\mathbf{v}_{a_{j+1}b_{j+1}}=\mathbf{v}_{j,j+1}$. Finally, we move to $i=n$. Note that our work above shows that $\mathbf{v}_{a_jb_j}=\mathbf{v}_{j-1,j}$ for $2\le j\le n-1$, i.e., we showed that $\mathbf{v}^*=\mathbf{v}_{a_nb_n}\mathbf{v}_{n-2,n-1}\cdots\mathbf{v}_{12}\mathbf{v}_{n,1}$. Thus, since $$\mathbf{\widehat{v}}\bullet_1(n|1,2,\hdots,n)=\mathbf{v}^*\bullet_1(n|1,2,\hdots,n)\neq 0,$$ it follows that $\mathbf{v}_{a_nb_n}=a_{n-1,n}$. Therefore, $\mathbf{\widehat{v}}=\mathbf{v}^*$ and there are no nontrivial, non-zero equivalences involving compositions of operators of the form (3). The result follows.
\end{proof}

Applying the cyclic shift transformation to the equivalence of Proposition~\ref{prop:minnzrel}, we obtain the following list of zero equivalences of arbitrary degree. See Figure~\ref{fig:arblengcycshift} for an illustration of these compositions for $n=4$.

\begin{corollary}\label{cor:cycofnew}
Let $n\geq 4$ and consider $a_1<a_2<\cdots<a_n$. Then the following equivalences are minimal
\begin{enumerate}
    \item[$(i)$] $\mathbf{v}_{a_{n}a_2}\mathbf{v}_{a_{n-1}a_{n}}\cdots \mathbf{v}_{a_{2}a_{3}}\mathbf{v}_{a_{1}a_{2}}\equiv \mathbf{v}_{a_1a_2}\mathbf{v}_{a_na_1}\mathbf{v}_{a_{n-1}a_{n}}\cdots \mathbf{v}_{a_{3}a_{4}}\mathbf{v}_{a_{1}a_{3}}$,
    \item[$(ii)$] for $2\le j\le n-2$,
    \begin{multline*}
        \mathbf{v}_{a_{j-1}a_{j+1}}\mathbf{v}_{a_{j-2}a_{j-1}}\cdots\mathbf{v}_{a_2a_3}\mathbf{v}_{a_1a_2}\mathbf{v}_{a_na_1}\mathbf{v}_{a_{n-1}a_n}\cdots\mathbf{v}_{a_{j+1}a_{j+2}}\mathbf{v}_{a_ja_{j+1}}\\
        \equiv \mathbf{v}_{a_ja_{j+1}}\cdots \mathbf{v}_{a_2a_3}\mathbf{v}_{a_1a_2}\mathbf{v}_{a_na_1}\mathbf{v}_{a_{n-1}a_n}\cdots\mathbf{v}_{a_{j+4}a_{j+5}}\mathbf{v}_{a_{j+3}a_{j+4}}\mathbf{v}_{a_{j}a_{j+3}},
    \end{multline*}
    \item[$(iii)$] $\mathbf{v}_{a_{n-2}a_n}\mathbf{v}_{a_{n-3}a_{n-2}}\cdots\mathbf{v}_{a_2a_3}\mathbf{v}_{a_1a_2}\mathbf{v}_{a_na_1}\mathbf{v}_{a_{n-1}a_{n}}\equiv \mathbf{v}_{a_{n-1}a_n}\cdots \mathbf{v}_{a_2a_3}\mathbf{v}_{a_1a_2}\mathbf{v}_{a_{n-1}a_1}$.
\end{enumerate}
\end{corollary}

\begin{figure}[H]
    \centering
    \begin{equation*}
\begin{array}{ccc}
\begin{tikzpicture}[scale=0.75]
        \draw[rounded corners] (-0.5, 0) rectangle (3.5, 2.5) {};
        
        \node (1) at (0,0.5) [circle, draw = black,fill=black, inner sep = 0.5mm] {};
        \node (2) at (1,0.5) [circle, draw = black,fill=black, inner sep = 0.5mm] {};
        \draw[ForestGreen] (1)--(2);
        
        \node (3) at (1,1) [circle, draw = black,fill=black, inner sep = 0.5mm] {};
        \node (4) at (2,1) [circle, draw = black,fill=black, inner sep = 0.5mm] {};
        \draw[ForestGreen](3)--(4);
        
        \node (5) at (2,1.5) [circle, draw = black,fill=black, inner sep = 0.5mm] {};
        \node (6) at (3,1.5) [circle, draw = black,fill=black, inner sep = 0.5mm] {};
        \draw[ForestGreen](5)--(6);

        \node (7) at (1,2) [circle, draw = black,fill=black, inner sep = 0.5mm] {};
        \node (8) at (3,2) [circle, draw = black,fill=black, inner sep = 0.5mm] {};
        \draw[decorate sep={0.5mm}{1mm},fill, BrickRed] (7)--(8);
        
        \draw[dashed,gray] (0,0)--(0,2.5);
        \draw[dashed,gray] (1,0)--(1,2.5);
        \draw[dashed,gray] (2,0)--(2,2.5);
        \draw[dashed,gray] (3,0)--(3,2.5);
    \end{tikzpicture}\hspace{0.25cm}\begin{tikzpicture}
        \node at (0,0){};
        \node at (0,0.75){\large{$\equiv$}};
    \end{tikzpicture}\hspace{0.25cm}\begin{tikzpicture}[scale=0.8]
        \draw[rounded corners] (-0.5, 0) rectangle (3.5, 2.5) {};
        
        \node (1) at (0,0.5) [circle, draw = black,fill=black, inner sep = 0.5mm] {};
        \node (2) at (2,0.5) [circle, draw = black,fill=black, inner sep = 0.5mm] {};
        \draw[ForestGreen] (1)--(2);
        
        \node (3) at (2,1) [circle, draw = black,fill=black, inner sep = 0.5mm] {};
        \node (4) at (3,1) [circle, draw = black,fill=black, inner sep = 0.5mm] {};
        \draw[ForestGreen](3)--(4);
        
        \node (5) at (0,1.5) [circle, draw = black,fill=black, inner sep = 0.5mm] {};
        \node (6) at (3,1.5) [circle, draw = black,fill=black, inner sep = 0.5mm] {};
        \draw[decorate sep={0.5mm}{1mm},fill, BrickRed](5)--(6);

        \node (7) at (0,2) [circle, draw = black,fill=black, inner sep = 0.5mm] {};
        \node (8) at (1,2) [circle, draw = black,fill=black, inner sep = 0.5mm] {};
        \draw[ForestGreen] (7)--(8);
        
        \draw[dashed,gray] (0,0)--(0,2.5);
        \draw[dashed,gray] (1,0)--(1,2.5);
        \draw[dashed,gray] (2,0)--(2,2.5);
        \draw[dashed,gray] (3,0)--(3,2.5);
    \end{tikzpicture}
     &\hspace{0.5cm} &
\begin{tikzpicture}[scale=0.75]
        \draw[rounded corners] (-0.5, 0) rectangle (3.5, 2.5) {};
        
        \node (1) at (1,0.5) [circle, draw = black,fill=black, inner sep = 0.5mm] {};
        \node (2) at (2,0.5) [circle, draw = black,fill=black, inner sep = 0.5mm] {};
        \draw[ForestGreen] (1)--(2);
        
        \node (3) at (2,1) [circle, draw = black,fill=black, inner sep = 0.5mm] {};
        \node (4) at (3,1) [circle, draw = black,fill=black, inner sep = 0.5mm] {};
        \draw[ForestGreen](3)--(4);
        
        \node (5) at (0,1.5) [circle, draw = black,fill=black, inner sep = 0.5mm] {};
        \node (6) at (3,1.5) [circle, draw = black,fill=black, inner sep = 0.5mm] {};
        \draw[decorate sep={0.5mm}{1mm},fill, BrickRed](5)--(6);

        \node (7) at (0,2) [circle, draw = black,fill=black, inner sep = 0.5mm] {};
        \node (8) at (2,2) [circle, draw = black,fill=black, inner sep = 0.5mm] {};
        \draw[ForestGreen] (7)--(8);
        
        \draw[dashed,gray] (0,0)--(0,2.5);
        \draw[dashed,gray] (1,0)--(1,2.5);
        \draw[dashed,gray] (2,0)--(2,2.5);
        \draw[dashed,gray] (3,0)--(3,2.5);
    \end{tikzpicture}\hspace{0.25cm}\begin{tikzpicture}
        \node at (0,0){};
        \node at (0,0.75){\large{$\equiv$}};
    \end{tikzpicture}\hspace{0.25cm}\begin{tikzpicture}[scale=0.8]
        \draw[rounded corners] (-0.5, 0) rectangle (3.5, 2.5) {};
        
        \node (1) at (1,0.5) [circle, draw = black,fill=black, inner sep = 0.5mm] {};
        \node (2) at (3,0.5) [circle, draw = black,fill=black, inner sep = 0.5mm] {};
        \draw[ForestGreen] (1)--(2);
        
        \node (3) at (0,1) [circle, draw = black,fill=black, inner sep = 0.5mm] {};
        \node (4) at (3,1) [circle, draw = black,fill=black, inner sep = 0.5mm] {};
        \draw[decorate sep={0.5mm}{1mm},fill, BrickRed](3)--(4);
        
        \node (5) at (0,1.5) [circle, draw = black,fill=black, inner sep = 0.5mm] {};
        \node (6) at (1,1.5) [circle, draw = black,fill=black, inner sep = 0.5mm] {};
        \draw[ForestGreen](5)--(6);

        \node (7) at (1,2) [circle, draw = black,fill=black, inner sep = 0.5mm] {};
        \node (8) at (2,2) [circle, draw = black,fill=black, inner sep = 0.5mm] {};
        \draw[ForestGreen] (7)--(8);
        
        \draw[dashed,gray] (0,0)--(0,2.5);
        \draw[dashed,gray] (1,0)--(1,2.5);
        \draw[dashed,gray] (2,0)--(2,2.5);
        \draw[dashed,gray] (3,0)--(3,2.5);
    \end{tikzpicture}
\end{array}
\end{equation*}
$$\begin{tikzpicture}[scale=0.75]
        \draw[rounded corners] (-0.5, 0) rectangle (3.5, 2.5) {};
        
        \node (1) at (2,0.5) [circle, draw = black,fill=black, inner sep = 0.5mm] {};
        \node (2) at (3,0.5) [circle, draw = black,fill=black, inner sep = 0.5mm] {};
        \draw[ForestGreen] (1)--(2);
        
        \node (3) at (0,1) [circle, draw = black,fill=black, inner sep = 0.5mm] {};
        \node (4) at (3,1) [circle, draw = black,fill=black, inner sep = 0.5mm] {};
        \draw[decorate sep={0.5mm}{1mm},fill, BrickRed](3)--(4);
        
        \node (5) at (0,1.5) [circle, draw = black,fill=black, inner sep = 0.5mm] {};
        \node (6) at (1,1.5) [circle, draw = black,fill=black, inner sep = 0.5mm] {};
        \draw[ForestGreen](5)--(6);

        \node (7) at (1,2) [circle, draw = black,fill=black, inner sep = 0.5mm] {};
        \node (8) at (3,2) [circle, draw = black,fill=black, inner sep = 0.5mm] {};
        \draw[ForestGreen] (7)--(8);
        
        \draw[dashed,gray] (0,0)--(0,2.5);
        \draw[dashed,gray] (1,0)--(1,2.5);
        \draw[dashed,gray] (2,0)--(2,2.5);
        \draw[dashed,gray] (3,0)--(3,2.5);
    \end{tikzpicture}\hspace{0.25cm}\begin{tikzpicture}
        \node at (0,0){};
        \node at (0,0.75){\large{$\equiv$}};
    \end{tikzpicture}\hspace{0.25cm}\begin{tikzpicture}[scale=0.8]
        \draw[rounded corners] (-0.5, 0) rectangle (3.5, 2.5) {};
        
        \node (1) at (0,0.5) [circle, draw = black,fill=black, inner sep = 0.5mm] {};
        \node (2) at (2,0.5) [circle, draw = black,fill=black, inner sep = 0.5mm] {};
        \draw[decorate sep={0.5mm}{1mm},fill, BrickRed] (1)--(2);
        
        \node (3) at (0,1) [circle, draw = black,fill=black, inner sep = 0.5mm] {};
        \node (4) at (1,1) [circle, draw = black,fill=black, inner sep = 0.5mm] {};
        \draw[ForestGreen](3)--(4);
        
        \node (5) at (1,1.5) [circle, draw = black,fill=black, inner sep = 0.5mm] {};
        \node (6) at (2,1.5) [circle, draw = black,fill=black, inner sep = 0.5mm] {};
        \draw[ForestGreen](5)--(6);

        \node (7) at (2,2) [circle, draw = black,fill=black, inner sep = 0.5mm] {};
        \node (8) at (3,2) [circle, draw = black,fill=black, inner sep = 0.5mm] {};
        \draw[ForestGreen] (7)--(8);
        
        \draw[dashed,gray] (0,0)--(0,2.5);
        \draw[dashed,gray] (1,0)--(1,2.5);
        \draw[dashed,gray] (2,0)--(2,2.5);
        \draw[dashed,gray] (3,0)--(3,2.5);
    \end{tikzpicture}$$
    \caption{Equivalences of Corollary~\ref{cor:cycofnew} for $n=4$}
    \label{fig:arblengcycshift}
\end{figure}
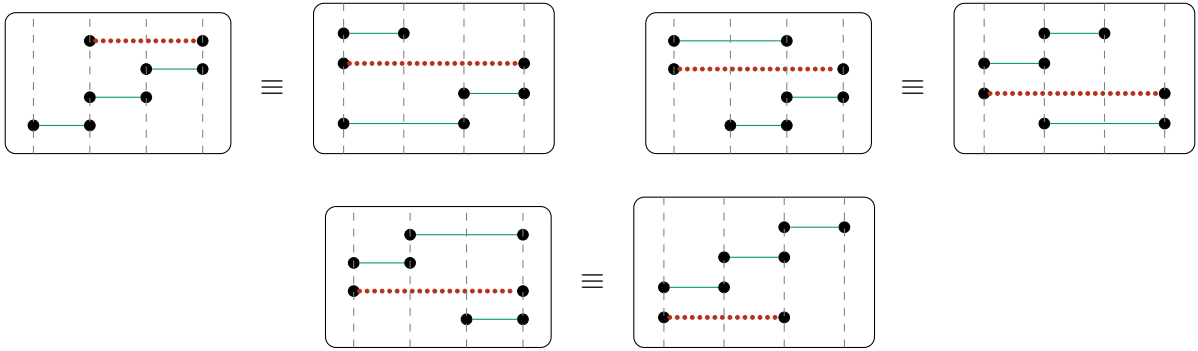

Based on experimental evidence, we conjecture that the zero and non-zero equivalences presented in this study form a complete list of equivalences.

\begin{conj}\label{conj}
    The complete collection of minimal equivalences satisfied by the elements of $\mathcal{F}_n^{\bf q}$ consists of those equivalences illustrated in Figure~\ref{fig:conj} along with all those related to them by the application of the equivalence-preserving transformations of Section~\ref{sec:ops}.
    
    \begin{figure}[H]
        \centering
        $$\begin{array}{ccc}
    \begin{tikzpicture}[scale=0.65]
    \draw[rounded corners] (-0.5, 0) rectangle (3.5, 1.5) {};
        
    \node (1) at (0,1) [circle, draw = black,fill=black, inner sep = 0.5mm] {};
        \node (2) at (1,1) [circle, draw = black,fill=black, inner sep = 0.5mm] {};
        \draw[ForestGreen](1)--(2);
        
        \node (3) at (2,0.5) [circle, draw = black,fill=black, inner sep = 0.5mm] {};
        \node (4) at (3,0.5) [circle, draw = black,fill=black, inner sep = 0.5mm] {};
        \draw[ForestGreen](3)--(4);

        \draw[dashed,gray] (0,0)--(0,1.5);
        \draw[dashed,gray] (1,0)--(1,1.5);
        \draw[dashed,gray] (2,0)--(2,1.5);
        \draw[dashed,gray] (3,0)--(3,1.5);
    \end{tikzpicture}\hspace{0.15cm}\begin{tikzpicture}
        \node at (0,0){};
        \node at (0,0.5){\large{$\equiv$}};
    \end{tikzpicture}\hspace{0.15cm}
    \begin{tikzpicture}[scale=0.65]
        \draw[rounded corners] (-0.5, 0) rectangle (3.5, 1.5) {};
        
        \node (1) at (0,0.5) [circle, draw = black,fill=black, inner sep = 0.5mm] {};
        \node (2) at (1,0.5) [circle, draw = black,fill=black, inner sep = 0.5mm] {};
        \draw[ForestGreen](1)--(2);
        
        \node (3) at (2,1) [circle, draw = black,fill=black, inner sep = 0.5mm] {};
        \node (4) at (3,1) [circle, draw = black,fill=black, inner sep = 0.5mm] {};
        \draw[ForestGreen](3)--(4);

        \draw[dashed,gray] (0,0)--(0,1.5);
        \draw[dashed,gray] (1,0)--(1,1.5);
        \draw[dashed,gray] (2,0)--(2,1.5);
        \draw[dashed,gray] (3,0)--(3,1.5);
    \end{tikzpicture} & \hspace{0.5cm} &
    \begin{tikzpicture}[scale=0.65]
        \draw[rounded corners] (-0.5, 0) rectangle (3.5, 1.5) {};
        
        \node (1) at (1,0.5) [circle, draw = black,fill=black, inner sep = 0.5mm] {};
        \node (2) at (2,0.5) [circle, draw = black,fill=black, inner sep = 0.5mm] {};
        \draw[ForestGreen](1)--(2);
        
        \node (3) at (0,1) [circle, draw = black,fill=black, inner sep = 0.5mm] {};
        \node (4) at (3,1) [circle, draw = black,fill=black, inner sep = 0.5mm] {};
        \draw[ForestGreen](3)--(4);

        \draw[dashed,gray] (0,0)--(0,1.5);
        \draw[dashed,gray] (1,0)--(1,1.5);
        \draw[dashed,gray] (2,0)--(2,1.5);
        \draw[dashed,gray] (3,0)--(3,1.5);
    \end{tikzpicture}\hspace{0.15cm}
    \begin{tikzpicture}
        \node at (0,0){};
        \node at (0,0.5){\large{$\equiv$}};
    \end{tikzpicture}\hspace{0.15cm}
    \begin{tikzpicture}[scale=0.65]
        \draw[rounded corners] (-0.5, 0) rectangle (3.5, 1.5) {};
        
        \node (1) at (1,1) [circle, draw = black,fill=black, inner sep = 0.5mm] {};
        \node (2) at (2,1) [circle, draw = black,fill=black, inner sep = 0.5mm] {};
        \draw[ForestGreen](1)--(2);
        
        \node (3) at (0,0.5) [circle, draw = black,fill=black, inner sep = 0.5mm] {};
        \node (4) at (3,0.5) [circle, draw = black,fill=black, inner sep = 0.5mm] {};
        \draw[ForestGreen](3)--(4);

        \draw[dashed,gray] (0,0)--(0,1.5);
        \draw[dashed,gray] (1,0)--(1,1.5);
        \draw[dashed,gray] (2,0)--(2,1.5);
        \draw[dashed,gray] (3,0)--(3,1.5);
    \end{tikzpicture} \\
    \begin{tikzpicture}[scale=0.65]
        \draw[rounded corners] (-0.5, 0) rectangle (3.5, 2) {};
        
        \node (1) at (1,0.5) [circle, draw = black,fill=black, inner sep = 0.5mm] {};
        \node (2) at (2,0.5) [circle, draw = black,fill=black, inner sep = 0.5mm] {};
        \draw[ForestGreen] (1)--(2);
        
        \node (3) at (2,1) [circle, draw = black,fill=black, inner sep = 0.5mm] {};
        \node (4) at (3,1) [circle, draw = black,fill=black, inner sep = 0.5mm] {};
        \draw[ForestGreen] (3)--(4);
        
        \node (5) at (0,1.5) [circle, draw = black,fill=black, inner sep = 0.5mm] {};
        \node (6) at (2,1.5) [circle, draw = black,fill=black, inner sep = 0.5mm] {};
        \draw[ForestGreen] (5)--(6);
        
        \draw[dashed,gray] (0,0)--(0,2);
        \draw[dashed,gray] (1,0)--(1,2);
        \draw[dashed,gray] (2,0)--(2,2);
        \draw[dashed,gray] (3,0)--(3,2);
    \end{tikzpicture}\hspace{0.15cm}\begin{tikzpicture}
        \node at (0,0){};
        \node at (0,0.65){\large{$\equiv$}};
    \end{tikzpicture}\hspace{0.15cm}
    \begin{tikzpicture}[scale=0.65]
        \draw[rounded corners] (-0.5, 0) rectangle (3.5, 2) {};
        
        \node (1) at (1,0.5) [circle, draw = black,fill=black, inner sep = 0.5mm] {};
        \node (2) at (3,0.5) [circle, draw = black,fill=black, inner sep = 0.5mm] {};
        \draw[ForestGreen] (1)--(2);
        
        \node (3) at (0,1) [circle, draw = black,fill=black, inner sep = 0.5mm] {};
        \node (4) at (1,1) [circle, draw = black,fill=black, inner sep = 0.5mm] {};
        \draw[ForestGreen] (3)--(4);
        
        \node (5) at (1,1.5) [circle, draw = black,fill=black, inner sep = 0.5mm] {};
        \node (6) at (2,1.5) [circle, draw = black,fill=black, inner sep = 0.5mm] {};
        \draw[ForestGreen] (5)--(6);
        
        \draw[dashed,gray] (0,0)--(0,2);
        \draw[dashed,gray] (1,0)--(1,2);
        \draw[dashed,gray] (2,0)--(2,2);
        \draw[dashed,gray] (3,0)--(3,2);
    \end{tikzpicture} &\hspace{0.5cm} &
    \begin{tikzpicture}[scale=0.65]
        \draw[rounded corners] (-0.5, 0) rectangle (2.5, 2) {};
        
        \node (1) at (0,0.5) [circle, draw = black,fill=black, inner sep = 0.5mm] {};
        \node (2) at (1,0.5) [circle, draw = black,fill=black, inner sep = 0.5mm] {};
        \draw[decorate sep={0.5mm}{1mm},fill, BrickRed] (1)--(2);
        
        \node (3) at (0,1) [circle, draw = black,fill=black, inner sep = 0.5mm] {};
        \node (4) at (1,1) [circle, draw = black,fill=black, inner sep = 0.5mm] {};
        \draw[ForestGreen](3)--(4);
        
        \node (5) at (1,1.5) [circle, draw = black,fill=black, inner sep = 0.5mm] {};
        \node (6) at (2,1.5) [circle, draw = black,fill=black, inner sep = 0.5mm] {};
        \draw[ForestGreen](5)--(6);

        \draw[dashed,gray] (0,0)--(0,2);
        \draw[dashed,gray] (1,0)--(1,2);
        \draw[dashed,gray] (2,0)--(2,2);
    \end{tikzpicture}\hspace{0.15cm}\begin{tikzpicture}
        \node at (0,0){};
        \node at (0,0.65){\large{$\equiv$}};
    \end{tikzpicture}\hspace{0.15cm}
    \begin{tikzpicture}[scale=0.65]
        \draw[rounded corners] (-0.5, 0) rectangle (2.5, 2) {};
        
        \node (1) at (0,1) [circle, draw = black,fill=black, inner sep = 0.5mm] {};
        \node (2) at (2,1) [circle, draw = black,fill=black, inner sep = 0.5mm] {};
        \draw[decorate sep={0.5mm}{1mm},fill, BrickRed] (1)--(2);
        
        \node (3) at (0,1.5) [circle, draw = black,fill=black, inner sep = 0.5mm] {};
        \node (4) at (2,1.5) [circle, draw = black,fill=black, inner sep = 0.5mm] {};
        \draw[ForestGreen](3)--(4);
        
        \node (5) at (1,0.5) [circle, draw = black,fill=black, inner sep = 0.5mm] {};
        \node (6) at (2,0.5) [circle, draw = black,fill=black, inner sep = 0.5mm] {};
        \draw[ForestGreen](5)--(6);

        \draw[dashed,gray] (0,0)--(0,2);
        \draw[dashed,gray] (1,0)--(1,2);
        \draw[dashed,gray] (2,0)--(2,2);
    \end{tikzpicture} \\
    \begin{tikzpicture}[scale=0.65]
        \node at (0,0) {};
        \draw[rounded corners] (-0.5, 0) rectangle (3.5, 1.5) {};
        
        \node (1) at (1,1) [circle, draw = black,fill=black, inner sep = 0.5mm] {};
        \node (2) at (3,1) [circle, draw = black,fill=black, inner sep = 0.5mm] {};
        \draw[ForestGreen](1)--(2);
        
        \node (3) at (0,0.5) [circle, draw = black,fill=black, inner sep = 0.5mm] {};
        \node (4) at (2,0.5) [circle, draw = black,fill=black, inner sep = 0.5mm] {};
        \draw[ForestGreen](3)--(4);

        \draw[dashed,gray] (0,0)--(0,1.5);
        \draw[dashed,gray] (1,0)--(1,1.5);
        \draw[dashed,gray] (2,0)--(2,1.5);
        \draw[dashed,gray] (3,0)--(3,1.5);
    \end{tikzpicture}\hspace{0.15cm}\begin{tikzpicture}
        \node at (0,0){};
        \node at (0,0.5){\large{$\equiv$}};
    \end{tikzpicture}\hspace{0.15cm}
    \begin{tikzpicture}[scale=0.65]
        \node at (0,0) {};
        \draw[rounded corners] (-0.5, 0) rectangle (2.5, 1.5) {};
        
        \node (1) at (0,1) [circle, draw = black,fill=black, inner sep = 0.5mm] {};
        \node (2) at (2,1) [circle, draw = black,fill=black, inner sep = 0.5mm] {};
        \draw[ForestGreen](1)--(2);
        
        \node (3) at (0,0.5) [circle, draw = black,fill=black, inner sep = 0.5mm] {};
        \node (4) at (1,0.5) [circle, draw = black,fill=black, inner sep = 0.5mm] {};
        \draw[ForestGreen](3)--(4);

        \draw[dashed,gray] (0,0)--(0,1.5);
        \draw[dashed,gray] (1,0)--(1,1.5);
        \draw[dashed,gray] (2,0)--(2,1.5);
    \end{tikzpicture}\hspace{0.15cm}\begin{tikzpicture}
        \node at (0,0){};
        \node at (0,0.5){\large{$\equiv$}};
    \end{tikzpicture}\hspace{0.15cm}
    \begin{tikzpicture}[scale=0.65]
        \node at (0,0) {};
        \draw[rounded corners] (-0.5, 0) rectangle (1.5, 1.5) {};
        
        \node (1) at (0,1) [circle, draw = black,fill=black, inner sep = 0.5mm] {};
        \node (2) at (1,1) [circle, draw = black,fill=black, inner sep = 0.5mm] {};
        \draw[ForestGreen](1)--(2);
        
        \node (3) at (0,0.5) [circle, draw = black,fill=black, inner sep = 0.5mm] {};
        \node (4) at (1,0.5) [circle, draw = black,fill=black, inner sep = 0.5mm] {};
        \draw[ForestGreen](3)--(4);

        \draw[dashed,gray] (0,0)--(0,1.5);
        \draw[dashed,gray] (1,0)--(1,1.5);
    \end{tikzpicture}\hspace{0.15cm}\begin{tikzpicture}
        \node at (0,0){};
        \node at (0,0.5){\large{$\equiv \hspace{0.15cm} \mathbf{0}$}};
    \end{tikzpicture} &\hspace{0.5cm} &
    \begin{tikzpicture}[scale=0.65]
        \draw[rounded corners] (-0.5, 0) rectangle (2.5, 2) {};
    
        \node (6) at (1,1.5) [circle, draw = black,fill=black, inner sep = 0.5mm] {};
        \node (7) at (2,1.5) [circle, draw = black,fill=black, inner sep = 0.5mm] {};
        \draw[ForestGreen](6)--(7);

        \node (12) at (0,1) [circle, draw = black,fill=black, inner sep = 0.5mm] {};
        \node (13) at (1,1) [circle, draw = black,fill=black, inner sep = 0.5mm] {};
        \draw[ForestGreen](12)--(13);
        
        \node (14) at (1,0.5) [circle, draw = black,fill=black, inner sep = 0.5mm] {};
        \node (15) at (2,0.5) [circle, draw = black,fill=black, inner sep = 0.5mm] {};
        \draw[ForestGreen](14)--(15);

        \draw[dashed,gray] (0,0)--(0,2);
        \draw[dashed,gray] (1,0)--(1,2);
        \draw[dashed,gray] (2,0)--(2,2);
    \end{tikzpicture}\hspace{0.15cm}\begin{tikzpicture}
        \node at (0,0){};
        \node at (0,0.65){\large{$\equiv$}};
    \end{tikzpicture}\hspace{0.15cm}
    \begin{tikzpicture}[scale=0.65]
        \draw[rounded corners] (-0.5, 0) rectangle (3.5, 2) {};
        
        \node (1) at (0,0.5) [circle, draw = black,fill=black, inner sep = 0.5mm] {};
        \node (2) at (2,0.5) [circle, draw = black,fill=black, inner sep = 0.5mm] {};
        \draw[decorate sep={0.5mm}{1mm},fill, BrickRed] (1)--(2);
        
        \node (3) at (0,1) [circle, draw = black,fill=black, inner sep = 0.5mm] {};
        \node (4) at (1,1) [circle, draw = black,fill=black, inner sep = 0.5mm] {};
        \draw[ForestGreen](3)--(4);
        
        \node (5) at (1,1.5) [circle, draw = black,fill=black, inner sep = 0.5mm] {};
        \node (6) at (3,1.5) [circle, draw = black,fill=black, inner sep = 0.5mm] {};
        \draw[ForestGreen](5)--(6);

        \draw[dashed,gray] (0,0)--(0,2);
        \draw[dashed,gray] (1,0)--(1,2);
        \draw[dashed,gray] (2,0)--(2,2);
        \draw[dashed,gray] (3,0)--(3,2);
    \end{tikzpicture}\begin{tikzpicture}
        \node at (0,0){};
        \node at (0,0.65){\large{$\equiv \hspace{0.15cm} \mathbf{0}$}};
    \end{tikzpicture} \\
    \begin{tikzpicture}[scale=0.65]
        \draw[rounded corners] (-0.5, 0) rectangle (4.5, 3) {};
        
        \node (1) at (0,0.5) [circle, draw = black,fill=black, inner sep = 0.5mm] {};
        \node (2) at (4,0.5) [circle, draw = black,fill=black, inner sep = 0.5mm] {};
        \draw[decorate sep={0.5mm}{1mm},fill, BrickRed] (1)--(2);
        
        \node (3) at (0,1) [circle, draw = black,fill=black, inner sep = 0.5mm] {};
        \node (4) at (1,1) [circle, draw = black,fill=black, inner sep = 0.5mm] {};
        \draw[ForestGreen](3)--(4);
        
        \draw[decorate sep={0.3mm}{1mm},fill,black] (1.25,1.25)--(1.75,1.75);

        \node (7) at (2,2) [circle, draw = black,fill=black, inner sep = 0.5mm] {};
        \node (8) at (3,2) [circle, draw = black,fill=black, inner sep = 0.5mm] {};
        \draw[ForestGreen](7)--(8);

        \node (9) at (0,2.5) [circle, draw = black,fill=black, inner sep = 0.5mm] {};
        \node (10) at (3,2.5) [circle, draw = black,fill=black, inner sep = 0.5mm] {};
        \draw[decorate sep={0.5mm}{1mm},fill, BrickRed] (9)--(10);
        
        \draw[dashed,gray] (0,0)--(0,3);
        \draw[dashed,gray] (1,0)--(1,3);
        \draw[dashed,gray] (2,0)--(2,3);
        \draw[dashed,gray] (3,0)--(3,3);
        \draw[dashed,gray] (4,0)--(4,3);
    \end{tikzpicture}\hspace{0.15cm}\begin{tikzpicture}
        \node at (0,0){};
        \node at (0,1){\large{$\equiv$}};
    \end{tikzpicture}\hspace{0.15cm}
    \begin{tikzpicture}[scale=0.65]
        \draw[rounded corners] (-0.5, 0) rectangle (4.5, 3) {};
        
        \node (1) at (1,0.5) [circle, draw = black,fill=black, inner sep = 0.5mm] {};
        \node (2) at (4,0.5) [circle, draw = black,fill=black, inner sep = 0.5mm] {};
        \draw[decorate sep={0.5mm}{1mm},fill, BrickRed] (1)--(2);
        
        \node (3) at (1,1) [circle, draw = black,fill=black, inner sep = 0.5mm] {};
        \node (4) at (2,1) [circle, draw = black,fill=black, inner sep = 0.5mm] {};
        \draw[ForestGreen](3)--(4);

        \draw[decorate sep={0.3mm}{1mm},fill,black] (2.25,1.25)--(2.75,1.75);

        \node (7) at (3,2) [circle, draw = black,fill=black, inner sep = 0.5mm] {};
        \node (8) at (4,2) [circle, draw = black,fill=black, inner sep = 0.5mm] {};
        \draw[ForestGreen](7)--(8);

        \node (9) at (0,2.5) [circle, draw = black,fill=black, inner sep = 0.5mm] {};
        \node (10) at (4,2.5) [circle, draw = black,fill=black, inner sep = 0.5mm] {};
        \draw[decorate sep={0.5mm}{1mm},fill, BrickRed] (9)--(10);
        
        \draw[dashed,gray] (0,0)--(0,3);
        \draw[dashed,gray] (1,0)--(1,3);
        \draw[dashed,gray] (2,0)--(2,3);
        \draw[dashed,gray] (3,0)--(3,3);
        \draw[dashed,gray] (4,0)--(4,3);
    \end{tikzpicture} &\hspace{0.5cm} &
    \begin{tikzpicture}[scale=0.65]
        \draw[rounded corners] (-0.5, 0) rectangle (5.5, 3) {};
        
        \node (1) at (0,0.5) [circle, draw = black,fill=black, inner sep = 0.5mm] {};
        \node (2) at (4,0.5) [circle, draw = black,fill=black, inner sep = 0.5mm] {};
        \draw[decorate sep={0.5mm}{1mm},fill, BrickRed] (1)--(2);
        
        \node (3) at (0,1) [circle, draw = black,fill=black, inner sep = 0.5mm] {};
        \node (4) at (1,1) [circle, draw = black,fill=black, inner sep = 0.5mm] {};
        \draw[ForestGreen](3)--(4);

        \draw[decorate sep={0.3mm}{1mm},fill,black] (1.25,1.25)--(1.75,1.75);
        
        \node (5) at (2,2) [circle, draw = black,fill=black, inner sep = 0.5mm] {};
        \node (6) at (3,2) [circle, draw = black,fill=black, inner sep = 0.5mm] {};
        \draw[ForestGreen](5)--(6);

        \node (7) at (3,2.5) [circle, draw = black,fill=black, inner sep = 0.5mm] {};
        \node (8) at (5,2.5) [circle, draw = black,fill=black, inner sep = 0.5mm] {};
        \draw[ForestGreen](7)--(8);
        
        \draw[dashed,gray] (0,0)--(0,3);
        \draw[dashed,gray] (1,0)--(1,3);
        \draw[dashed,gray] (2,0)--(2,3);
        \draw[dashed,gray] (3,0)--(3,3);
        \draw[dashed,gray] (4,0)--(4,3);
        \draw[dashed,gray] (5,0)--(5,3);
    \end{tikzpicture}\hspace{0.15cm}
    \begin{tikzpicture}
        \node at (0,0){};
        \node at (0,0.5){\large{$\equiv \hspace{0.15cm} \mathbf{0}$}};
    \end{tikzpicture}
    \end{array}$$
        \caption{Equivalences conjectured to generate all those defining $\mathcal{F}_n^{\bf q}$}
        \label{fig:conj}
    \end{figure}
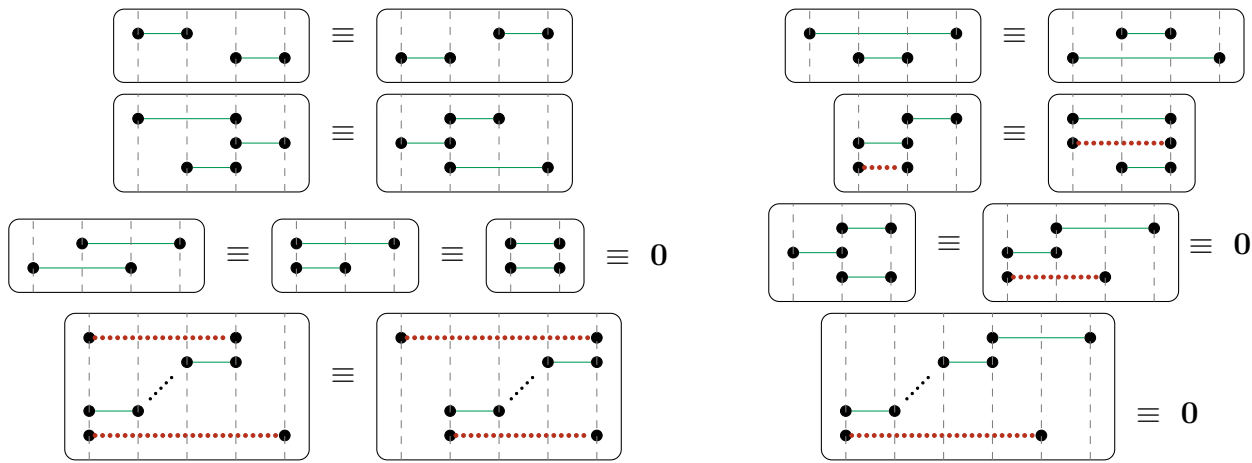
\end{conj}

In the classical case, an understanding of the chain structure of the $k$-Bruhat order played a pivotal role in proving the classical version of Conjecture~\ref{conj}, i.e., Theorem~\ref{thm:BS2}~(a). Specifically, the proof utilized a canonical choice of maximal chain between related elements in the $k$-Bruhat order. For the quantum $k$-Bruhat order, we can't even tell when two elements are related in general, let alone identify a choice of maximal chain between two elements. In future work, the present authors plan to study the quantum $k$-Bruhat order further in hopes of proving Conjecture~\ref{conj}.

\bibliographystyle{abbrv}
\bibliography{Quantum}

@article {BS1,
    AUTHOR = {Bergeron, N. and Sottile, F.},
     TITLE = {Schubert polynomials, the {B}ruhat order, and the geometry of flag manifolds.},
   JOURNAL = {Duke Math. J.},
  FJOURNAL = {Duke Mathematical Journal},
    NUMBER = {94},
      YEAR = {1998},
     PAGES = {373--423},
}

@article {BS2,
    AUTHOR = {Bergeron, N. and Sottile, F.},
     TITLE = {A monoid for the {G}rassmannian {B}ruhat order.},
   JOURNAL = {European J. Combin.},
  FJOURNAL = {The European Journal of Combinatorics},
    VOLUME={20},
    NUMBER = {3},
      YEAR = {1999},
     PAGES = {197--211},
}

@article {qkB1,
    AUTHOR = {Benedetti, C. and  Bergeron, N. and Colmenarejo, L. and Saliola, F.V. and Sottile, F.},
     TITLE = {A quantum {M}urnaghan--{N}akayama rule for the flag manifold},
   JOURNAL = {Algebraic Combinatorics},
  FJOURNAL = {Algebraic Combinatorics},
    VOLUME={8},
    NUMBER = {3},
      YEAR = {2025},
     PAGES = {619--653},
}

@article {Quantum,
    AUTHOR = {Fomin, S. and Gelfand, S. and Postnikov, A.},
     TITLE = {Quantum {S}chubert polynomials.},
   JOURNAL = {J. Amer. Math. Soc.},
    VOLUME={10},
    NUMBER = {3},
      YEAR = {1997},
     PAGES = {565--596},
}

@article {Monk,
    AUTHOR = {Monk, D.},
     TITLE = {The geometry of flag manifolds.},
   JOURNAL = {Proc. London Math. Soc.},
    VOLUME={9},
      YEAR = {1959},
     PAGES = {253--286},
}

@book {McDonald,
    AUTHOR = {Monk, D.},
     TITLE = {Notes on {S}chubert polynomials},
   PUBLISHER = {Publications LACIM},
    CITY={Montreal},
    YEAR={1991}
}

@inproceedings{Macdonald,
  title={Notes on Schubert polynomials},
  author={Ian G. MacDonald},
  year={1991},
  url={https://api.semanticscholar.org/CorpusID:117370502}
}

@article {Schubert,
    AUTHOR = {Lascoux, A. and Sch\"utzenberger, M.P.},
     TITLE = {Polyn\^omes de {S}chubert.},
   JOURNAL = {C. R. Math. Acad. Sci. Paris, S\'er. I Math.},
    VOLUME={294},
    NUMBER = {13},
      YEAR = {1982},
     PAGES = {447--450},
}

@unpublished{Saliola,
  AUTHOR = {Saliola, F.V.},
     TITLE = {Private communication}, 
}

@article {Sottile,
    AUTHOR = {Sottile, F.},
     TITLE = {Pieri’s formula for flag manifolds and {S}chubert polynomials},
   JOURNAL = {Ann. Inst. Fourier},
    VOLUME={46},
      YEAR = {1996},
     PAGES = {89--110},
}

@article {Ciocan,
    AUTHOR = {Ciocan-Fontanine, I.},
     TITLE = {Quantum cohomology of flag varieties},
   JOURNAL = {Intern. Math. Research Notices},
    NUMBER = {6},
      YEAR = {1995},
     PAGES = {263--277},
}

@article {GK1,
    AUTHOR = {Givental, A. and Kim, B.},
     TITLE = {Quantum cohomology of flag manifolds and {T}oda lattices},
   JOURNAL = {Comm. Math. Phys.},
    VOLUME = {168},
    NUMBER = {3},
      YEAR = {1995},
     PAGES = {609--641},
}

@article {K1,
    AUTHOR = {Kim, B.},
     TITLE = {Quantum cohomology of partial flag manifolds and a residue formula for their intersection pairing},
   JOURNAL = {Intern. Math. Research Notices},
    NUMBER = {1},
      YEAR = {1995},
     PAGES = {1--16},
}

@article {K2,
    AUTHOR = {Kim, B.},
     TITLE = {On equivariant quantum cohomology},
   JOURNAL = {Intern. Math. Research Notices},
    NUMBER = {17},
      YEAR = {1996},
     PAGES = {841--851},
}

@article {K3,
    AUTHOR = {Kim, B.},
     TITLE = {Quantum cohomology of flag manifolds G/B and quantum {T}oda lattices},
   JOURNAL = {Ann. of Math.},
VOLUME = {149},
    NUMBER = {1},
      YEAR = {1999},
     PAGES = {129--148},
}

@article {KM,
    AUTHOR = {Kontsevich, M. and Manin, Y.},
     TITLE = {Gromov-{W}itten classes, quantum cohomology, and enumerative geometry},
   JOURNAL = {Comm. Math. Phys.},
VOLUME = {164},
    NUMBER = {3},
      YEAR = {1994},
     PAGES = {525--562},
}

@article {RT,
    AUTHOR = {Ruan, Y. and Tian, G.},
     TITLE = {Mathematical theory of quantum cohomology},
   JOURNAL = {J. Diff. Geom.},
VOLUME = {42},
    NUMBER = {2},
      YEAR = {1995},
     PAGES = {259--367},
}

@article {Kirillov,
    AUTHOR = {Kirillov, A.N.},
     TITLE = {Quantum Schubert polynomials and quantum {S}chur functions},
   JOURNAL = {Internat. J. Algebra Comput.},
VOLUME = {9},
    NUMBER = {03n04},
      YEAR = {1999},
     PAGES = {385--404},
}

@unpublished{qBP,
  AUTHOR = {Le, T. and Ouyang, S. and Tao, L. and Restivo, J. and Zhang, A.},
     TITLE = {Quantum bumpess pipe dreams}, 
  note={In: \href{https://arxiv.org/abs/2403.16168}{arXiv:2403.16168} (2024)},
}

@article {Winkel,
    AUTHOR = {Winkel, R.},
     TITLE = {On the multiplication of {S}chubert polynomials},
   JOURNAL = {Adv. in Appl. Math.},
VOLUME = {20},
    NUMBER = {1},
      YEAR = {1998},
     PAGES = {73--97},
}

@article {Kirillov2,
    AUTHOR = {Kirillov, A.N. and Maeno, T.},
     TITLE = {Quantum double {S}chubert polynomials, quantum {S}chubert polynomials and {V}afa-{I}ntriligator formula},
   JOURNAL = {Discrete Math.},
VOLUME = {217},
    NUMBER = {1-3},
      YEAR = {2000},
     PAGES = {191--223},
}

@article {AssafSchu,
    AUTHOR = {Assaf, S. H.},
     TITLE = {A bijective proof of {K}ohnert's rule for {S}chubert
              polynomials},
   JOURNAL = {Comb. Theory},
  FJOURNAL = {Combinatorial Theory},
    VOLUME = {2},
      YEAR = {2022},
    NUMBER = {1},
     PAGES = {Paper No. 5, 9},
       DOI = {10.5070/c62156877},
       URL = {https://doi.org/10.5070/c62156877},
}

@article {Winkel1,
    AUTHOR = {Winkel, R.},
     TITLE = {Diagram rules for the generation of {S}chubert polynomials},
   JOURNAL = {J. Combin. Theory Ser. A},
  FJOURNAL = {Journal of Combinatorial Theory. Series A},
    VOLUME = {86},
      YEAR = {1999},
    NUMBER = {1},
     PAGES = {14--48},
       DOI = {10.1006/jcta.1998.2931},
       URL = {https://doi.org/10.1006/jcta.1998.2931},
}

@article {Winkel2,
    AUTHOR = {Winkel, R.},
     TITLE = {A derivation of {K}ohnert's algorithm from {M}onk's rule},
   JOURNAL = {S\'{e}m. Lothar. Combin.},
  FJOURNAL = {S\'{e}minaire Lotharingien de Combinatoire},
    VOLUME = {48},
      YEAR = {2002},
     PAGES = {Art. B48f, 14},
}

@article {BB93,
    AUTHOR = {Bergeron, N. and Billey, S.},
     TITLE = {R{C}-graphs and {S}chubert polynomials},
JOURNAL = {Exp. Math.},
    VOLUME = {2},
    NUMBER = {4},
    YEAR = {1993},
    PAGES = {257--269},
}

@article {BJS93,
    AUTHOR = {Billey, S. and Jockusch, W. and Stanley, R.},
     TITLE = {Some combinatorial properties of {S}chubert polynomials},
JOURNAL = {J. Algebraic Combin.},
  FJOURNAL = {Journal of Algebraic Combinatorics},
    VOLUME = {2},
    NUMBER = {4},
    YEAR = {1993},
    PAGES = {345--374},
}

@article {SF94,
    AUTHOR = {Fomin, S. and and Stanley, R.},
     TITLE = {Schubert polynomials and the nil{C}oxeter algebra},
JOURNAL = {Adv. Math.},
    VOLUME = {103},
    NUMBER = {2},
    YEAR = {1994},
    PAGES = {196--207},
}

\end{document}